\documentclass[final,3p, times, 10.8pt]{elsarticle}

\usepackage{amssymb}
\usepackage[english]{babel}
\usepackage{graphicx,epsfig,color}
\usepackage{float}
\usepackage{wrapfig}
\usepackage{subfig}
\usepackage{amsfonts,amsmath,latexsym}
\usepackage{amsmath}
\usepackage{amsthm}
\usepackage[makeroom]{cancel}
\usepackage{tabularx}
\usepackage{soul}

\addcontentsline{toc}{section}{References}
\definecolor{light}{gray}{.85}
\newtheorem{cor}{Corollary}[section]
\newtheorem{prn}{Proposition}[section]
\newtheorem{asm}{Assumption}
\newtheorem{rem}{Remark}
\numberwithin{equation}{section} \numberwithin{prn}{section}
\numberwithin{cor}{section} \numberwithin{thm}{section}
\numberwithin{lea}{section}

\makeatletter
\def\hlinewd#1{%
\noalign{\ifnum0=`}\fi\hrule \@height #1 %
\futurelet\reserved@a\@xhline}
\makeatother

\usepackage{lipsum}
\makeatletter
\def\ps@pprintTitle{%
 \let\@oddhead\@empty
 \let\@evenhead\@empty
 \def\@oddfoot{}%
 \let\@evenfoot\@oddfoot}
\makeatother

\begin{document}
\begin{frontmatter}



\title{\bf
Determination of a structural break
in a men-reverting process}


\author[*]{Fuqi Chen}
\author[*]{Rogemar Mamon\corref{cor1}}
\author[Windsor]{S\'ev\'erien Nkurunziza}
\address[*]{Department of Statistical and Actuarial Sciences,
University of Western Ontario, London, Ontario, Canada}
\address[Windsor]{Department of Mathematics and Statistics,
University of Windsor,
Windsor, Ontario, Canada}
\cortext[cor1]{Corresponding author. Department of Statistical
and Actuarial Sciences, University of Western Ontario,
1151 Richmond Street, London, Ontario, Canada, N6A 5B7.
Email: \textit{rmamon@stats.uwo.ca}}

\begin{abstract}
Determining accurately when regime and structural changes occur in various time-series data is critical in many social and natural sciences. We develop and show further the equivalence of two consistent estimation techniques in locating the change point under the framework of a generalised version of the Ornstein-Uhlehnbeck process. Our methods are based on the least sum of squared error and the maximum log-likelihood approaches. The case where both the existence and the location of the change point are unknown is investigated and an informational methodology is employed to address these issues. Numerical illustrations are presented to assess the methods’ performance.
\end{abstract}

\end{frontmatter}

\section{Introduction}
{\subsection{Background}}
\noindent The Ornstein-Uhlenbeck (OU) process is
commonly used to model
the stochastic dynamics of various financial variables. Certain economic indicators also
have stylised properties that are adequately captured by the OU process. Vasicek's pioneering work
(1977), employing an OU model to price a zero-coupon bond, inspired
a multitude of research investigations from extensions addressing
the model's weakness of
constant mean-reverting level to various applications in economic and financial practice. The importance of
this stochastic process is also demonstrated by its ubiquity in many fields.
Amongst the finance and finance-related areas highlighting the usefulness of OU process are electricity
market (e.g., Erlwein, et al. (2010)), commodity futures market
(e.g., Date, et al. (2010)), weather derivatives (e.g., Elias,
et al. (2014)), centra-bank setting rate policy (e.g., Elliott and Wilson (2007)),
stochastic control-driven insurance problems (e.g., Liang, et al. (2011)),
spot freight rates in the shipping industry (e.g., Benth, et al. (2015)), risk management (e.g.,
Date and Bustreo (2015), and power generation (e.g., Howell, et al. (2011)).
Applications of the OU process could be as well found in biology (see Rohlfs, et al. (2010)),
neurology (see Shinomoto, et al. (1999)), survival analysis (see Aalen and Gjessing (2004)),
physics (see Lansk\'{y} amd Sacerdote (2001)), and chemistry (see Lu (2003) and (2004)). \\
\ \\
\noindent To rectify the classical OU model's inability to capture the evolution of a process whose mean level
varies with time, Dehling, et al. (2010) introduced a generalised OU process in
which a time-dependent function describes its mean-reverting level. Such a generalised version
incorporates time-inhomogeneity and seasonality of mean reversion simultaneously.
Moreover, the generalised OU process is capable of modelling drastic changes in certain
time points (e.g., interest rates undergoing drastic moves
due to financial crisis, war, etc). Dehling, et al. (2014) developed the
framework in the study of a change-point phenomenon in the generalised OU process
to model the drastic change. \\
\ \\
\noindent Our contributions in this paper hinged on the research results
primarily from two research articles detailed as follows.
The first is the paper of Dehling, et al. (2010) in which
a maximum likelihood estimator (MLE) for the drift parameters of the diffusion
process is derived and the asymptotic properties, such as the
asymptotic distribution of the proposed MLE, are studied.
Dehling, et al. (2014) considered an extended model, where
there is one unknown change point and constructed a likelihood-ratio test statistic in
determining a candidate change point. This line of enquiry was continued
by Zhang~(2015) who examined the asymptotic properties of both the unrestricted and restricted MLE for
the drift parameters of the the generalised OU process with a single change point.
In particular, based on the established
asymptotic distribution of the MLEs, a James-Stein-type shrinkage estimator
for the drift parameters is proposed in Zhang~(2015) as an improvement.
In the estimation of the unknown change point, Zhang~(2015)
showed that the previously established asymptotic properties also hold for any consistent
estimator for the rate of change point. \\
\ \\
\noindent Nonetheless, both Dehling, et al.~(2014) and Zhang~~(2015)
did not provide any explicit method to estimate the change point. This deficiency
inspired the three main contributions of our paper.
Firstly, we present two consistent methods to
estimate the unknown change point. Secondly, we consider the case
where the existence of the change point is uncertain and propose an informational
approach to address this existence issue. Thirdly,
the performance of the proposed methods is theoretically analysed
and validated by a numerical implementation.
In practice, many data series are characterised by some potential changes
in structure, i.e., a sudden change in mean or variance and other model parameters.
It is then of interest to determine the (i) existence
and (ii) location of the change point. This implies segregating the data series
into different segments and analysing them in a more efficient way. \\
\ \\
\noindent Investigations concerning change point problems are not new.
Inaugural contributions to this field were spearheaded, for example, by
Page~(1954) and Shiryaev~(1963). Recent developments
have focused on (i) the estimation of change points and coefficients
of linear regression models with multiple change points (cf. Bai and Perron~(1998);
Perron and Qu~(2006); Lu and Lund~(2007), Gombay~(2010), Chen and Nkurunziza~(2015)); (ii) change-point
testing for the drift parameters of a periodic mean-reverting process (cf. Dehling,
et al.~(2014)); (iii) applications in finance
(cf. Spokoiny~(2009)); (iv) detection of malware within software
(Yan, et al.~(2008)); and (v) climatology (Reeves, et al. (2007), Robbins, et al.~(2011), Gallagher et al.~(2012)).
In general, the analysis of change points could be described as a
hypothesis-testing problem for the existence of change points in various
locations.  Alternatively, this could be viewed as a model selection problem
that treats the change points as
the additional unknown parameters to be estimated. However, unlike ordinary
least squares estimation, there is so far no closed-form estimation
methods to calculate the change point directly
or in a few steps. The existing change-point estimation approaches are
predominantly designed to perform a search at every possible location of
a change point with some efficient computational algorithms until some criteria
are satisfied. The well-known algorithms for
change point detection are the (i) binary segmentation algorithm
(Scott and Knott, (1974); Sen and Shrivastava, (1975)),
(ii) segment-neighbourhood algorithm (Auger and Lawrence (1989);
Bai and Perron (1998)) with adaption to the restricted regression
model (Perron and Qu (2006)), and (iii) PELT algorithm
(Killick, et al.~(2012)). \\
\ \\
\noindent There are two types of scenarios for which change-point problems
are examined in the literature.  In the first scenario,
the number of change points is known but their exact
locations are unknown (see Perron and Qu (2006) and Chen and Nkurunziza~(2015)).
The second scenario covers the more general situation in which both the number and the exact locations
of the change points are unknown. The estimation methods under the first scenario only require
the identification of the exact locations of the change points. Clearly, the performance assessment in the
former scenario is relatively easier than that in the latter scenario. \\

\subsection{Motivating examples}
\subsubsection{West Texas Intermediate (WTI) Cushing crude oil spot prices}\label{wti}
\noindent The first motivating example of this paper is the West Texas Intermediate (WTI) Cushing crude oil spot
prices, which is often being considered as a benchmark in oil pricing. The data set was compiled by
Bloomberg with code ``USCRWTIC" and covers approximately 4 years of daily prices from 09 November 2011 to 09 November 2015 (i.e., 1008 trading days). During this period (see Figure~\ref{fig:realdata1.1}), there is a price decline after September 2014 due to the conflict in the Middle East. It could be recalled that
 in September 2014, there was an increase in the  OPEC oil production led by a rebound in the Libyan output.
 The dollar, on the other hand, continued to get stronger. These events caused the decline of the crude oil prices and suggested the potential existence of a change point in the data series.  \\
\ \\
\noindent Due to the potential existence of the change point, for this data set the classical
OU process without change-point is inappropriate. Our findings show, as detailed in section 6,
that applying the classical method to the original data set produces very large Schwarz information criterion
(SIC) as compared to the proposed method which takes into account the change point.
Indeed, the proposed method increases the log-likelihood value from 0.72 to 11.89 and reduces the SIC by $69\%$.
\subsubsection{XAU currency}\label{xau}
\noindent The second motivating example of this paper is the XAU currency, which is the standard ticker symbol for one troy ounce of gold, considered as a currency to US dollar. This implementation is carried out to show the nuances
in dealing with data sets or its transformed version whose change point is not clear-cut at the outset.
The data set is also obtained from Bloomberg with code ``XAU", and it is a 15-year data series ranging from 03 November 2000 to 04 November 2015 (i.e., 3913 trading days). \\
\ \\
\noindent Descriptive analysis of the data set suggests a certain trend in the XAU currency series
that changes over time. In particular, the currency was increasing since the beginning of the period until August 2011. Most notably after 2008 crisis, the increasing slope became sharper as the investors flocked to gold market. The price was close to \$1900 in August 2011 and it remained above \$1500 until April 2013. Then the price plunged due to the banking crisis in Cyprus and increasing worries about an imminent change in the Federal Reserve's monetary policy. These features in the price movement suggest the potential existence of a change point.\\
\ \\
\noindent By using the proposed method to the original data set, the log likelihood is increased
from 1.43 to 12.62 and reduce the SIC by $43\%$. However, by using the proposed method to the
log-transformed data set, the proposed method does not allow to confirm the existence of a change point.
Nonetheless, the proposed method preserves a good performance in terms of log likelihood (7.06 versus 3.31 for the classical method).
\\
\ \\
\noindent The remainder of this paper is organised as follows.
In Section 2, we look at the formulation of the change-point problem.
We recapitulate in Section 3 the
results of both Dehling, et al. (2014) and Zhang~(2015)
on MLE and the related asymptotic properties
which are useful in delving into the asymptotic performance of our proposed methods.
Section 4 considers the case where the existence of the change point is certain
but its location is unknown; two estimation methods are put forward to determine
the unknown change
point. The asymptotics of the estimators are also discussed and hence,
the asymptotic properties established
in Zhang~(2015) also hold in our proposed techniques.
The case where the existence and the location of the change point are
both unknown is explored in Section 5. We
develop an informational approach to detect the change point, and the
consistency of our methods is likewise theoretically demonstrated. Section 6 provides the
numerical implementation of our proposed methods on both simulated and observed financial
market data. The final section gives some concluding remarks.

\section{Description of the single-change point problem}
\noindent Our main consideration in this paper is the change-point estimation strategy
under a generalised Ornstein-Uhlenbeck process with only one change point.
We start with Zhang's framework (2015), which
assumes that a consistent
estimator exists for the unknown change point $\tau \in [0, T]$.
The model under examination is the generalised version of the
Ornstein-Uhlenbeck (OU) process with SDE representation
\begin{eqnarray}\label{ou1}
\textit{d}X_t&=&S(\theta,t,X_t)\textit{d}t+\sigma \textit{d}W_t,\quad 0<t\leq T,
\end{eqnarray}
where $S(\theta,t,X_t)=L(t)-a X_t=\sum_{i=1}^p\mu_i\varphi_i(t)-a X_t$, $i=1,...,p$,
$\theta=\theta^{(1)}=(\mu_1,...,\mu_p,-a)'$,
$'$ denotes the transpose of a matrix. Also, $W_t$ is a one-dimensional Brownian motion
defined on some probability space $(\Omega, {\cal F}, P).$\\
\ \\
\noindent In particular, we assume that there is one unknown change point $\tau=sT$, $0<s<1$ such that
\begin{eqnarray*}
S(\theta,t,X_t)&=&\sum_{i=1}^p\mu_i^{(1)}\varphi_i(t)-a^{(1)} X_t,\quad 0<t<\tau,\\
S(\theta,t,X_t)&=&\sum_{i=1}^p\mu_i^{(2)}\varphi_i(t)-a^{(2)} X_t,\quad \tau\leq t\leq T,
\end{eqnarray*}
where $\theta=\theta^{(1)}=(\mu_1^{(1)},...,\mu_p^{(1)}, -a^{(1)})'$, for $0<t<\tau$,
and $\theta=\theta^{(2)}=(\mu_1^{(2)},...,\mu_p^{(2)},a^{(2)})'$, for $\tau\leq t\leq T$.\\
\ \\
\noindent For the case when there is no change point, maximum likelihood estimators
for the drift parameters and their related asymptotic properties were derived
in Dehling, et al. (2010). These results are reviewed in the next section
and serve as a springboard for our theoretical discussion.

\section{Earlier MLE-based results and our new results}
\noindent This section consists of two subsections: (i) review of the results for the
MLE of the drift parameters (without change point) along
with the related asymptotic properties demonstrated in Dehling, et al. (2010);
and (ii) review of the MLE for the drift parameters (with one change point) and the related asymptotic
properties studied in Zhang (2015).  \\
\ \\
\noindent In Zhang (2015), however, asymptotic normality for the MLE estimator of the drift parameters
is derived under the assumption that the estimator is already consistent. In our case,
we shall show (instead of assume) that such an estimator of the change point is consistent
thereby proving the consistency properties of the proposed estimator.\\
\ \\
\noindent {\bf Notation:} The expressions ``$\xrightarrow[T\rightarrow \infty]{p}$", $``\xrightarrow[T\rightarrow \infty]{D}"$, and $``\xrightarrow[T\rightarrow \infty]{a.s.}"$ denotes convergence in probability, convergence in distribution, and convergence almost surely, respectively. The ``$O(\cdot)$'' stands for ``Big O'' describing the asymptotic behaviour of functions; i.e., for a sequence of random variables $U_n$ and a corresponding set of constants
$a_n$, $U_n=O_p(a_n)$ means $U_n/a_n$ is stochastically bounded in the sense that $\forall \epsilon > 0
,~~\exists~M >0
,~~\ni~P(\left|U_n/a_n \right| > M) < \epsilon,~~\forall n.$
The symbol ``small o'' means $U_n=o_p(a_n)$, i.e., $U_n/a_n$ converges in probability to zero as $n$ approaches an appropriate limit. Considering $U_n=o_p(a_n)$ is equivalent to $U_n/a_n=0_p(1)$,
convergence in probability is defined here as $\displaystyle \lim_{n \rightarrow \infty}\left(
P(|U_n/a_n)| \geq \epsilon \right)=0.$

\subsection{Maximum likelihood estimator of the drift parameters}
\noindent To gain some useful insights, we first consider the case where there is no change point
($\theta^{(1)}=\theta^{(2)}$). We review briefly the MLE of the drift
parameters proposed in Dehling, et al. (2010) under the following
assumptions.\\
\begin{asm}$\textrm{P}\left(\int_{0}^{T}S^2(\theta,t,X_t)<\infty\right)=1$,
for all $0<T<\infty$, for all $\theta\in\theta$.\end{asm}
\noindent With this assumption, Theorem 7.6 in Lipster and
Shiryayev (2001) may be used to find an explicit expression for the corresponding
likelihood function. \\
\ \\
\noindent Suppose that there is no any change point in
$[0,T]$. Then, let $\mathcal{C}[0,T]$ be the space of continuous, real-valued
function on $[0,T]$ and let $\mathcal{B}[0,T]$ be the Borel $\sigma$-field
associated with $\mathcal{C}[0,T]$. Let $P_B$ be the probability measure
generated by the Brownian motion on $(\mathcal{C}[0,T],\mathcal{B}[0,T])$,
i.e., $P_B(A)=P\{\omega: B\in A\}$, $A\in \mathcal{B}[0,T]$. Suppose further that
$P_X$ is the probability measure generated by the observation $X_T$
of the process with SDE specified in (\ref{ou1}). Then, the likelihood function
of $X_T$ is
\begin{equation}\label{lf1}
\mathcal{L}(\theta,X_T)=\frac{dP_X}{dP_B}(X_T)=\exp \left(\frac{1}{\sigma^2}
\int_{0}^TS(\theta,t,X_t)dX_t-\frac{1}{2\sigma^2}\int_{0}^TS^2(\theta,t,X_t)dX_t\right).
\end{equation}
Therefore, the MLE of the drift parameters is given by
\begin{equation}\label{mle}
\hat{\theta}=Q_{(0,T)}^{-1}\tilde{R}_{(0,T)}=\left( \frac{1}{T}Q_{(0,T)} \right)^{-1}
\frac{1}{T}\tilde{R}_{(0,T)},
\end{equation}
where
$$Q_{(0,T)}=
\left[\begin{array}{cccc}
    \int_0^T\varphi_1^2(t)dt &\dots& \int_0^T\varphi_1(t)\varphi_p(t)dt
    & -\int_0^T\varphi_1(t)X_tdt \\
    \dots&&& \\
    -\int_0^TX_t\varphi_1(t)dt &\dots & -\int_0^TX_t\varphi_p(t)dt & \int_0^TX_t^2dt
\end{array}\right],$$
and $\tilde{R}_{(0,T)}=(\int_0^T\varphi_1(t)dX_t,...,\int_0^T\varphi_p(t)dX_t,-\int_0^TX_tdX_t)'.$
\\
\ \\
\noindent Note that the MLE introduced above could be evaluated by applying the
Euler's discretisation to (\ref{ou1}), and then getting a linear model and applying the ordinary
least-squares estimation method to provide an estimator containing the Riemann and
Ito sums. Then, the OLS estimator will converge into the MLE estimator
as $\Delta t\rightarrow 0$. \\
\ \\
\noindent For the existence of $Q_{(0,T)}^{-1}$, it is shown in Remark 3
of Dehling, et al. (2010) that $TQ_{(0,T)}^{-1}$ exists almost
surely if $T$ is large enough. Moreover, Proposition 2.1.1 of
Zhang~(2015) states the positive definiteness of $\frac{1}{T}Q_{(0,T)}$
under the following assumption.

\begin{asm}For any $T>0$, the base function $\{\varphi_i(t),i=1,..,p\}$ is
Riemann-integrable on $[0,T]$ and satisfies
\begin{enumerate}
  \item Periodicity. That is, $\varphi_i(t+v)=\varphi_i(t)$, for all $i=1,...,p$ and $v$ is the
  period observed in the data.
  \item Orthogonality. That is, for all $j,k=1,...,p$, $\int_{0}^v\varphi_j(t)\varphi_k(t)dt$ is
  equal to $v$ if $j=k$ and 0 otherwise.
\end{enumerate}
\end{asm}

\begin{prn}[Proposition 2.1.1 of Zhang,~(2015)] \label{prn01}
Under Assumption 2, for any $T>0$ and $t\in [0,T]$, the base
functions $\{ \varphi_i(t), i=1,...,p \}$ are incomplete if and
only if $\frac{1}{T}Q_{(0,T)}$ is a positive definite matrix.
\end{prn}

\noindent Hence, for the rest of this paper, we assume that the sample
size $T$ is an integral multiple of the period length $v$,
i.e., $T=Nv$ for some integer $N$. Without loss of generality,
we let $v=1$ and this implies that $\varphi_j(t+1)=\varphi_j(t)$.\\
\ \\
\noindent Inspired by the results of Dehling, et al.~(2010) and
Dehling, et al.~(2014), Zhang~(2015) first studied
the case where the change point in (\ref{ou1}) exists and known
to be $\tau^0=s^0T$, $0< s^0< 1$ and derived the results
in estimating $\theta^{(1)}$ and $\theta^{(2)}$. In particular,
\begin{equation}\label{mle2}
\hat{\theta}^{(1)}=Q_{(0,s^0T)}^{-1}\tilde{R}_{(0,s^0T)}=
\theta^{(1)}+\left(\frac{1}{T}Q_{(0,s^0T)}\right)^{-1}\frac{1}{T}{R}_{(0,s^0T)}
\end{equation}
and
\begin{equation}\label{mle3}
\hat{\theta}^{(2)}=Q_{(s^0T,T)}^{-1}\tilde{R}_{(s^0T,T)}=
\theta^{(2)}+\left(\frac{1}{T}Q_{(s^0T,T)}\right)^{-1}\frac{1}{T}{R}_{(s^0T,T)},
\end{equation}
where
 ${R}_{(a,b)}=\left(\int_a^b\varphi_1(t)dW_t,...,\int_a^b\varphi_p(t)dW_t,-\int_a^b X_tdW_t\right)'$ for $0\leq a<b\leq T$. \\
 \ \\
\noindent Also, the asymptotic properties of the above proposed MLEs
are well-studied in Dehling, et al. (2010) for the case
when there is no change point and Zhang~(2015) for the case of a
single change point. To summarise these results, we first go back to the
case when there is no change point. By (\ref{ou1}), we have
$$\int_0^T\varphi_i(t)dX_t=\sum_{j=1}^p\mu_j\int_0^T\varphi_i(t)
\varphi_j(t)dt-a\int_0^T\varphi_i(t)X_tdt+\sigma\int_0^T\varphi_i(t) dW_t,$$
for $i=1,...,p$, and
$$\int_0^TX_tdX_t=\sum_{j=1}^p\mu_j\int_0^TX_t\varphi_j(t)dt-a\int_0^TX_t^2dt
+\sigma\int_0^TX_t dW_t.$$
It follows that
$$\hat{\theta}=Q_{(0,T)}^{-1}\tilde{R}_{(0,T)}=\theta+\sigma Q_{(0,T)}^{-1}{R}_{(0,T)}
=\theta+\sigma TQ_{(0,T)}^{-1}\frac{1}{T}R_{(0,T)}.$$
\noindent By Ito's lemma, the SDE in (\ref{ou1}) has the
solution
\begin{equation}\label{sol1}X_t=e^{-at}X_0+h(t)+N_t,\end{equation}
where
$\displaystyle{h(t)=e^{-at}\sum_{i=1}^p\mu_i\int_{0}^te^{as}\varphi_i(s)ds}$
and
$\displaystyle{N_t=\sigma e^{-at}\int_{0}^te^{as}dW_s}.$\\
\ \\
\noindent The uniform boundedness of solution (\ref{sol1}) was studied in Zhang~(2015). Using similar methods employed for the proof of Proposition 2.2.1 in Zhang~(2015) together with the mean-reversion property in the drift term of the OU process, one may verify that the SDE (\ref{ou1}) admits a strong and unique solution that is uniformly bounded in $L^2$, and
\begin{equation}\label{sol2}\sup_{t\geq 0}\mathrm{E}(X_t^2)\leq K_1,
\end{equation}
for $0<K_1<\infty$.\\
\ \\
\noindent Note that the process $\{X_t,t\geq 0\}$ is not stationary in the ordinary sense.
Thus, it is impossible to apply the ergodic theorem directly. To go around this problem,
Dehling, et al. (2010) introduced a stationary solution for $t\in \mathbb{R}$
instead of $t\geq 0$. That is, \begin{equation}\label{sol3}\tilde{X}_t=\tilde{h}(t)
+\tilde{N}_t,\end{equation} where
$\displaystyle{\tilde{h}(t)=e^{-at}\sum_{i=1}^p\mu_i\int_{-\infty}^te^{as}\varphi_i(s)ds}$
and
$\displaystyle{\tilde{N}_t=\sigma e^{-at}\int_{-\infty}^te^{as}d\tilde{B}_s},$
with $(\tilde{B}_s)_{s\in  \mathbb{R}}$ denotes a bilateral Brownian motion, i.e.,
$$\tilde{B}_s={B}_s \mathbf{1}_{ \mathbb{R}_{+}}(s)+\bar{B}_{-s}
\mathbf{1}_{ \mathbb{R}_{-}}(s).$$
\ \\
\noindent Here, $({B}_s)_{s\geq 0}$ and $(\bar{B}_s)_{s\geq 0}$ are two independent
standard Brownian motions and $\mathbf{1}_A$ stands for the indicator function
over the set $A$. It follows from (\ref{sol2}) and Lemma~{4.3} in Dehling,
et al. (2010) that the sequence of $\mathcal{C}[0,1]$-valued random
variables $W_k(s)=\tilde{X}_{k-1+s}$, $0\leq s\leq 1$, $k\in \mathbb{N}$ is
stationary and ergodic. Then, by Proposition 4.5 of Dehling, et al.
(2010),
\begin{equation}\label{cx2}\frac{1}{T}\int_0^T\tilde{X}_t\varphi_j(t)dt
\xrightarrow[T\rightarrow \infty]{a.s.}\int_0^1\tilde{h}(t)\varphi_j(t)dt\end{equation}
and
\begin{equation}\label{cx3}\frac{1}{T}\int_0^T\tilde{X}_t^2dt\xrightarrow[T\rightarrow \infty]
{a.s.}\int_0^1\tilde{h}^2(t)dt+\frac{\sigma^2}{2a}.\end{equation}
\ \\
\noindent Moreover, it follows from Lemma 4.4 in Dehling, et al. (2010) that under Assumption 2,
\begin{equation}\label{cx1}|\tilde{X}_t-X_t|\xrightarrow[t \rightarrow \infty]{a.s.} 0.\end{equation}

\noindent Using (\ref{cx1}), the following properties hold:
$$\frac{1}{T}\int_0^T\tilde{X}_t\varphi_j(t)dt-\frac{1}{T}\int_0^T{X}_t\varphi_j(t)dt
\xrightarrow[T\rightarrow \infty]{a.s.} 0$$
$$\mbox{and}~~\frac{1}{T}\int_0^T\tilde{X}_t^2dt-\frac{1}{T}\int_0^T{X}_t^2dt
\xrightarrow[T\rightarrow \infty]{a.s.} 0.$$
Then it follows from  (\ref{cx2}) and (\ref{cx3}) that
$$\frac{1}{T}\int_0^T{X}_t\varphi_j(t)dt\xrightarrow[T\rightarrow
\infty]{a.s.}\int_0^1\tilde{h}(t)\varphi_j(t)dt$$
$$\mbox{and}~~\frac{1}{T}\int_0^T{X}_t^2dt\xrightarrow[T\rightarrow \infty]{a.s.}
\int_0^1\tilde{h}^2(t)dt+\frac{\sigma^2}{2a}.$$
Hence,
\begin{equation}\label{conQ-0}TQ_{(0,T)}^{-1}\xrightarrow[T\rightarrow
\infty]{a.s.}\Sigma_0^{-1},\end{equation}
where
$$\Sigma_0=\begin{bmatrix}
    I_p  & \Lambda \\
    \Lambda'   & w \\
\end{bmatrix},$$
with ${\Lambda}_{(0,T)}=(\int_0^1\tilde{h}(t)\varphi_1(t)dt,...,\int_0^1
\tilde{h}(t)\varphi_p(t)dt)'$ and $w=\int_0^1\tilde{h}^2(t)dt+\frac{\sigma^2}{2a}$.\\
\ \\
\noindent Furthermore, under Assumptions 1--2, the following properties hold
for $R_{(0,T)}$.
\begin{enumerate}
\item $\{R_{(0,T)}, T>0\}$ is a  martingale.
\item $\frac{1}{T}R_{(0,T)}\xrightarrow[T\rightarrow \infty]{a.s.}0$.
\item $\frac{1}{\sqrt{T}}R_{(0,T)}\xrightarrow[T\rightarrow
\infty]{D}R\sim \mathcal{N}_{p+1}(0,\Sigma_0)$.
\end{enumerate}
For detailed proofs of the above properties 1-3, see Theorem A.4.3., Propositions 2.1.4 and 2.1.6 in
Zhang~(2015), respectively. \\
\ \\
\noindent Based on the above properties, it follows from the Slutsky's Theorem that
$$\frac{1}{\sqrt{T}}(\hat{\theta}-\theta)\xrightarrow[T\rightarrow \infty]
{D}\rho \sim \mathcal{N}_{p+1}(0,\Sigma_0^{-1}).$$
In Chapter 2 and pertinent proofs in Appendix B of Zhang~(2015), the above asymptotic properties are extended to the case
of a single change point in the following way.\\
\ \\
\noindent Write
$$\displaystyle{\tilde{h}^{(1)}(t):=e^{-a^{(1)}t}\sum_{i=1}^p\mu_i^{(1)}
\int_{-\infty}^te^{a^{(1)}s}\varphi_i(s)ds}$$
and
$$\displaystyle{\tilde{h}^{(2)}(t):=e^{-a^{(2)}t}\sum_{i=1}^p\mu_i^{(2)}
\int_{-\infty}^te^{a^{(2)}s}\varphi_i(s)ds}.$$

\noindent Then
\begin{equation}\label{cxs1}\frac{1}{T}\int_0^{s^0T}\tilde{X}_t\varphi_j(t)dt
\xrightarrow[T\rightarrow \infty]{a.s.}s^0\int_0^1(\tilde{h}^{(1)})(t)\varphi_j(t)dt\end{equation}
and
\begin{equation}\label{cxs2}\frac{1}{T}\int_0^{s^0T}\tilde{X}_t^2dt
\xrightarrow[T\rightarrow \infty]{a.s.}s^0\left(\int_0^1(\tilde{h}^{(1)})^2(t)dt
+\frac{\sigma^2}{2a^{(1)}}\right),\end{equation}
where $\tilde{X}_t$ is the process defined in (\ref{sol3}). Similarly,
\begin{equation}\label{cxs3}\frac{1}{T}\int_{s^0T}^{T}\tilde{X}_t\varphi_j(t)dt
\xrightarrow[T\rightarrow \infty]{a.s.}(1-s^0)\int_0^1(\tilde{h}^{(2)})(t)
\varphi_j(t)dt\end{equation}
and
\begin{equation}\label{cxs4}\frac{1}{T}\int_{s^0T}^{T}\tilde{X}_t^2dt
\xrightarrow[T\rightarrow \infty]{a.s.}(1-s^0)\left(\int_0^1(\tilde{h}^{(2)})^2(t)dt
+\frac{\sigma^2}{2a^{(2)}}\right).\end{equation}
Using (\ref{cx1}), the following properties hold:
\begin{eqnarray*}
&&\frac{1}{T}\int_0^{s^0T}\tilde{X}_t\varphi_j(t)dt-\frac{1}{T}\int_0^{s^0T}{X}_t
\varphi_j(t)dt\xrightarrow[T\rightarrow \infty]{a.s.} 0,\\
&&\frac{1}{T}\int_0^{s^0T}\tilde{X}_t^2dt-\frac{1}{T}\int_0^{s^0T}{X}_t^2dt
\xrightarrow[T\rightarrow \infty]{a.s.} 0,\\
&&\frac{1}{T}\int_{s^0T}^{T}\tilde{X}_t\varphi_j(t)dt-\frac{1}{T}\int_{s^0T}^{T}{X}_t
\varphi_j(t)dt\xrightarrow[T\rightarrow \infty]{a.s.} 0,\\
&&\mbox{and}~~\frac{1}{T}\int_{s^0T}^{T}\tilde{X}_t^2dt-\frac{1}{T}\int_{s^0T}^{T}{X}_t^2dt
\xrightarrow[T\rightarrow \infty]{a.s.} 0.\end{eqnarray*}
Hence, it follows that
\begin{eqnarray}
\frac{1}{T}Q_{(0,s^0T)}&\xrightarrow[T\rightarrow \infty]{a.s.}&{s^0}
\Sigma_1,\label{conQ-1}\\
TQ_{(0,s^0T)}^{-1}&\xrightarrow[T\rightarrow \infty]{a.s.}&\frac{1}{s^0}
\Sigma_1^{-1},\label{conQ-2}\\
\frac{1}{T}Q_{(s^0T,T)}&\xrightarrow[T\rightarrow \infty]{a.s.}&(1-s^0)
\Sigma_2,\label{conQ-3}\end{eqnarray}
and
\begin{equation}TQ_{(s^0T,T)}^{-1}\xrightarrow[T\rightarrow \infty]{a.s.}
\frac{1}{(1-s^0)}\Sigma_2^{-1},\label{conQ-4}\end{equation}
where
$$\Sigma_1=\begin{bmatrix}
    I_p  & \Lambda_1 \\
    \Lambda_1'   & w_1 \\
\end{bmatrix}\quad \mbox{and} \quad
\Sigma_2=\begin{bmatrix}
    I_p  & \Lambda_2 \\
    \Lambda_2'   & w_2 \\
\end{bmatrix}$$
with ${\Lambda}_{i}=\left (\int_0^1\tilde{h}^{(i)}(t)\varphi_1(t)dt,...,
\int_0^1\tilde{h}^{(i)}(t)\varphi_p(t)dt \right)'$ and $w_i=\int_0^1(\tilde{h}^{(i)})^2(t)dt
+\frac{\sigma^2}{2a^{(i)}}$, $i=1,2$.
Furthermore, it follows from Proposition 2.2.6 in Zhang~(2015) that both $\Sigma_1$
and $\Sigma_2$ are positive definite provided that Assumptions 1--2 hold.

\subsection{New results for the analysis of asymptotic properties}
\noindent Based on the established results in Subsection 3.1, we provide two propositions which
are useful in illustrating the asymptotic properties of the estimator for the change point.

\begin{prn}\label{conQ1}For any $\eta\in (0,s^0]$, we have
\begin{equation}\frac{1}{T}Q_{(0,\eta T)}\xrightarrow[T\rightarrow \infty]{a.s.}
\eta\Sigma_1\label{conQ-5},\end{equation}
\begin{equation}\frac{1}{T}Q_{(\eta T,s^0T)}\xrightarrow[T\rightarrow \infty]{a.s.}
(s^0-\eta)\Sigma_1\label{conQ-6},\end{equation}
and
\begin{equation}\frac{1}{T}Q_{(\eta T, T)}\xrightarrow[T\rightarrow \infty]{a.s.}
(s^0-\eta)\Sigma_1 + (1-s^0)\Sigma_2\label{conQ-7}.\end{equation}
\end{prn}
\begin{proof} Result (\ref{conQ-5}) follows directly from (\ref{conQ-2}) with the upper
bound of the integrals reduces from $s^0T$ to $\eta T$. Similarly, (\ref{conQ-6})
follows directly from (\ref{conQ-2}) with the lower bound of the integrals
increases from $0$ to $\eta T$. Finally, (\ref{conQ-7}) could be obtained by combining
(\ref{conQ-5}) and (\ref{conQ-6}).
\end{proof}

\begin{prn}\label{conQ2}For any $\eta\in (s^0,1]$, we have
\begin{equation}\frac{1}{T}Q_{(\eta T,T)}\xrightarrow[T\rightarrow \infty]{a.s.}
(1-\eta)\Sigma_2\label{conQ-8}.\end{equation}
\begin{equation}\frac{1}{T}Q_{(s^0T,\eta T)}\xrightarrow[T\rightarrow \infty]
{a.s.}(\eta-s^0)\Sigma_2\label{conQ-9}.\end{equation}
\begin{equation}\frac{1}{T}Q_{(0,\eta T)}\xrightarrow[T\rightarrow \infty]{a.s.}s^0\Sigma_1
+ (\eta-s^0)\Sigma_2\label{conQ-10}.\end{equation}
\end{prn}
\begin{proof} Result (\ref{conQ-8}) follows directly from (\ref{conQ-3}) with the
lower bond of the integrals increases from $s^0T$ to $\eta T$. Similarly,
(\ref{conQ-9}) follows directly from (\ref{conQ-3}) with the upper bond of
the integrals reduces from $T$ to $\eta T$ and  (\ref{conQ-10}) is established
by combining (\ref{conQ-8}) and (\ref{conQ-9}).  \end{proof}

\noindent Moreover, for the case where the change point is unknown,
Zhang~(2015) assumed that there exists a consistent estimator of
the unknown change point and derives the same asymptotic properties for
the drift parameters based on the consistent estimator assumption of a change point.

\section{Two methods in the estimation of the change point}\label{scpe}
\noindent We develop the least sum of squared error (LSSE)
and maximum log-likelihood (MLL) methods to yield an estimator for
the unknown change point $\tau$,
and investigate its consistency under these two methods.

\subsection{Least sum of squared error method}
\noindent This subsection first introduces the LSSE method then studies the consistency
of the proposed estimator. To calculate the residuals, we apply the Euler-Maruyama
discretisation method to (\ref{ou1}). Consider a partition
$0=t_{0}<...<t_{n}=T$ on a given time period $[0,T]$  with a constant increment $\Delta_t=t_{i+1}-t_i.$
Hence,
$Y_i=X_{t_{i+1}}-X_{t_i}$ and $Z_{i}=(\varphi_1(t_i),...,\varphi_p(t_i),-X_{t_i})(\Delta_t)$,
and the discretized process is given by
\begin{eqnarray}\label{ou2}
Y_{i}&=&Z_{i}\theta+\epsilon_i,\quad t_i\in[0,T],
\end{eqnarray}
where $\epsilon$ is the error term given by $\sigma\sqrt{\Delta_t}N$, and $N$ is
the standard normal term. In this case, we could use the least-squared (LS) method to
estimate the change point. The details of the LS method will be discussed in the next section.
Based on (\ref{ou2}), let $\tau^0$ be the exact value of the unknown change-point
$\tau$. Then, $\tau^0$ can be estimated using the least sum of squared errors (SSE) method
described as
\begin{eqnarray}\label{cp1}
\hat{\tau}=\arg\min_{\tau} SSE(\tau),
\end{eqnarray}
where
\begin{eqnarray}\label{ssr1}
SSE(\tau)&=&\sum_{t_i\in[0,T]}(Y_i-Z_i\hat{\theta}(\tau))'(Y_i-Z_i\hat{\theta}(\tau))
\end{eqnarray}
and $\hat{\theta}(\tau)$ is the estimator of $\theta$ with the change
point given by $\tau$. More precisely, from Zhang~(2015), $\hat{\theta}=
(\hat{\theta}^{(1)}, \hat{\theta}^{(2)})$ where
\begin{eqnarray*}
\hat{\theta}^{(1)}&=&Q_{(0,\hat{\tau})}^{-1}\tilde{R}_{(0,\hat{\tau})}\quad \textrm{and}\quad
\hat{\theta}^{(2)}=Q_{(\hat{\tau},T)}^{-1}\tilde{R}_{(\hat{\tau},T)}.
\end{eqnarray*}
\ \\
\noindent {\bf Consistency of the proposed estimator} \\
Under Assumptions 1--2, $\sum_{t_i\in [0,\tau^0]}Z_i'Z_i$ and
$\sum_{t_i\in (\tau^0,T]}Z_i'Z_i$, the respective discretised versions of
$Q_{(0,{\tau}^0)}$ and $Q_{({\tau}^0,T)}$ are both positive definite with probability 1 provided that the base
functions $\{ \varphi_i(t), i=1,...,p \}$ are incomplete. Moreover, it follows from Proposition 2.2.6 in Zhang~(2015) that both
$\frac{1}{s^0T}Q_{(0,{\tau}^0)}$ and $\frac{1}{(1-s^0)T}Q_{({\tau}^0,T)}$ converge in probability to some positive definite matrices for large $T$, and so are their respective discretised versions. Hence, for large $T$, it is reasonable to impose a useful assumption in proving the consistency of the estimator of the change point.

\begin{asm}\label{asm3}Suppose that there exists an
$L_0>0$ such that for all $L>L_0$ the minimum eigenvalues of
$\frac{1}{L}\sum_{t_i\in(\tau^0,\tau^0+L]}Z_i'Z_i$ and of
$\frac{1}{L}\sum_{t_i\in(\tau^0-L,\tau^0]}Z_i'Z_i$, as well as their respective continuous-time
versions $\frac{1}{L}Q_{(\tau^0,\tau^0+L)}$ and $\frac{1}{L}Q_{(\tau^0-L,\tau^0]}$,
are all bounded away from 0.\end{asm}
\ \\
\noindent For more details about the above assumption, reader is referred to Perron and Qu (2006) (see also Chen and Nkurunziza~(2015)). Below are two propositions pertinent to the
consistency of the rate of change point specified by $\hat{s}=\hat{\tau}/T$,
where $\hat{\tau}$ is given by (\ref{cp1}).
\begin{prn}\label{prn13}
Suppose that $\theta^{(1)}-\theta^{(2)}$, the shift in the drift parameters, is of fixed non-zero magnitude independent of $T$. Then, under Assumptions 1--\ref{asm3}, $\hat{s}-s^0\xrightarrow[T\rightarrow \infty]{P}0$.
\end{prn}
\begin{prn}\label{prn14}
Suppose the conditions in Proposition~\ref{prn13} hold. Then, for every $\epsilon>0$, there exists a $C>0$ such that for
large $T$, $P(T|\hat{s}-s|>C)<\epsilon$.
\end{prn}
\ \\
\noindent The proofs of Propositions \ref{prn13} and \ref{prn14}
are provided in \ref{AppendixA}. Proposition~\ref{prn13} shows that the
estimated rate $\hat{s}$ is consistent for $s^0$ and Proposition~\ref{prn14}
shows that the rate of convergence is $T$. From these two propositions, we conclude
that the proposed estimator satisfies the consistency assumption required in Zhang~(2015). Hence, it follows from Corollary 2.3.1 in Zhang~(2015) that
\begin{equation}\label{normal1}\sqrt{T}(\hat{\theta}(\hat{\tau})-\theta^0)\xrightarrow[T\rightarrow \infty]{D} \mathcal{N}_{2p+2}(0, \sigma^2\tilde{\Sigma}^{-1}),\end{equation}
where $\tilde{ \Sigma}^{-1}=\mathrm{diag}\left( \frac{1}{s^0}{\Sigma}_1^{-1},\frac{1}{1-s^0}{ \Sigma}_2^{-1}\right)$ and $\hat{\tau}$ are obtained from (\ref{cp1}).\\
\begin{rem}\label{remark1} Note that in this paper, we focus on the case where the shift in drift parameters $\theta^{(1)}$ and $\theta^{(2)}$ indicated in (\ref{ou1}) is independent of time $T$. However, in reality we may encounter the case where the shift is time-dependent, and in particular, as $T$ tends to infinity, the shift may shrink towards 0 at rate $v_T$, i.e., $\theta^{(1)}-\theta^{(2)}=\mathbf{M}v_T$, where $\mathbf{M}$ is independent of $T$ and $v_T\xrightarrow[T\rightarrow \infty]{}0$. In this case, the validity of Proposition~\ref{prn13} and Proposition~\ref{prn14} depends on the speed $v_T$. In fact, using similar arguments as in the proofs of these two propositions (see \ref{AppendixA}), one may show that if $v_T\xrightarrow[T\rightarrow \infty]{}0$ and $T^{1/2-r^*}v_T\xrightarrow[T\rightarrow \infty]{}\infty$ for some $0<r^*<1/2$, then under Assumptions 1--\ref{asm3}, we have (i) $\hat{s}-s^0\xrightarrow[T\rightarrow \infty]{P}0$ and (ii), for every $\epsilon>0$, there exists a $C>0$ such that for
large $T$, $P(Tv_T^2|\hat{s}-s|>C)<\epsilon$. (See Remark~\ref{remark2} in \ref{AppendixA}).
\end{rem}

\subsection{Maximum log-likelihood method}
\noindent We introduce the maximum log-likelihood (MLL) method pertinent to the
study of the consistency of the proposed estimator.
By Theorem 7.6 in Lipster and Shiryayev~(2001), the log-likelihood
function of (\ref{ou1}) (see also Dehling, et al.,~(2010) and Zhang,~(2015))
is
$$\log\mathcal{L}^*((0, T),\theta)=\frac{1}{\sigma^2}\int_{0}^TS(\theta,t,X_t)dX_t-\frac{1}
{2\sigma^2}\int_{0}^TS^2(\theta,t,X_t)dt.$$

\noindent When a change point $\tau$ exists, the log-likelihood function is
\begin{equation}\label{llf1}\log\mathcal{L}(\tau,\theta)=\log\mathcal{L}^*((0,\tau),\theta^{(1)})
+\log\mathcal{L}^*((\tau, T),\theta^{(2)}). \end{equation}
Based on this log likelihood function, an alternative method to estimate the unknown
change point is using the maximum of the log likelihood function. That is,
\begin{equation}\label{mle1}\hat{\tau}=\arg\max_{{\tau}} \log \mathcal{L}({\tau},\hat{\theta}(\tau))
\end{equation}
or equivalently
\begin{equation}\label{mle2}\hat{\tau}=\arg\min_{{\tau}} -2 \log \mathcal{L}({\tau},\hat{\theta}(\tau)),
\end{equation}
where $\hat{\theta}$ is the MLE of $\theta$ based on $\hat{\tau}$. \\
\ \\
\noindent {\bf Consistency of the proposed estimator} \\
\noindent To investigate the consistency behavior of the proposed estimator $\hat{\tau}$ obtained from (\ref{mle1}), we update Assumption \ref{asm3} to the following version:
\begin{asm}\label{asm4}Suppose that there exists an
$L_0^*>0$ such that under Assumption 1--3, for all $L>L_0^*$ the minimum eigenvalues of the following two symmetric matrices
$\frac{1}{2L}[Q_{(0,\tau^0)}Q_{(0,\tau^0+L)}^{-1}Q_{(\tau^0,\tau^0+L)}+Q_{(\tau^0,\tau^0+L)}$ $\times Q_{(0,\tau^0+L)}^{-1}Q_{(0,\tau^0)}]$ and $\frac{1}{2L}[Q_{(\tau^0,T)}Q_{(\tau^0-L,T)}^{-1}Q_{(\tau^0-L,\tau^0)}+Q_{(\tau^0-L,\tau^0)}Q_{(\tau^0-L,T)}^{-1} Q_{(\tau^0,T)}]$
are all bounded away from 0.\end{asm}
\noindent Then, similar to the LSSE method, for the estimator $\hat{\tau}$ given by (\ref{mle1}), we also provide two propositions below regarding the consistency of $\hat{s}$.
\begin{prn}\label{prn-mle1}
Suppose that $\theta^{(1)}-\theta^{(2)}$, the shift in the drift parameters, is of fixed non-zero magnitude independent of $T$. Then, under Assumptions 1--2 and Assumption \ref{asm4}, $\hat{s}-{s^0}\xrightarrow[T\rightarrow\infty]{P}0$.
\end{prn}
\noindent The next proposition gives the rate of convergence, $T$, for $\hat{\tau}$.
\begin{prn}\label{prn-mle2}
Under the same conditions of Proposition~\ref{prn-mle1}, we have that for every $\epsilon>0$, there exists a $C>0$ such that for large
$T$, $P(T|\hat{s}-s|>C)<\epsilon$.
\end{prn}
\noindent The proofs of Propositions \ref{prn-mle1} and \ref{prn-mle2}
are provided in \ref{AppendixA}. Proposition~\ref{prn-mle1} establishes that the
estimated rate $\hat{s}$ is
consistent for $s^0$ and Proposition~\ref{prn-mle2} shows that the rate of
convergence is $T$. Moreover, in case the shift in the drift parameters is of shrinking magnitude, the discussions in Remark~\ref{remark1} also hold for this case. Propositions \ref{prn-mle1} and \ref{prn-mle2} imply that the proposed estimator satisfies the consistency assumption required in Zhang~(2015). Thus, the asymptotic normality in (\ref{normal1}) also holds when $\hat{\tau}$ are obtained from (\ref{mle1}).

\begin{rem}\label{remark3} To see the connection between equations (\ref{mle1}) and (\ref{cp1}), one may apply the Riemann sum approximation with increment $\Delta_t$ to approximate the integrals inside the log-likelihood function $\log\mathcal{L}(\tau,\hat{\theta})$ specified in (\ref{llf1}). The result of the approximation is of the form $\frac{1}{\sigma^2}\sum_{t_i\in[0,T]}\hat{\theta}(\tau)'V(t)'(X_{t_{i+1}}-X_{t_i})-\frac{1}{2\sigma^2}\sum_{t_i\in[0,T]}(\hat{\theta}(\tau)'V(t)')^2\Delta_t$. Furthermore, if the increment $\Delta_t$ is same as that in the LSSE method,
we have $X_{t_{i+1}}-X_{t_i}=Y_i$ and $V(t)=Z_{i}/\Delta_t$. Then, after some algebraic computations, such approximation can be transformed to $\frac{1}{2\Delta_t\sigma^2}(\sum_{t_i\in[0,T]}Y_i'Y_i- SSE(\tau))$. Hence, for an observed process $X_{t}$, $t\in[0,T]$ with same constant $\Delta_t$ and known $\sigma$, (\ref{cp1}) and the Riemann sum approximation of (\ref{mle1}) are equivalent. This finding will also be confirmed by the simulation results highlighted in Section 6.
\end{rem}

\section{Existence of a change point}
\noindent In Section~\ref{scpe}, we introduced two estimation methods for the case where
the existence of the single change point is affirmative, that is, the number of change
points is known to be 1. In this section, we shall deal with the extended change-point
problem in which the number of change points may be either 0 or 1. In this case,
it is of interest to test the existence of the change point and to determine its exact location if it exists. \\
\ \\
\noindent One popular methodology in the change-point literature in detecting the
unknown number of change points is by treating it as a model-selection problem.
For instance, note that the existence of the change point in (\ref{ou1}) also increases
the number of drift parameters from $p+1$ to $2(p+1)$. Hence, detecting the existence
of change points is equivalent to selecting a statistical model from two candidate
models and this could be solved by using the informational approach. This approach deems
that the most appropriate model is the one that minimises the log-likelihood-based information criterion
\begin{equation}\label{ic} \mathcal{IC}(m)=-2 \log\mathcal{L}(\hat{\tau}, \hat{\theta})
+(m+1)h(p)\phi(T). \end{equation}
In \eqref{ic}, $\log\mathcal{L}(\tau, \hat{\theta})$ is defined in (\ref{llf1}); $\hat{\tau}$
is obtained via (\ref{mle1}) corresponding to each $m$, where $m$ is the potential number of change points to
be determined ($m=0$ or 1 in this case); $h(p)=p+1$ if there is no
change in the diffusion coefficient $\sigma$ before and after the change point or $p+2$ if there is a change in $\sigma$ (i.e. $\sigma=\sigma^{(1)}$ for $t\in [0,\tau^0]$ and $\sigma=\sigma^{(2)}$ for $t\in (\tau^0,T]$); and $\phi(T)$ is a non-decreasing
function of $T$.  \\
\ \\
\noindent  Note that if the number of change points is known, then the term
$(m+1)h(p)\phi(T)$ is fixed, and (\ref{ic}) is equivalent to the maximum
log likelihood method introduced in the previous section. The efficiency of the
information criterion depends on the choice of the penalty criterion $\phi(T)$.
For example, if $\phi(T)=2$,
then (\ref{ic}) reduces to the well-known Akaike information criterion (AIC)~(1973).
However, in practice, a model selected by minimising the AIC may not
be asymptotically consistent in terms of the model order; see for example, Schwarz~(1978).
Many modified versions, thus, were proposed to overcome this problem. One of the
modifications is the Schwarz information criterion (SIC) (1978)
entails the setting of $\phi(T)$ as the log transform of the sample size.
SIC has been
successfully applied to the change-point analysis in the
literature, and it gives an asymptotically consistent estimate
of the order of the true model. Hence, we only focus on SIC on this
particular theoretical development.  \\
\ \\
\noindent Further, it may be of interest to see which of the two penalty criteria we should use: $\phi(T)=\log (T)$
or $\phi(T)=\log (T/\Delta_t)$, where $\Delta_t$ is the increment defined in the
previous section. Hence, in the ensuing discussion of our examples, we take into account these two criteria. (Note that $\log T$ is just a special case of
$\log(T/\Delta_t)=\log(T)-\log(\Delta_t)$ with $\Delta_t=1$).\\
\\
\noindent Consider the hypothesis
\begin{equation}\label{hp-1}\textrm{H}_0:\quad m^0=0 \quad \textrm{versus}
\quad \textrm{H}_1:\quad m^0=1.\end{equation}
Based on (\ref{ic}), the rejection region for the null hypothesis in (\ref{hp-1}) is given by $\mathcal{IC}(m=0)\geq \mathcal{IC}(m=1)$.  Moreover, the asymptotic significant level
and power of the above test are investigated via the following results.

\begin{prn}\label{prn-ictest}
Suppose Assumptions 1--2 and \ref{asm4} hold. Then, under $H_0$ in (\ref{hp-1}),\\ $\displaystyle{\lim_{T\rightarrow \infty}P\left(\mathcal{IC}(m=0)\geq \mathcal{IC}(m=1)\right)=0}$. Moreover, under $H_1$, $\displaystyle{\lim_{T\rightarrow \infty}P\left(\mathcal{IC}(m=0)> \mathcal{IC}(m=1)\right)=1}$.
\end{prn}
\noindent The proof of Proposition \ref{prn-ictest} is presented in \ref{append-5}. \\
\ \\
\noindent Let $\hat{m}=\arg\min_{m\in\{0,1\}}\mathcal{IC}(m)$ with $\phi(T)=\log T$ or $\log(T/\Delta_t)$.  We then have the following.
\begin{cor}\label{prn-ic}
Under Assumptions 1--2 and \ref{asm4},
$\hat{m}-{m^0}\xrightarrow[T\rightarrow\infty]{P}0.$
\end{cor}
\noindent The proof of Corollary~\ref{prn-ictest} is immediate from Proposition~\ref{prn-ictest}.\\
\ \\
For a fixed $\Delta_t$ and large $T$, Corollary \ref{prn-ic}
shows that the two criteria, $\log (T)$ and $\log (T/\Delta_t)$,
lead to asymptotically consistent estimate
of the number of change points. For small $T$, we use Monte-Carlo
simulation to compare the performance of these two criteria. Simulation results indicate
that it would be more appropriate to use $\phi(T)=\log(T/\Delta_t)$ when $T$
is small as $\phi(T)=\log T$ tends to over-estimate the number of change points,
i.e., overfitting the model.

\section{Numerical demonstrations}
\noindent In this numerical work, we use in Subsection 6.1 the Monte-Carlo simulation technique to evaluate the performance
of the (i) two estimation methods proposed in (\ref{cp1}) and (\ref{mle1}) in detecting the unknown
location of a change point assumed to already exist, and (ii) method in (\ref{ic}) for testing the existence of
a change point. In Subsection 6.2, we implement the above methods on some observed financial market data as illustrate the various implementation details.
\ \\
\subsection{Monte-Carlo simulation study}
\noindent Our simulation considers two different scenarios. In the first scenario, we study
the performance of the proposed methods under a classical OU process. In the second scenario, the performance
evaluation of the proposed methods is applied to a periodic mean-reverting OU process.
Each scenario consists of 1000 iterations. In each iteration,
we first generate a desired simulated process based on the
Euler-Maruyama discretisation scheme given a period $T$ and pre-assigned ``true" parameters such as the
coefficients and rate of change point. Next, we estimate and record the rate of change
points by applying (\ref{cp1}) and (\ref{mle1}) on the simulated process as well as the number of
change points estimated by (\ref{ic}).
To investigate the performance of (\ref{ic}),
assuming there is no change point,  we re-generate a simulated process with no change
point and apply (\ref{ic}) to estimate and record again the number of change points.
After 1000 iterations, we analyse the performance of the proposed methods based on the
recorded results.\\

\subsubsection{Simulation setup}
\noindent {\bf Scenario 1: Classical OU process}\\
\noindent
Two classical OU processes are considered with stochastic dynamics
\begin{equation}\label{sim1-1}
\textit{d}X_t=
\begin{cases}
(0.08-0.1X_t)\textit{d}t+0.2 \textit{d}W_t,  & \mbox{if }0<t<0.5 T \\
(2.5-1X_t)\textit{d}t+0.2 \textit{d}W_t,  & \mbox{if }0.5 T<t< T.
\end{cases}
\end{equation}

\noindent Equation (\ref{sim1-1}) includes one change point occurring at $\tau^0=0.5 T$ ($s^0=0.5$).
This process is generated to determine the performance of the methods proposed in (\ref{cp1}) and (\ref{mle1}). \\
\ \\
\noindent When there is one change point, we study the performance of (\ref{ic}) by
considering the SDE
\begin{equation}\label{sim1-2}
\textit{d}X_t=(2.5-X_t)\textit{d}t+0.2 \textit{d}W_t, \quad 0<t<T. \\
\end{equation}

\noindent {\bf Scenario 2: Periodic-mean reverting OU process}\\
\noindent For this scenario, we consider a mean-reverting OU process,
with 2-dimensional periodic incomplete orthogonal set of functions
$\left\{1, \sqrt{2}\cos\left( \frac{\pi t}{2\Delta_t} \right)\right\}$,
given by
\begin{equation}\label{sim2-1}
\textit{d}X_t=
\begin{cases}
\left[ 0.08+0.02 \sqrt{2}\cos\left( \frac{\pi t}{2\Delta_t}\right)-0.1X_t \right]\textit{d}t+0.2 \textit{d}W_t,
& \mbox{if }0<t<0.5 T \\
\left[ 2.5+1.2\sqrt{2}\cos\left(\frac{\pi t}{2\Delta_t}\right)-1X_t \right]\textit{d}t+0.2 \textit{d}W_t,
& \mbox{if }0.5 T<t< T,
\end{cases}
\end{equation}
where $\Delta_t=t_{i+1}-t_i$ is the increment in the given time period $[0,T]$. \\
\ \\
Similarly, we study the performance of (\ref{ic}), under the assumption that there is no change point,
using the SDE
\begin{equation}\label{sim2-2}
\textit{d}X_t=\left(2.5+1.2\sqrt{2}\cos\left( \frac{\pi t}{2\Delta_t} \right)-X_t\right)\textit{d}t+0.2
\textit{d}W_t, \quad \mbox{if }0<t<T. \\
\end{equation}

\noindent We choose $T=5,~10,~20~\mbox{and}~50$, and $\Delta_t=1/252$ and the starting point $X_0=0.05$.
Simulated sample paths for the processes (\ref{sim1-1})--(\ref{sim2-2})
with different time periods ($T=5$ and $50$) are
shown in Figures~\ref{fig:sampleseries1}--\ref{fig:sampleseries4}.

\begin{figure}[htbp]
\includegraphics[height=4in,width=2.9in]{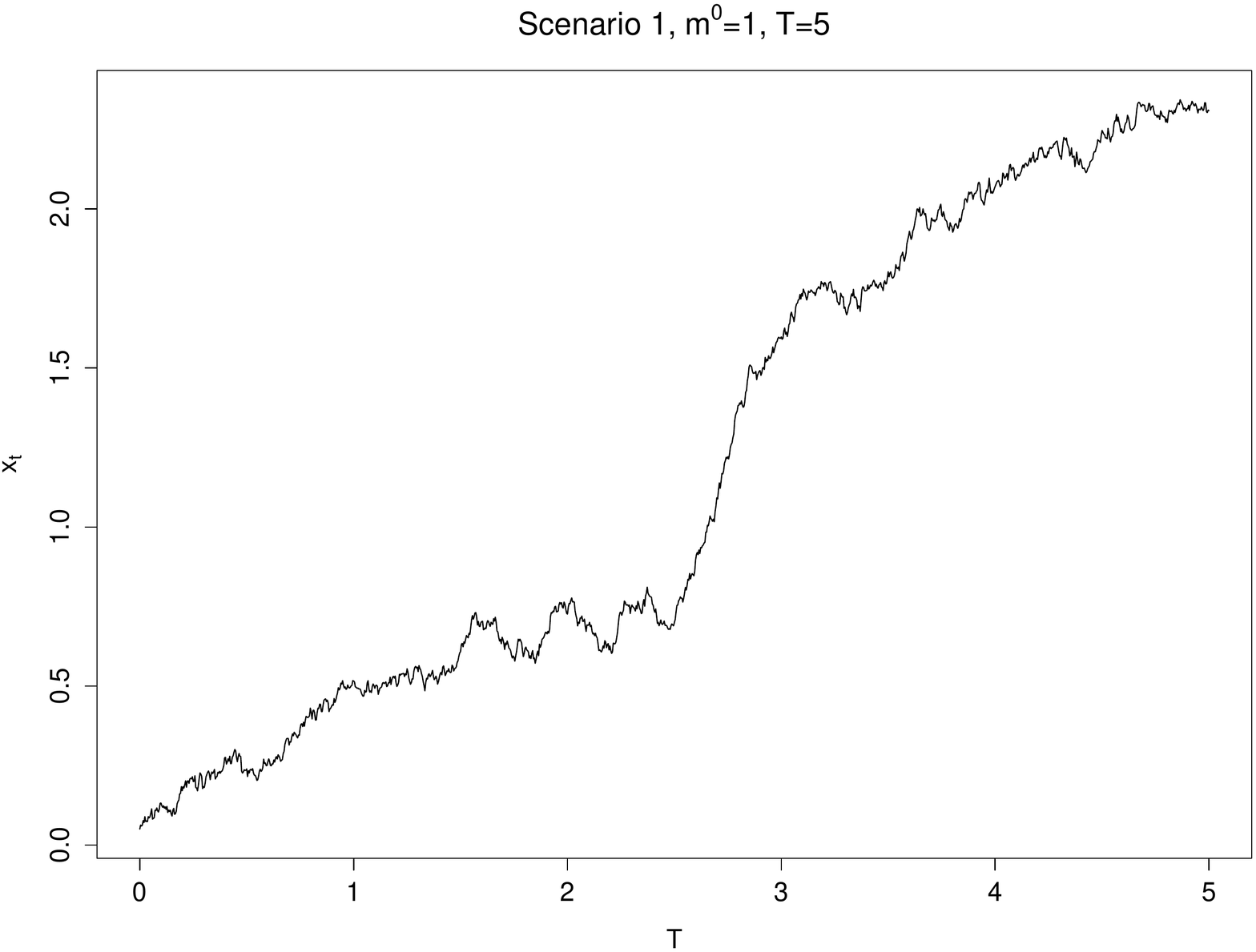}
\includegraphics[height=4in,width=2.9in]{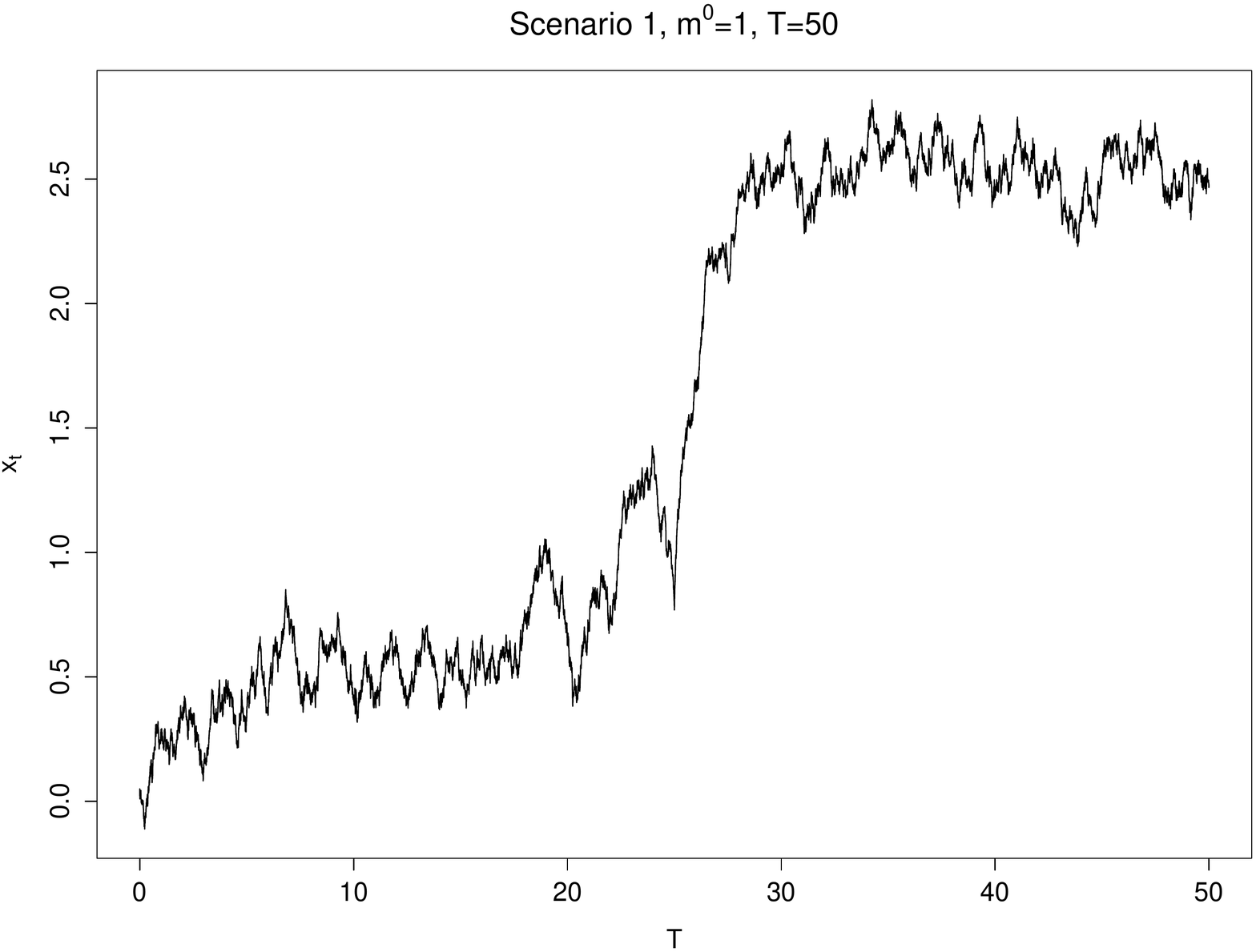}
\caption{\small Sample series for (\ref{sim1-1}) under scenario 1 with one change point}
\label{fig:sampleseries1}
\end{figure}
\begin{figure}[htbp]\label{sampleseries2}
\includegraphics[height=4in,width=2.9in]{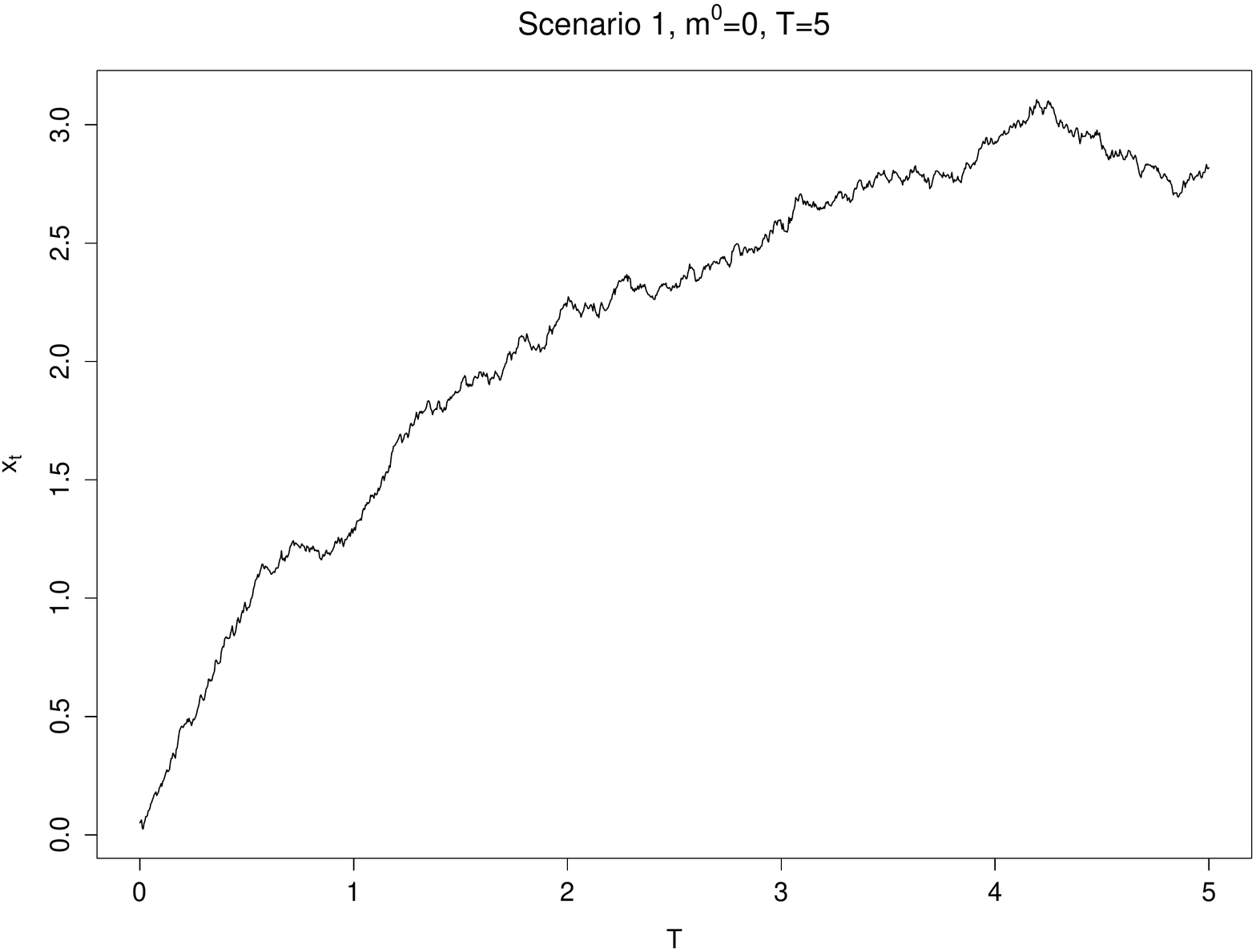}
\includegraphics[height=4in,width=2.9in]{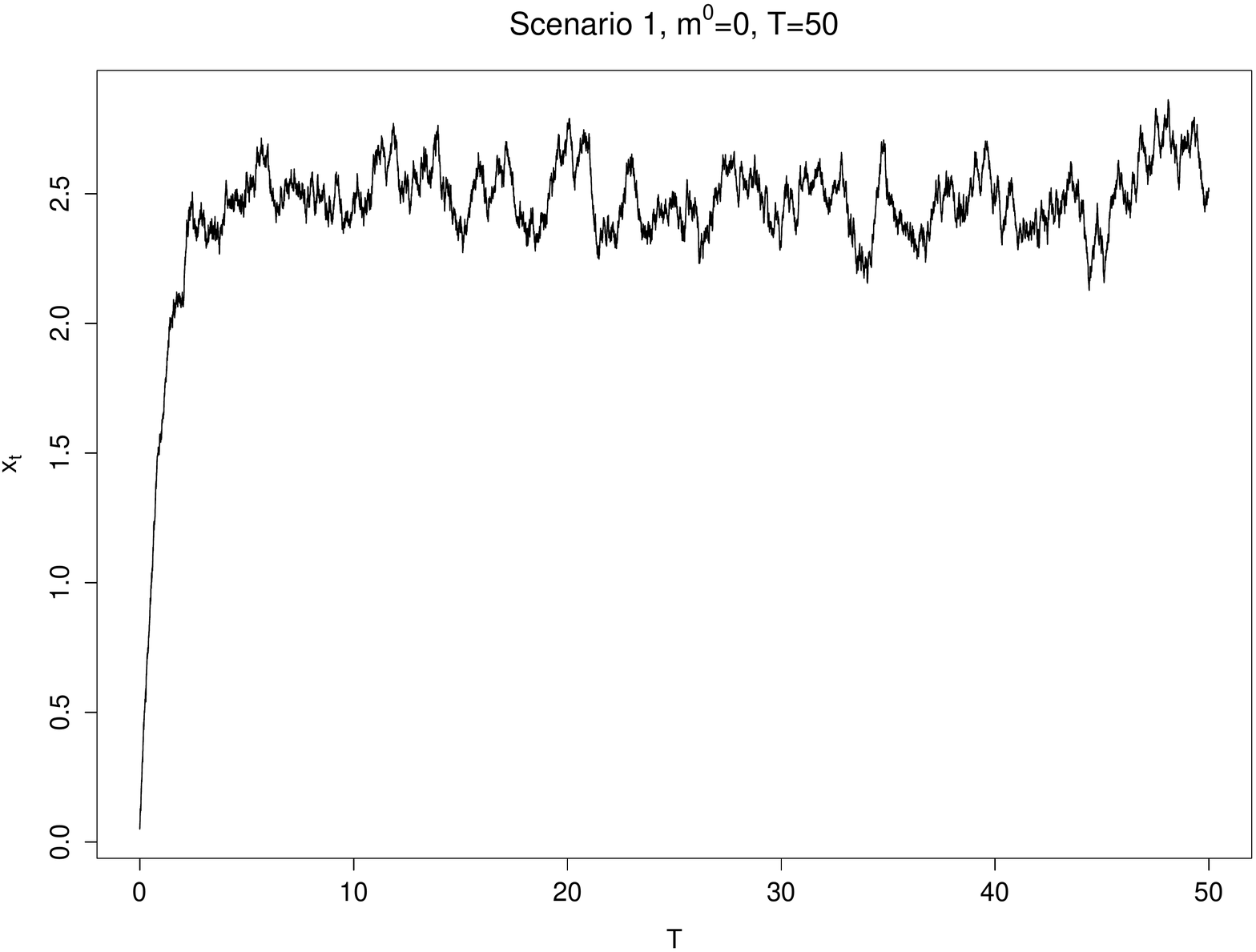}
\caption{\small Sample series for (\ref{sim1-2}) under scenario 1 without a change point}
\end{figure}
\begin{figure}[htbp]\label{sampleseries3}
\includegraphics[height=4in,width=2.9in]{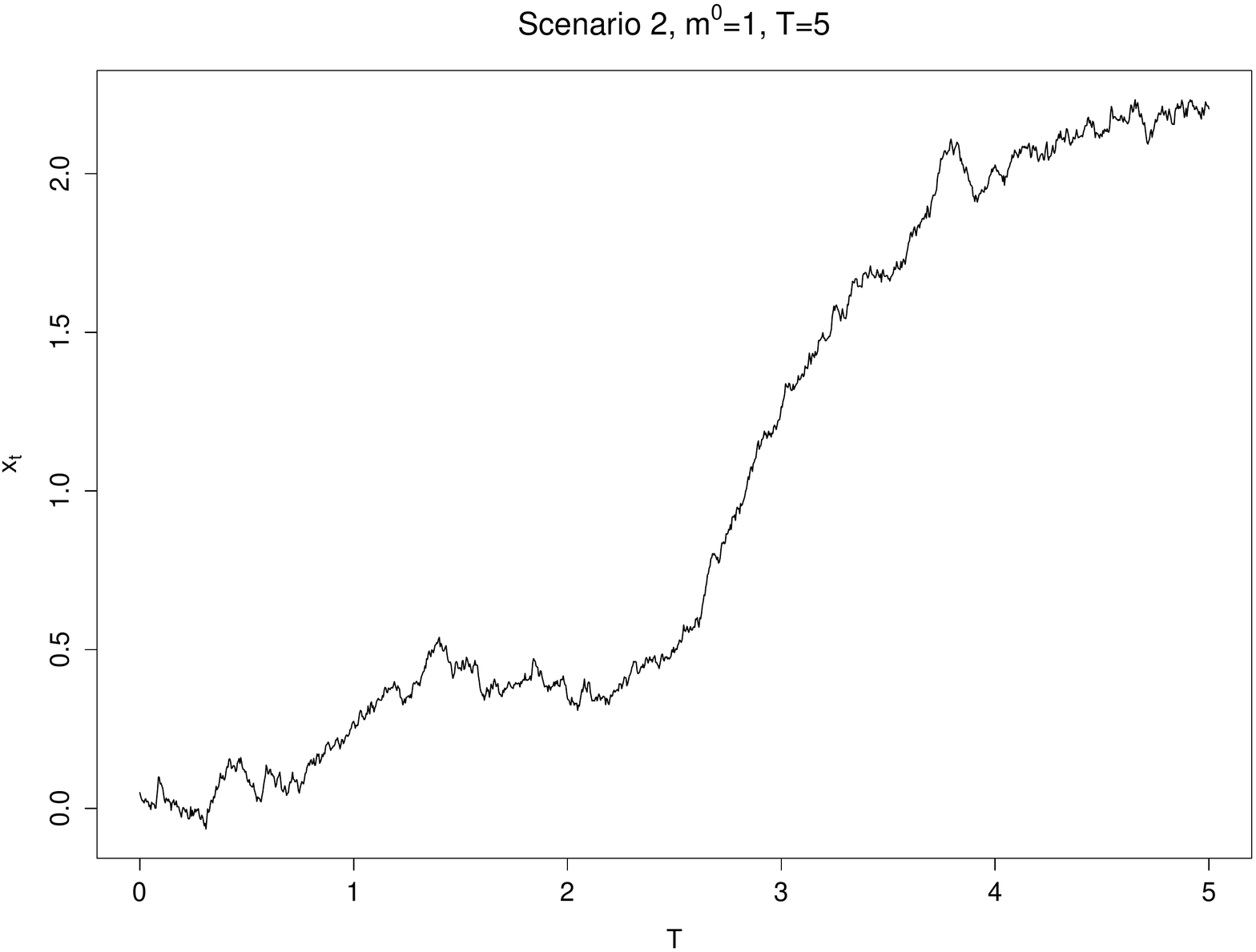}
\includegraphics[height=4in,width=2.9in]{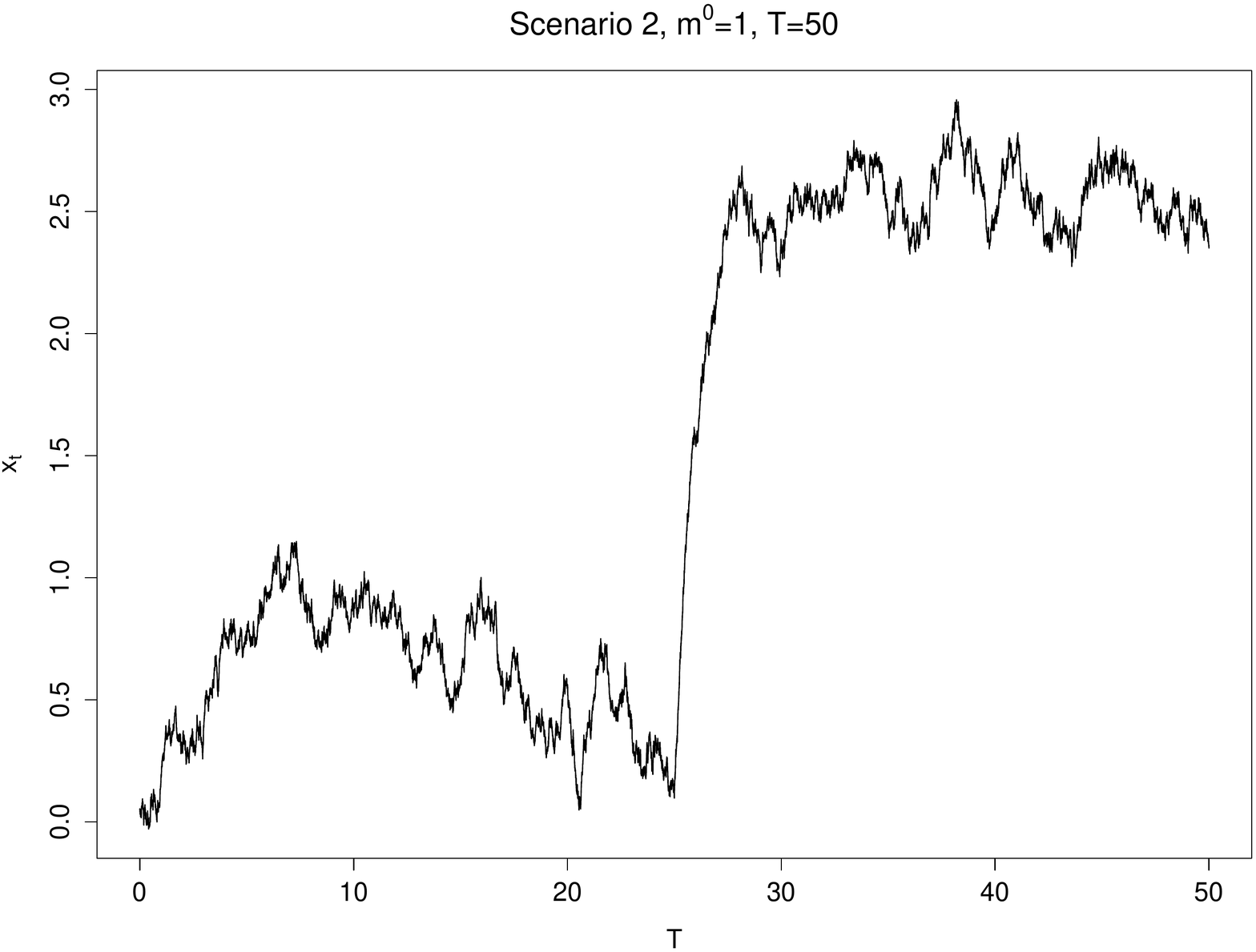}
\caption{\small Sample series for (\ref{sim2-1}) under scenario 2 with one change point}
\end{figure}

\begin{figure}[htbp]
\includegraphics[height=4in,width=2.9in]{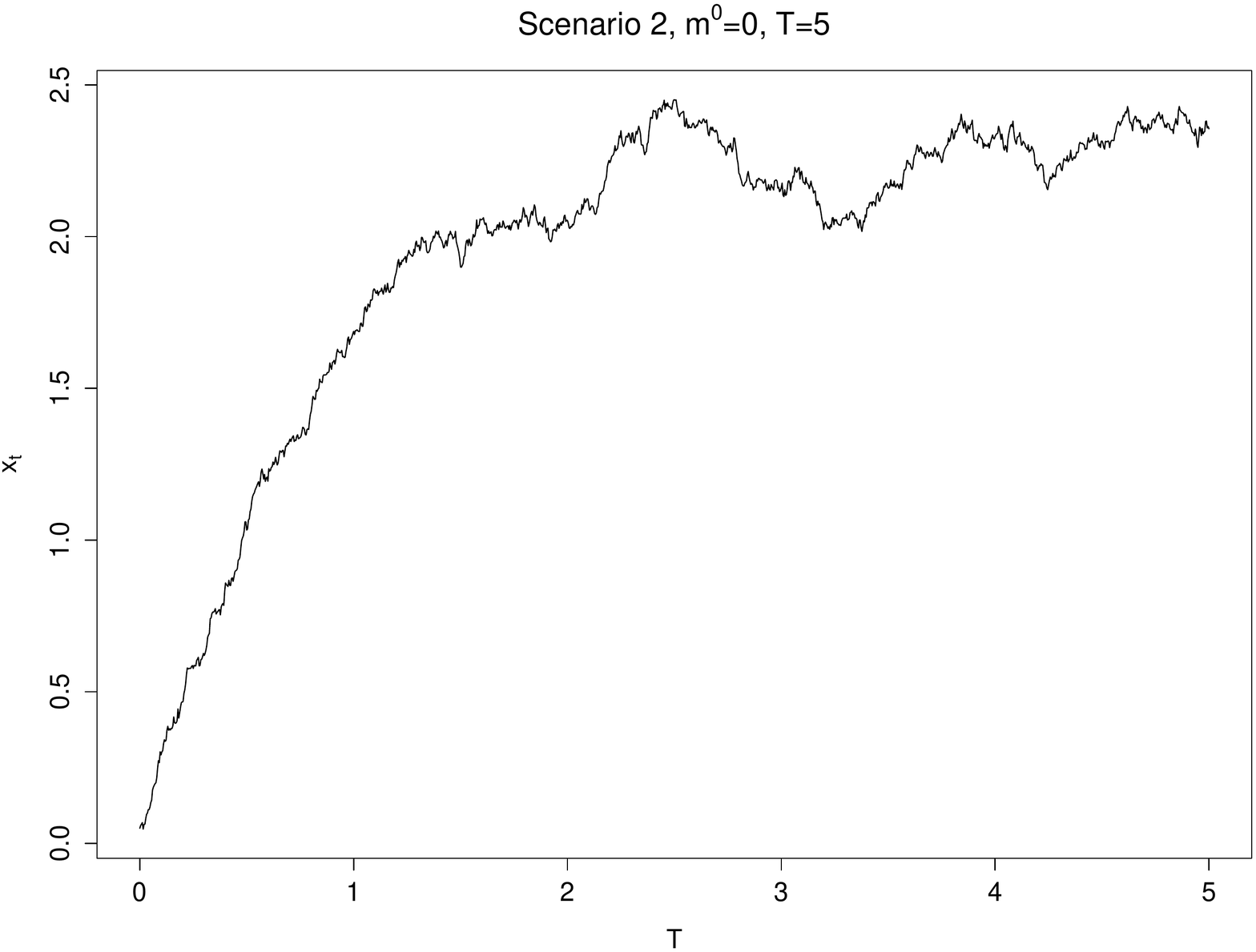}
\includegraphics[height=4in,width=2.9in]{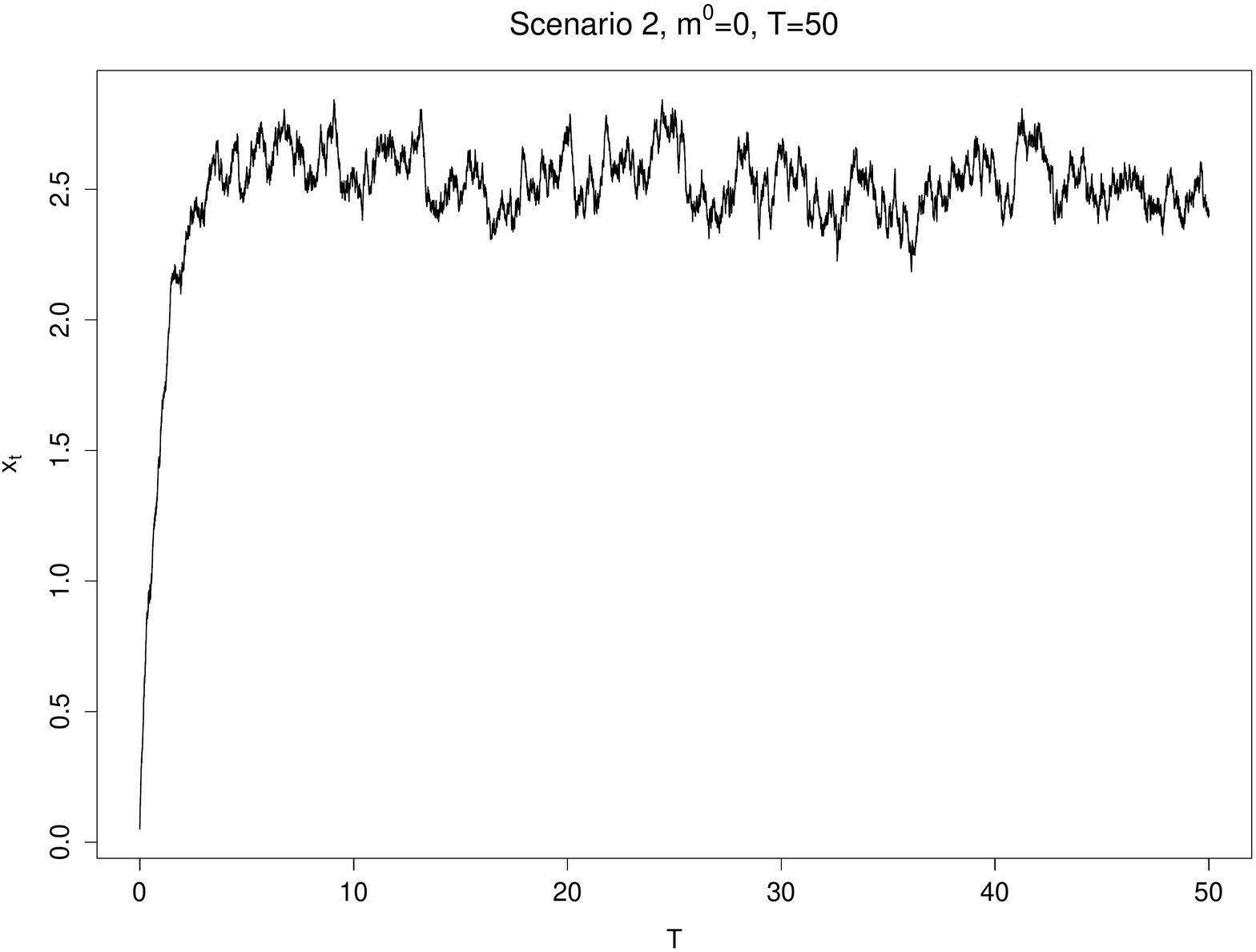}
\caption{\small Sample series for (\ref{sim2-2}) under scenario 2 without a change point}
\label{fig:sampleseries4}
\end{figure}

\subsubsection{Discussion of simulation results}
\noindent {\bf Estimating the rate of change point}\\
\noindent We first look at the performance of (\ref{cp1}) and (\ref{mle1}) in estimating
the corresponding rate of change point in (\ref{sim1-1}) and (\ref{sim2-1}), respectively. Note that
in each iteration, we apply (\ref{cp1}) and (\ref{mle1}) to the same simulated process;
hence, results from these two methods are expected to be close to each other.
The mean and mean-squared error (MSE) of the estimated rate of change point $\hat{s}$
for (\ref{sim1-1}) based on LSSE and MLL methods are summarised in Table~\ref{table1},
and the results for (\ref{sim2-1}) are displayed in Table~\ref{table2}. \\
\ \\
\begin{table}[!htbp]
\small \caption{Mean and MSE of $\hat{s}$ under scenario 1,~(\ref{sim1-1})} \centering
\begin{tabular}{|c|c|c|c|c|}
\hline
 & \multicolumn{2}{c|}{ LSSE method} &  \multicolumn{2}{c|}{MLL method}\\
\hline
  $T$ & Mean & MSE & Mean & MSE\\
\hline
5& 0.4986794& $0.0003834606$ &0.4987587&  0.0003823898\\
\hline
10& 0.499711 & 0.0001417382 & 0.4997503 &0.0001417427\\
\hline
20& 0.5004065 & $8.622138\times 10^{-5}$ & 0.5004069 & $8.623254\times 10^{-5}$\\
\hline
50&0.4999693 & $7.983829\times 10^{-6}$ & 0.4999696 & $7.982264\times 10^{-6}$\\
\hline
\end{tabular}
\label{table1}
\end{table}

\begin{table}[!htbp]
\small \caption{Mean and MSE of $\hat{s}$ under scenario 2,~(\ref{sim2-1})} \centering
\begin{tabular}{|c|c|c|c|c|}
\hline
 & \multicolumn{2}{c|}{ LSSE method} &  \multicolumn{2}{c|}{MLL method}\\
\hline
  $T$ & Mean & MSE & Mean & MSE\\
\hline
5& 0.4992968 & $0.0001146825$ & 0.4996579 &  0.0001042964 \\
\hline
10& 0.5003373 & $1.884511\times 10^{-5}$ & 0.500502 &  $2.130748\times 10^{-5}$ \\
\hline
20&0.5002268 & $6.765125\times 10^{-6}$ & 0.5002948 &  $6.991292\times 10^{-6}$ \\
\hline
50& 0.5001443 & $1.750852\times 10^{-6}$ &0.5001291 & $1.742431\times 10^{-6}$\\
\hline
\end{tabular}
\label{table2}
\end{table}
\noindent To illustrate further the simulated results, the corresponding
histograms for $\hat{s}$ in (\ref{sim1-1}) are depicted in Figure~\ref{figure1}
under scenario 1, and in Figure~\ref{figure2} under scenario 2.

\begin{figure}[htbp]
\includegraphics[height=2.1in,width=2.9in]{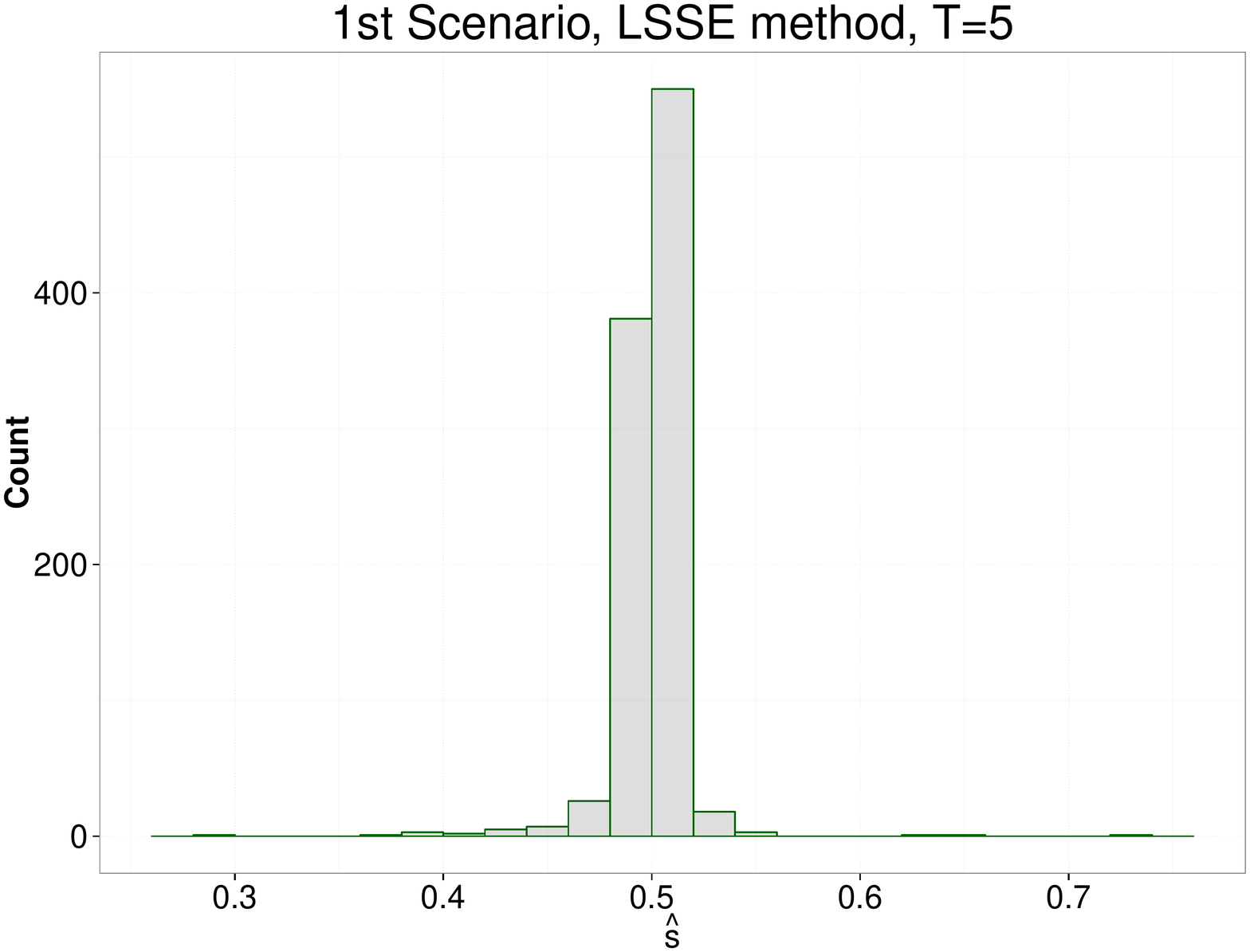}
\includegraphics[height=2.1in,width=2.9in]{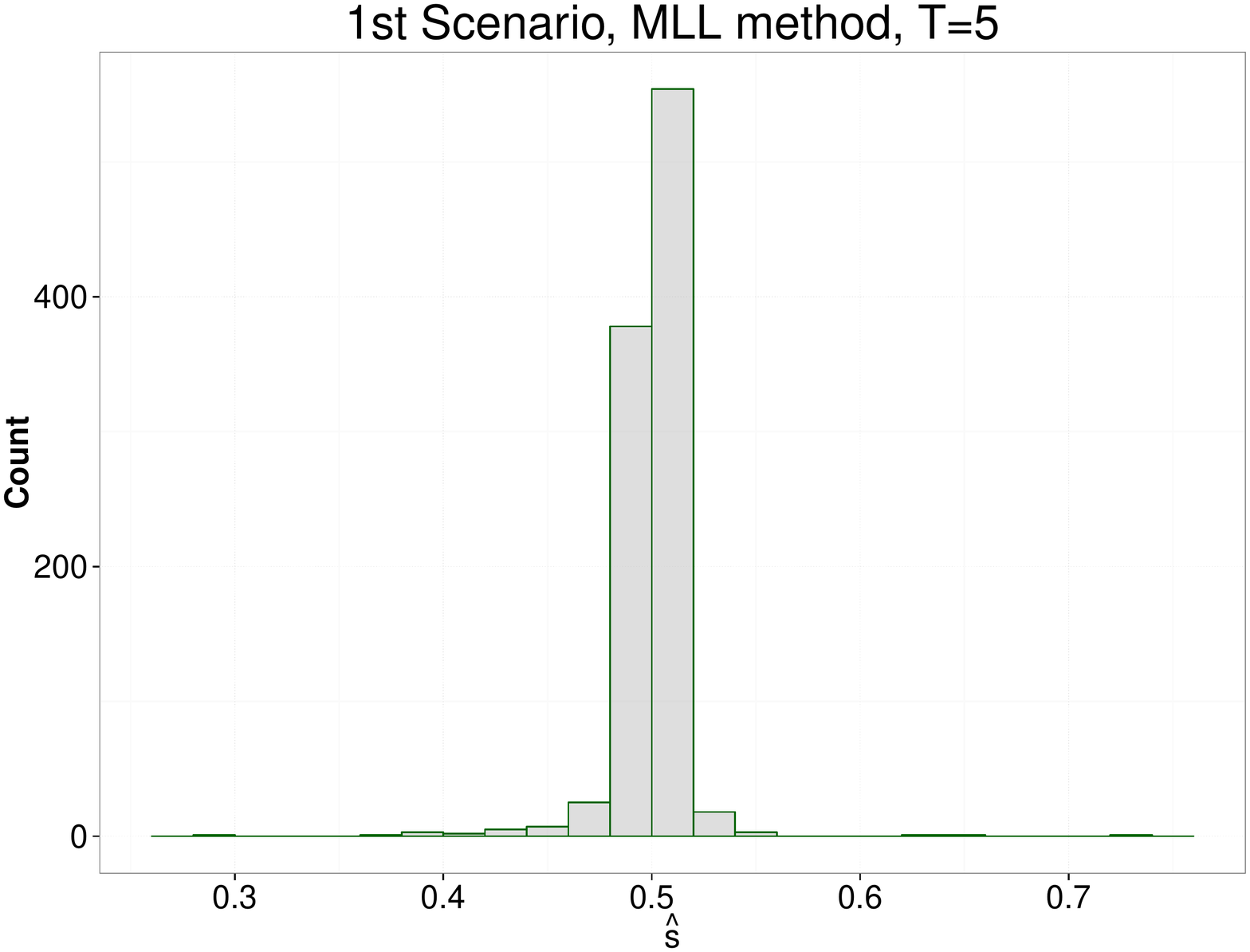}\\
\includegraphics[height=2.1in,width=2.9in]{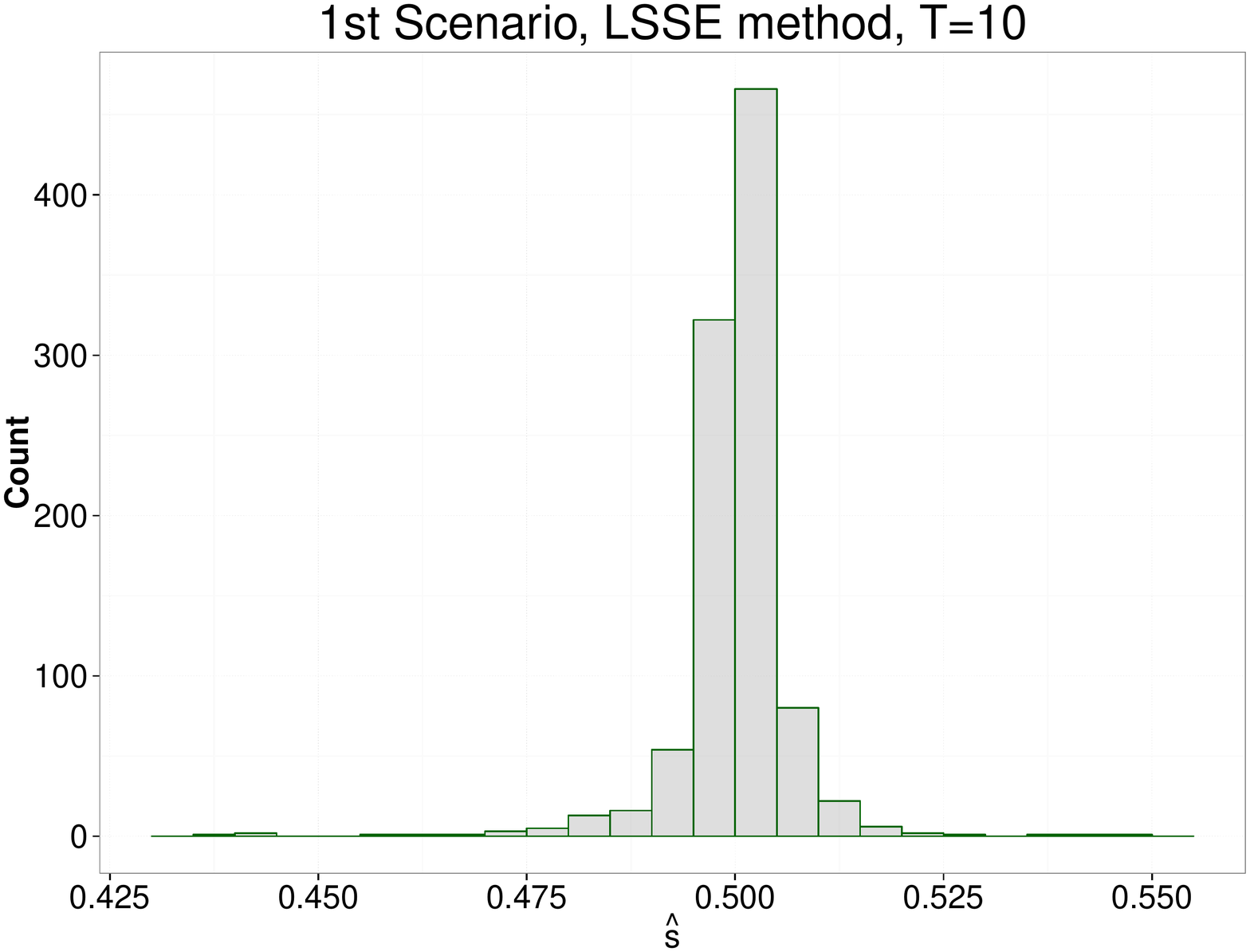}
\includegraphics[height=2.1in,width=2.9in]{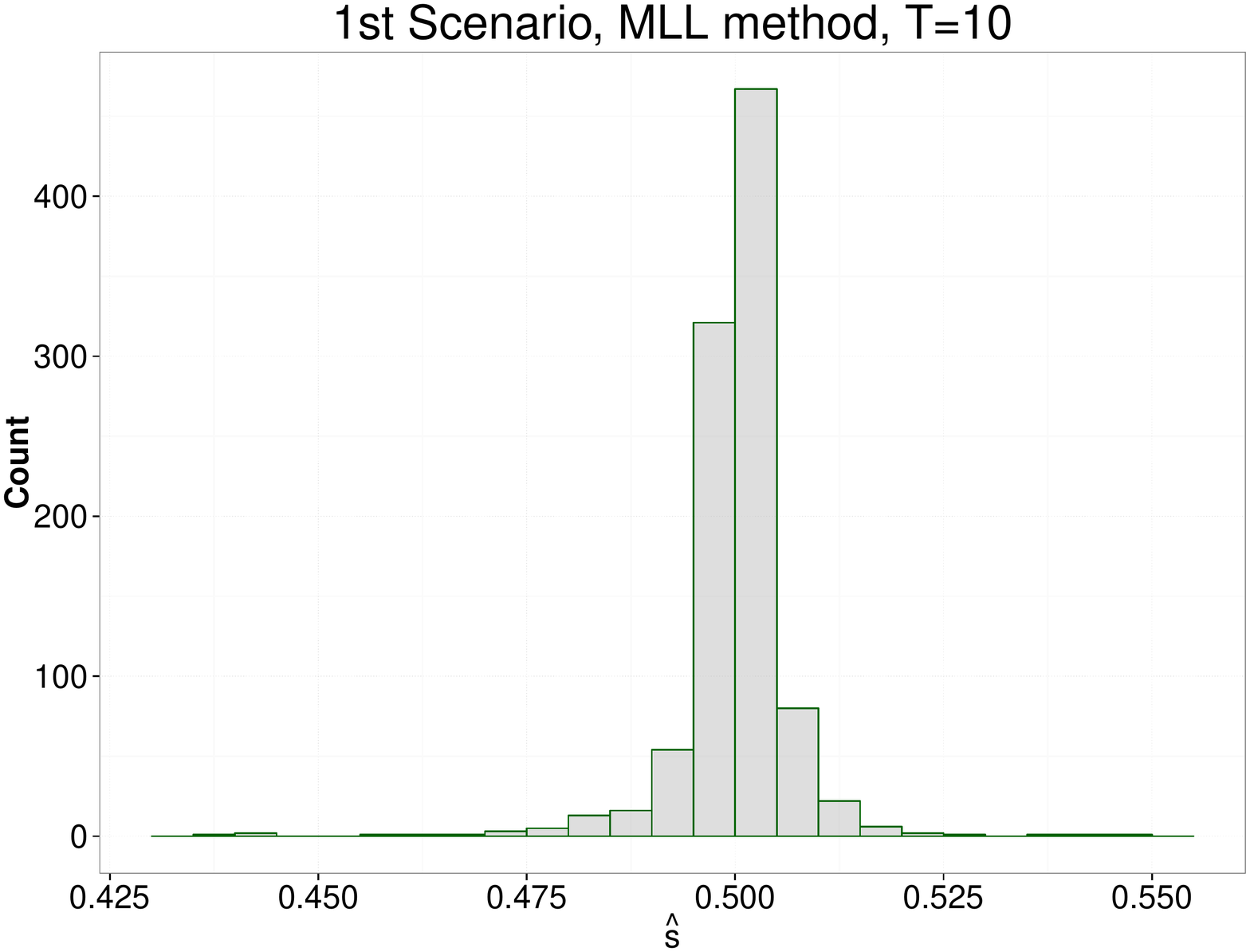}\\
\includegraphics[height=2.1in,width=2.9in]{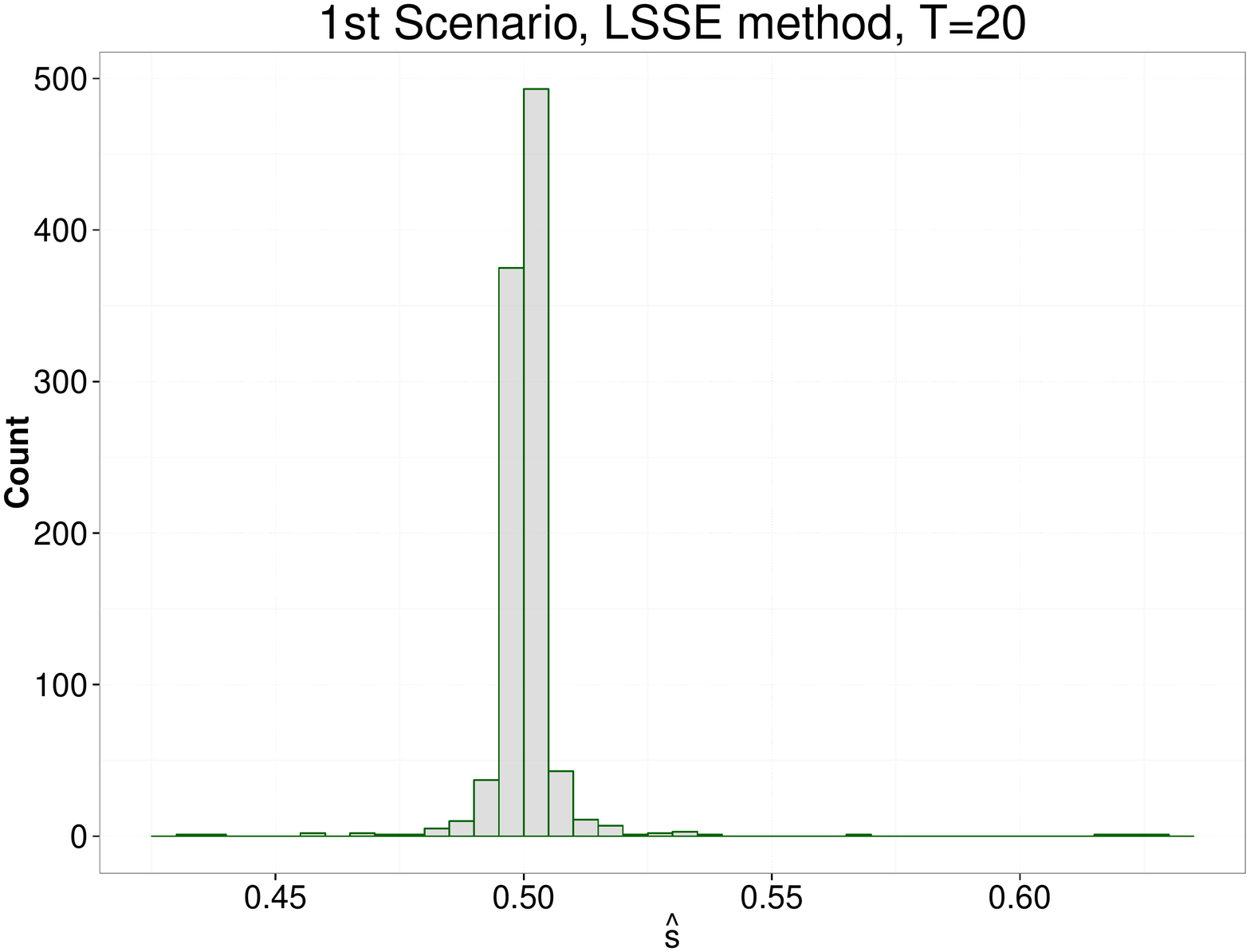}
\includegraphics[height=2.1in,width=2.9in]{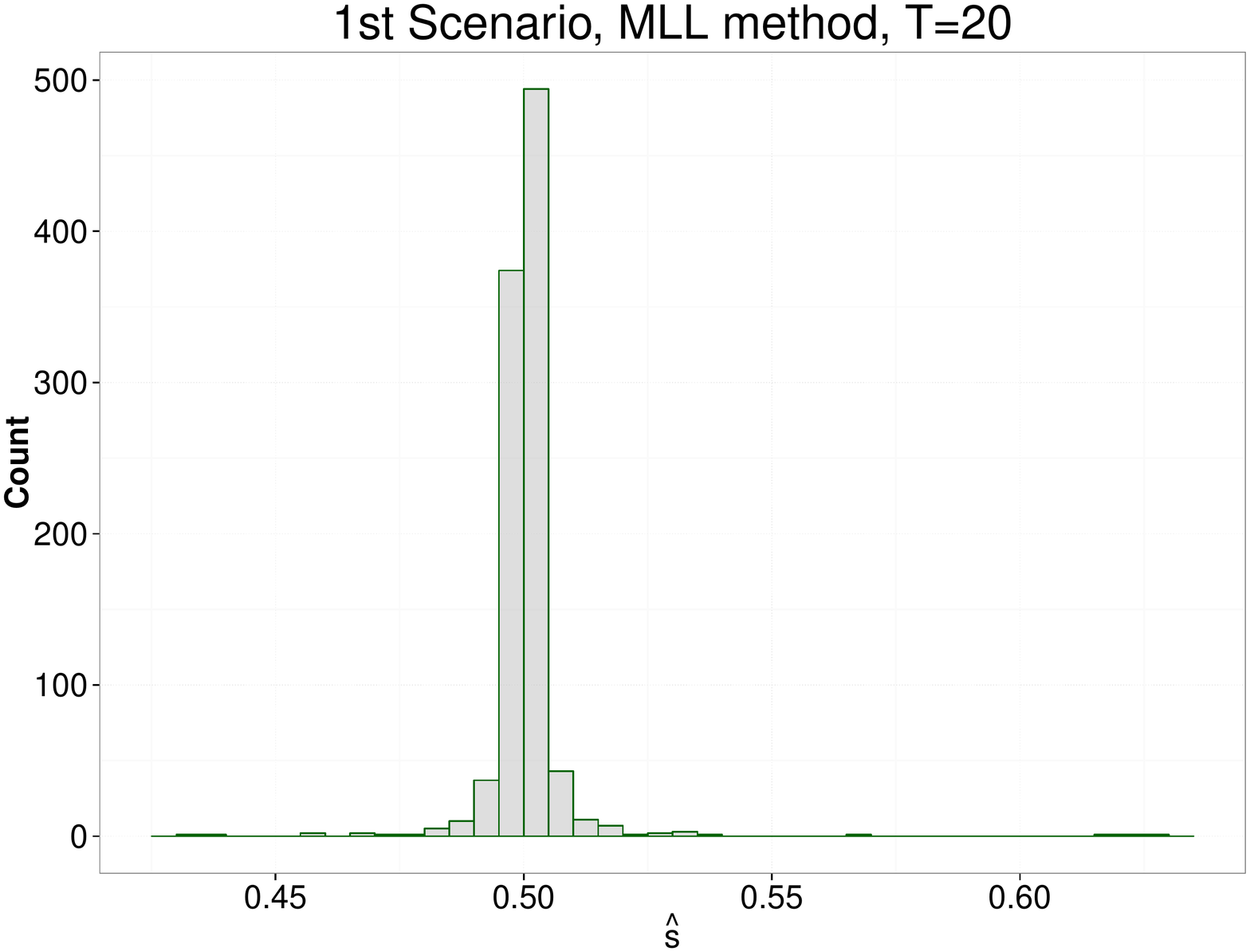}\\
\includegraphics[height=2.1in,width=2.9in]{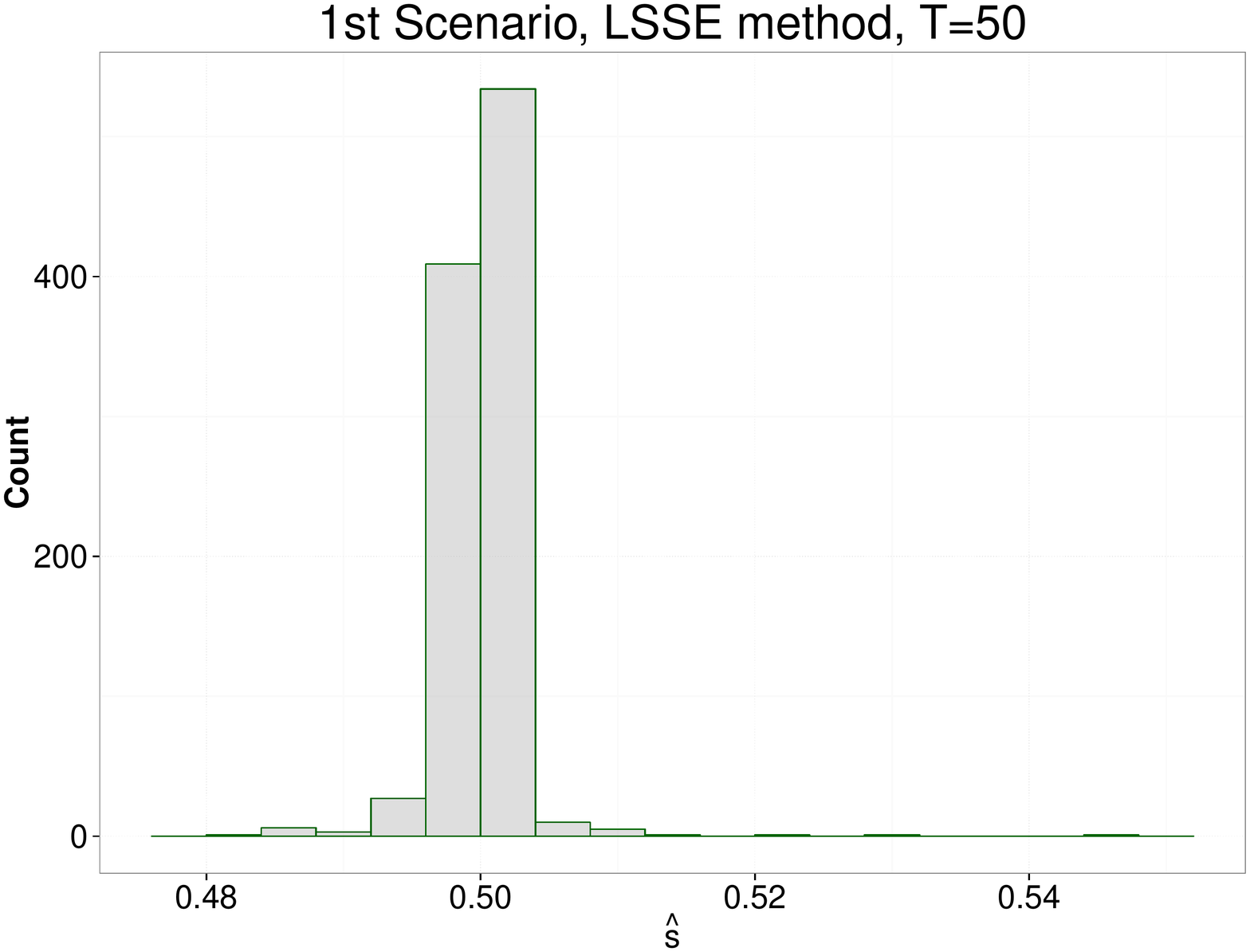}
\includegraphics[height=2.1in,width=3in]{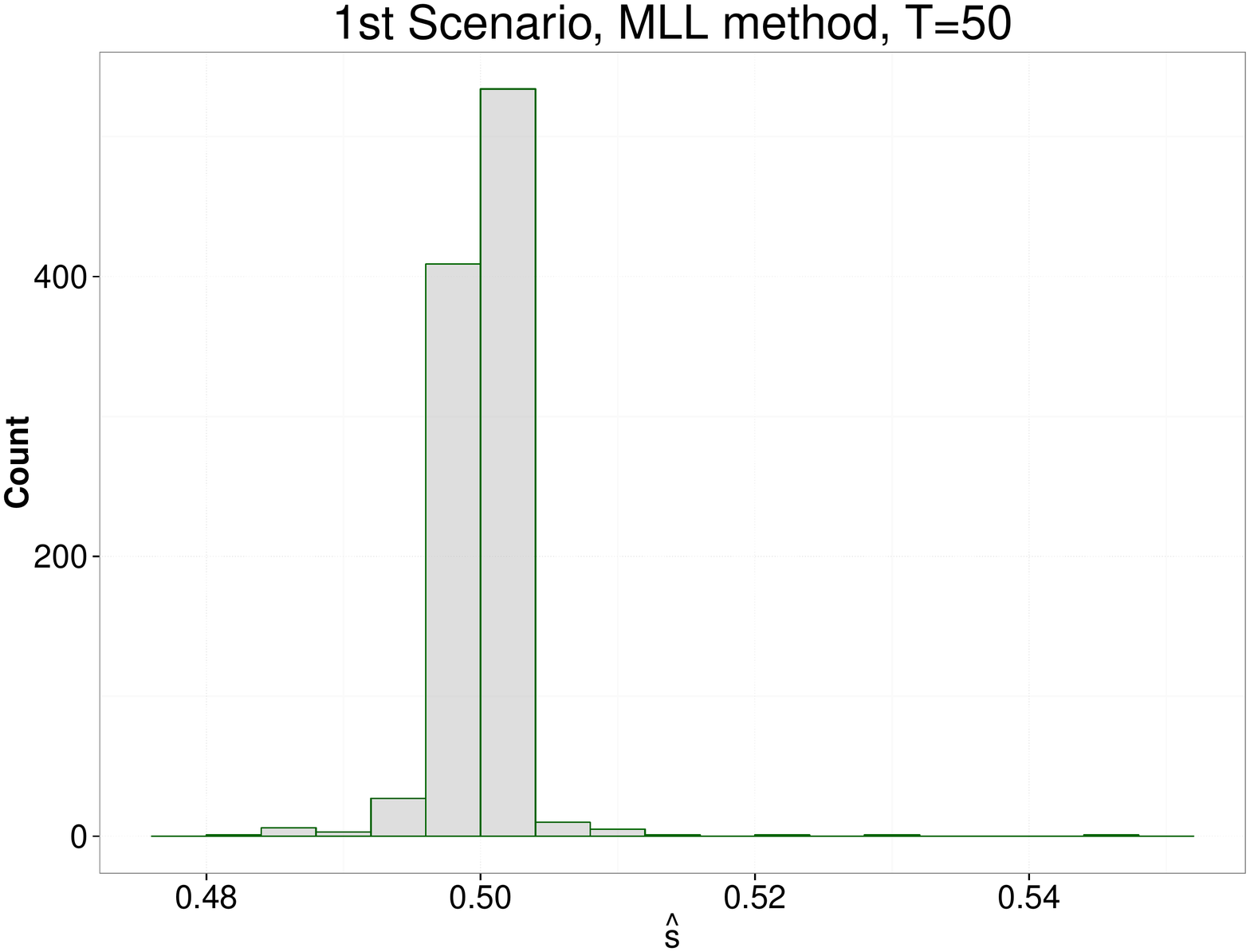}
\caption{\small Histogram of $\hat{s}$ for Scenario 1 (\ref{sim1-1})}
\label{figure1}
\end{figure}

\begin{figure}[htbp]
\includegraphics[height=2.05in,width=2.9in]{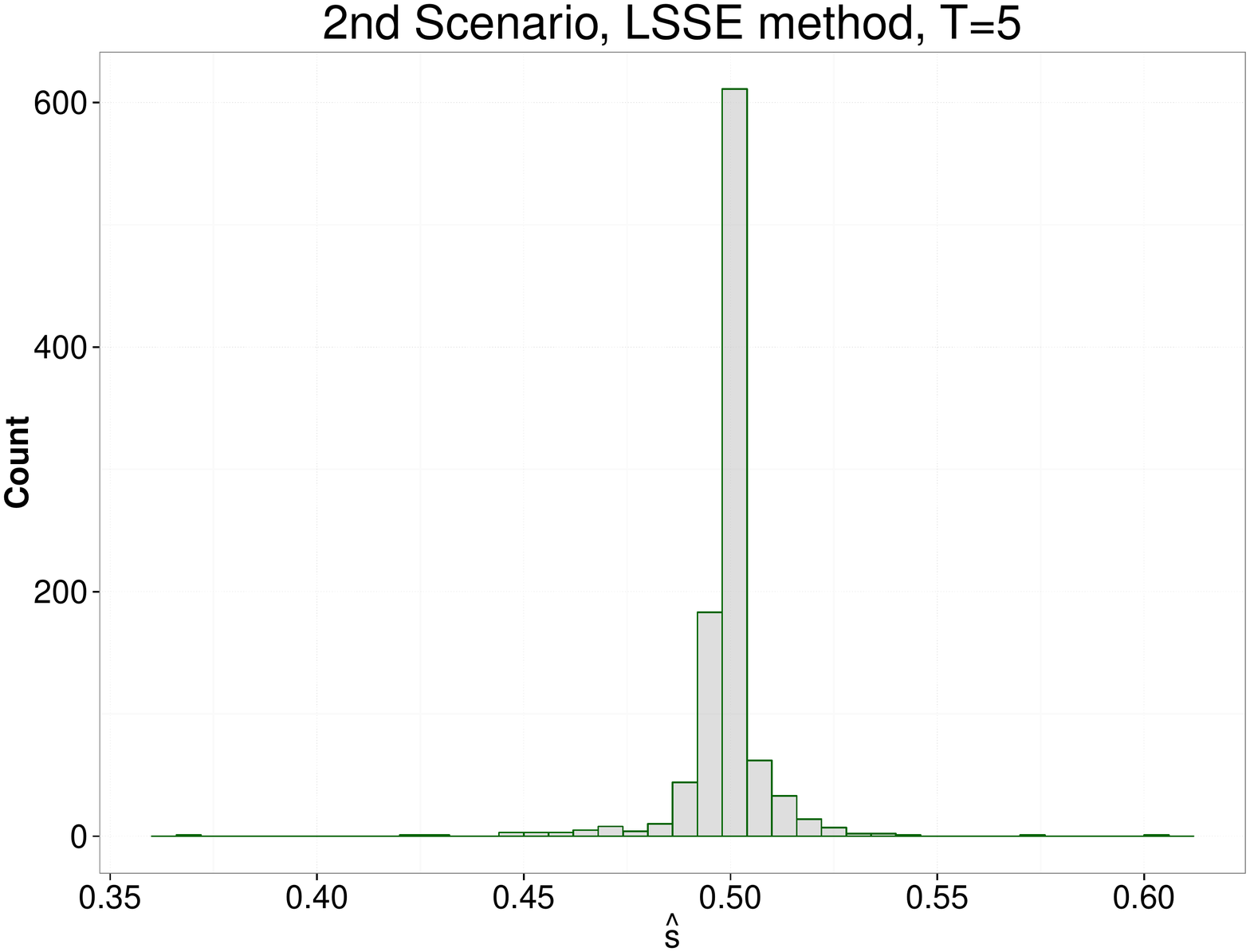}
\includegraphics[height=2.05in,width=2.9in]{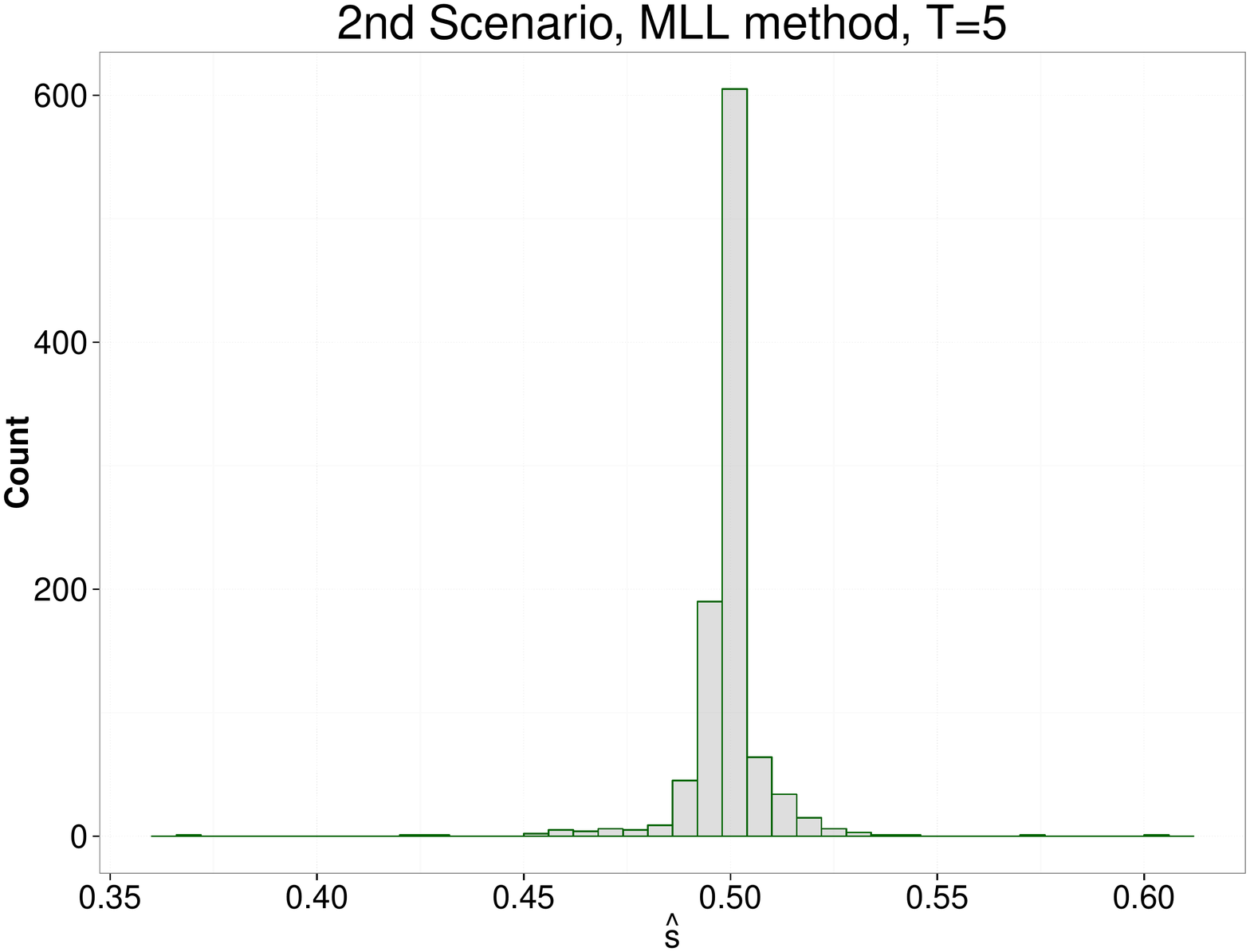}\\
\includegraphics[height=2.05in,width=2.9in]{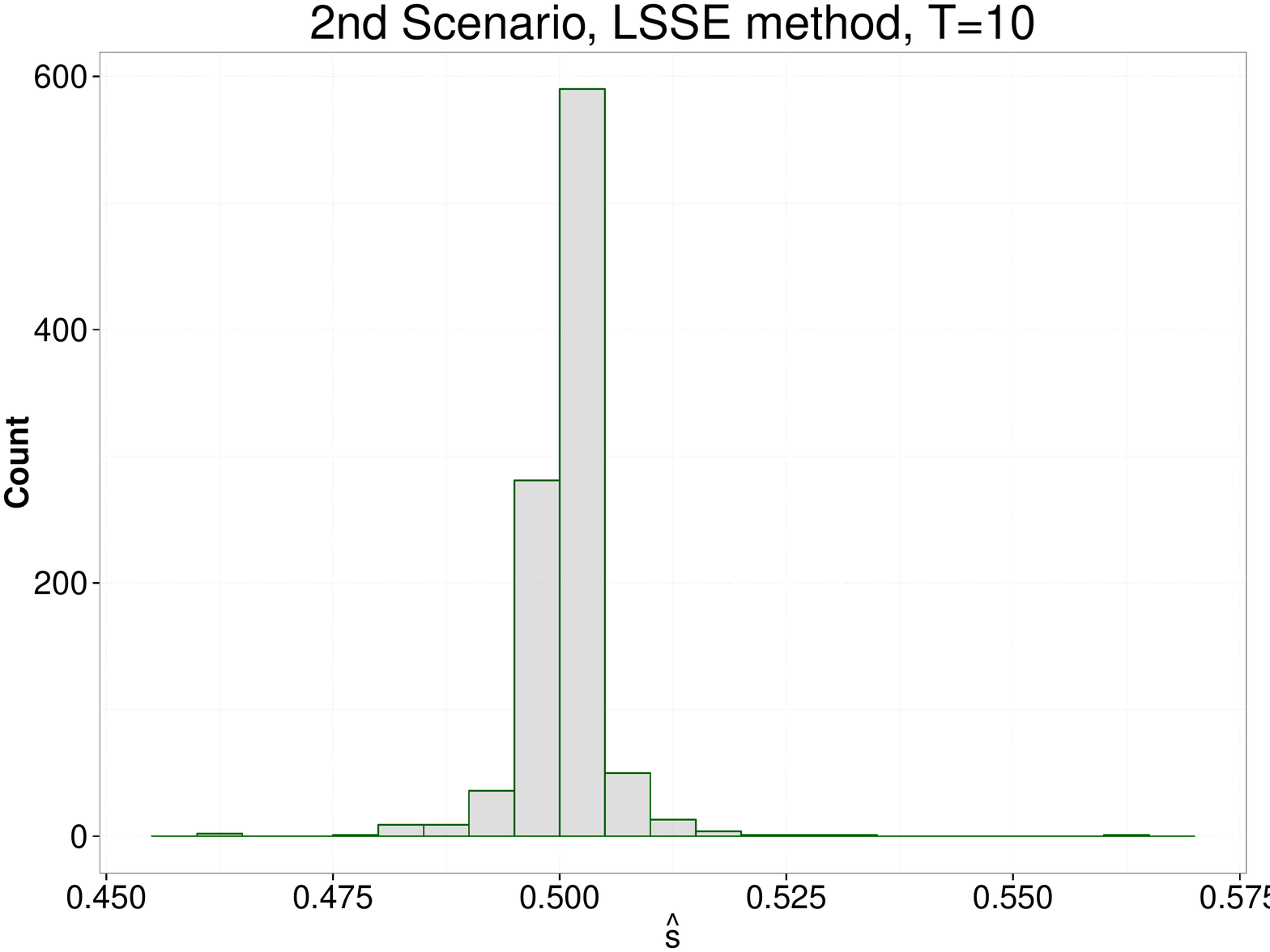}
\includegraphics[height=2.05in,width=2.9in]{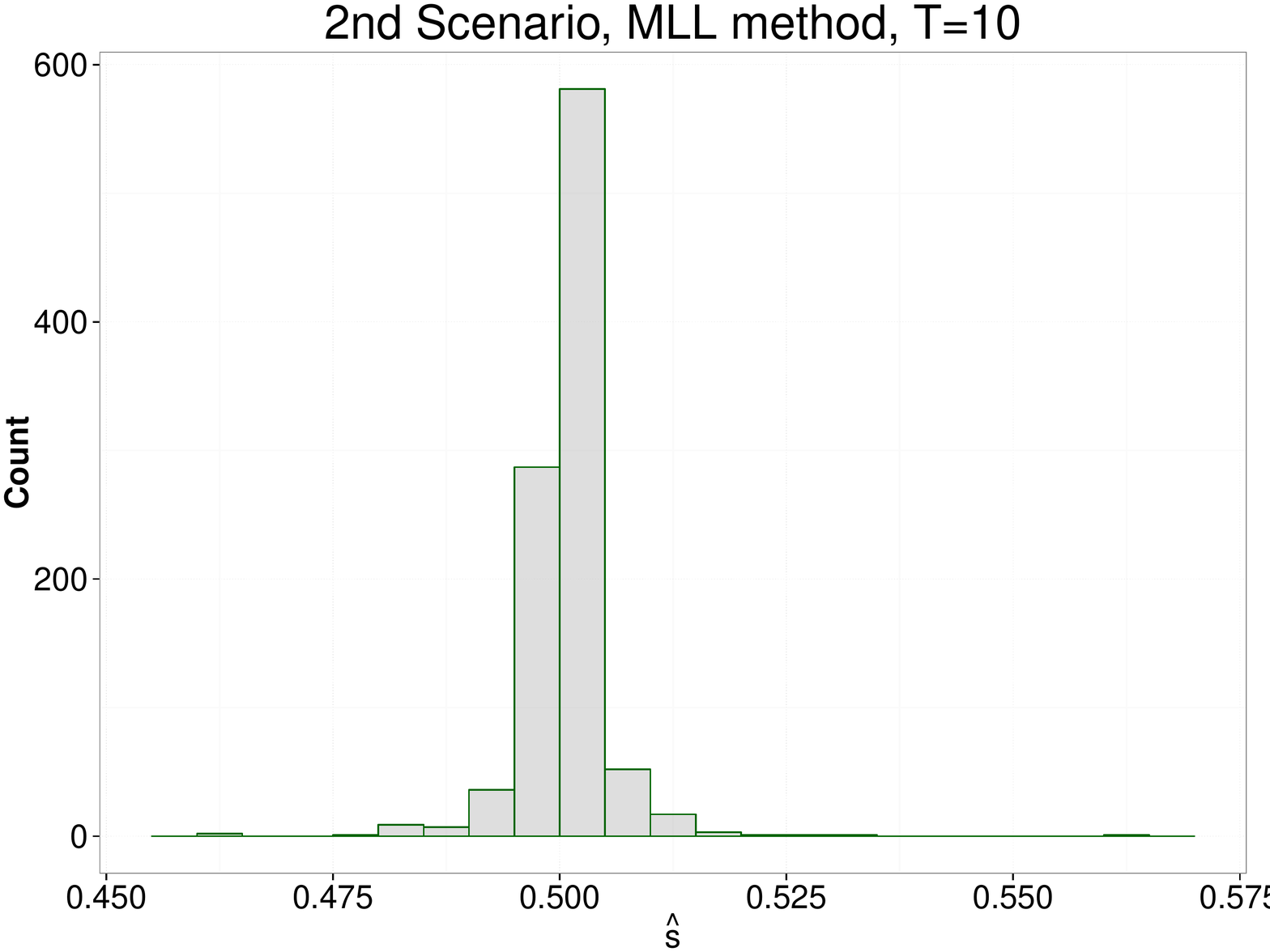}\\
\includegraphics[height=2.05in,width=2.9in]{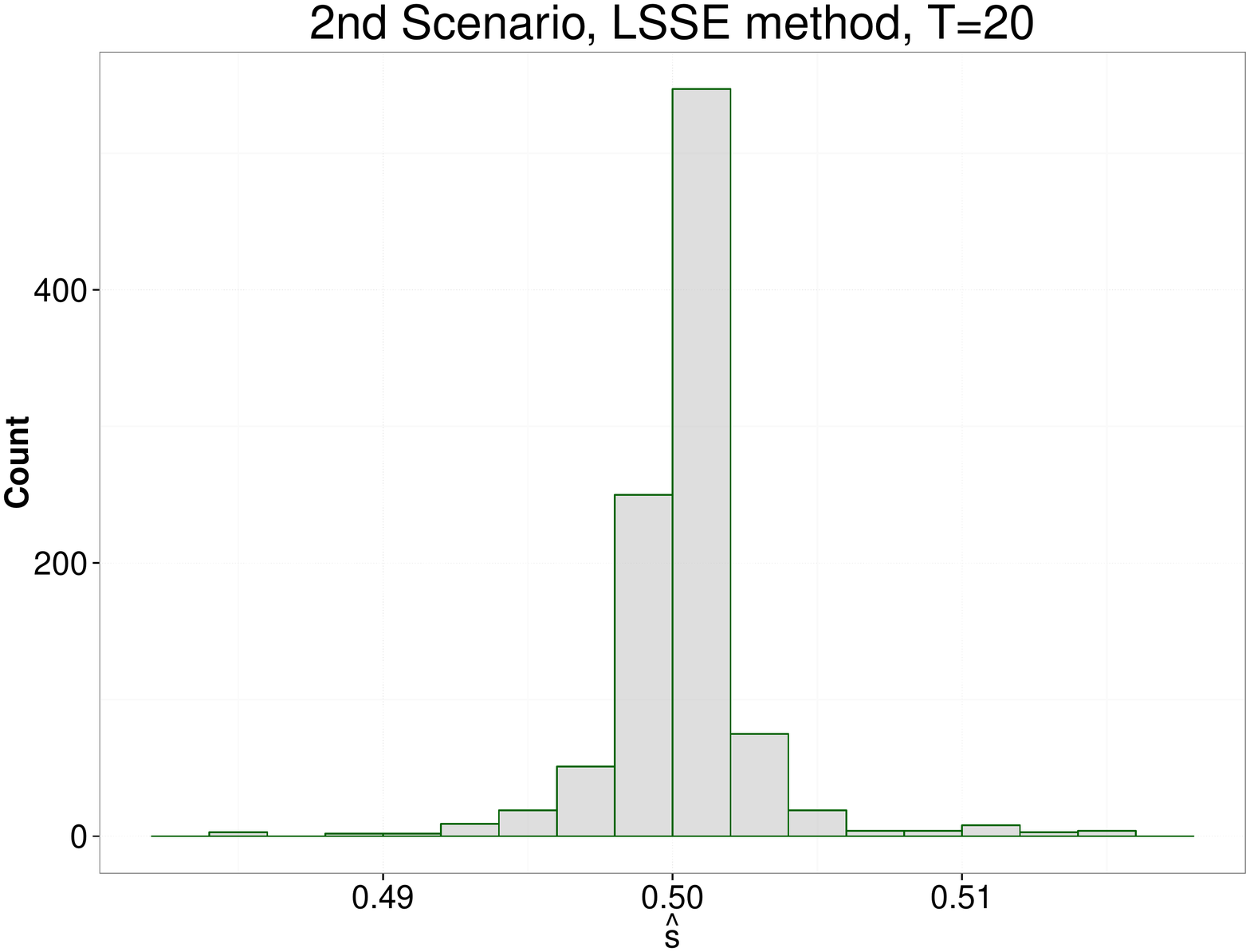}
\includegraphics[height=2.05in,width=2.9in]{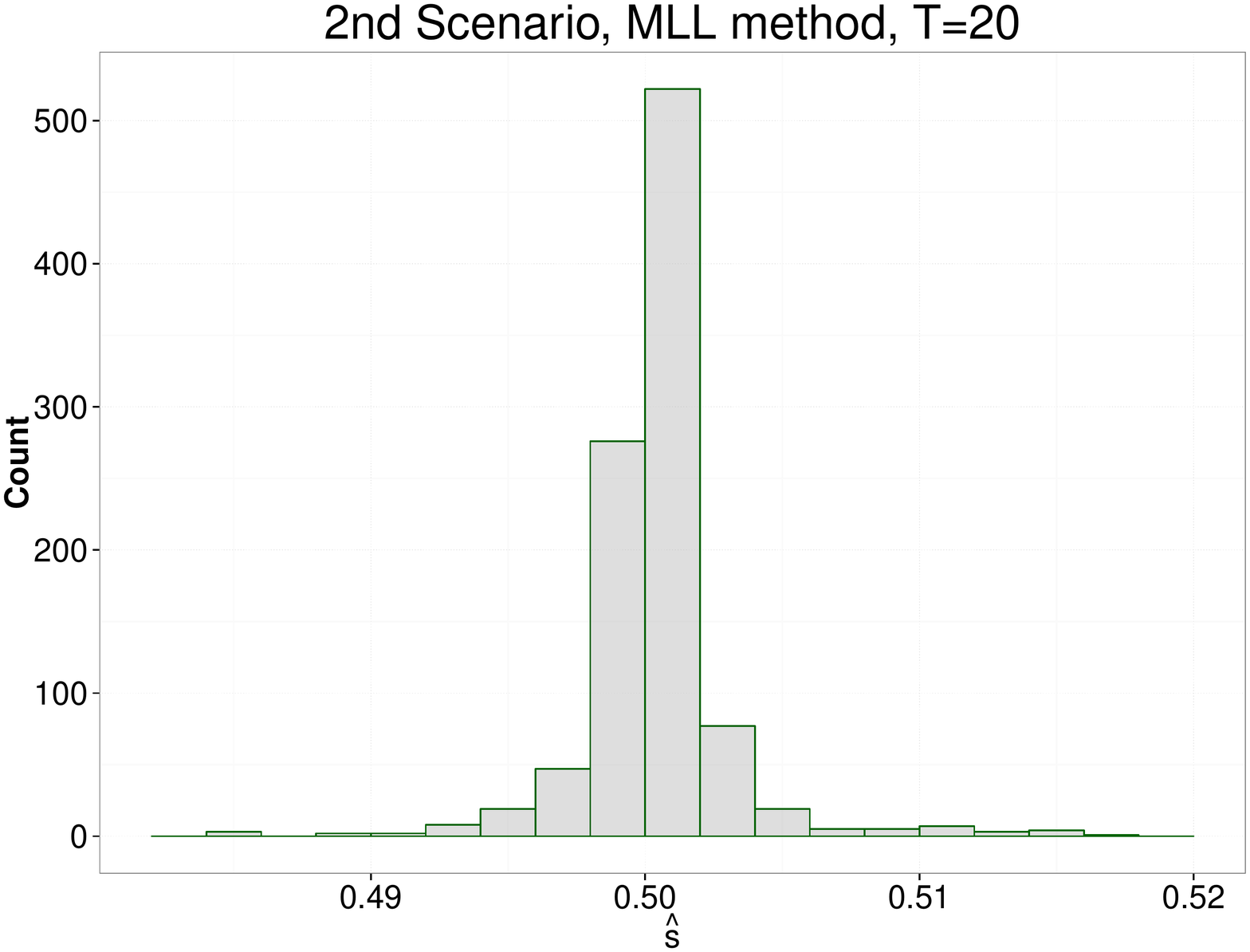}\\
\includegraphics[height=2.05in,width=2.9in]{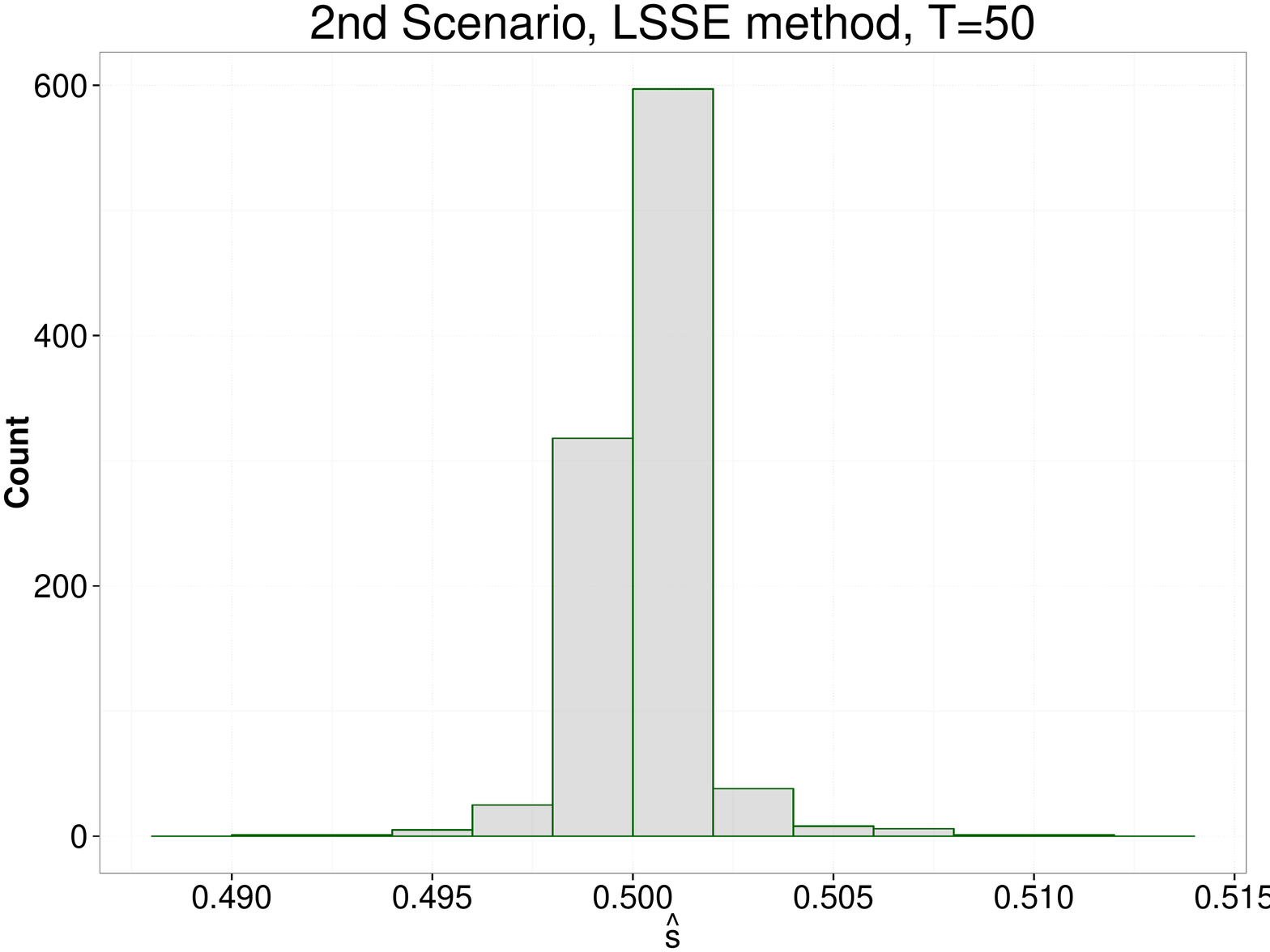}
\includegraphics[height=2.05in,width=3in]{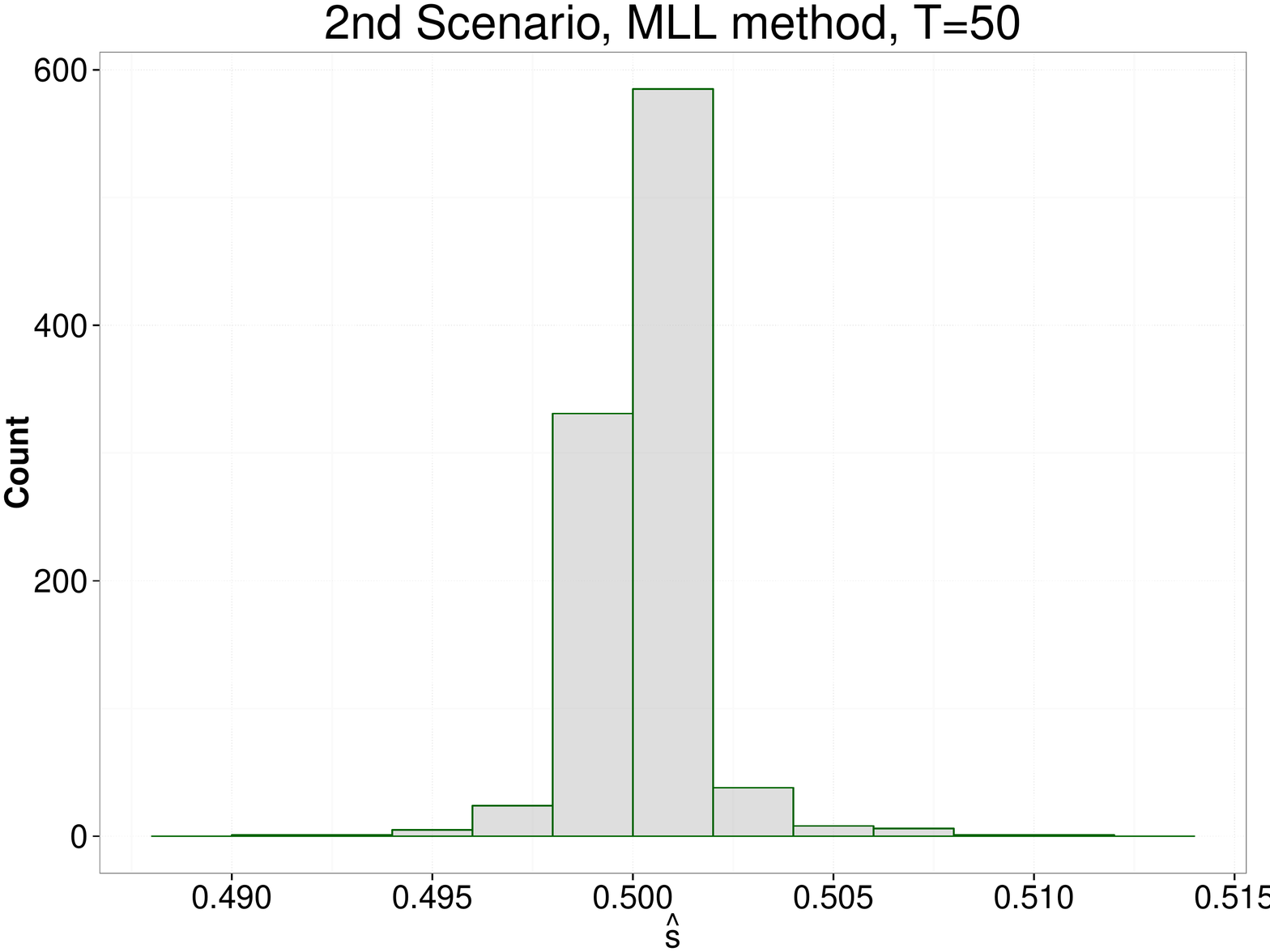}
\caption{\small Histogram of $\hat{s}$ for scenario 2,~(\ref{sim2-1})}
\label{figure2}
\end{figure}

\ \\
\noindent {\bf Estimating and testing the number of change points}\\
\noindent Here, we evaluate the performance of the proposed method by utilising the percent accuracy (PA) metric defined by
$$\mathrm{PA}(m^0)=\frac{1}{1000}\displaystyle{\sum_{i=1}^{1000}}I_{(\hat{m}_i=m^0)}\times 100\%, $$
where $\hat{m}_i$ is the estimated number of change points in the $i$th iteration. Note that $$1-\mathrm{PA}(0)=\frac{1}{1000}\displaystyle{\sum_{i=1}^{1000}}I_{\left(\mathcal{IC}(\hat{m}_i=0)\geq \mathcal{IC}(\hat{m}_i=1)\right)},$$
which is the empirical significance level. Further,
$$\mathrm{PA}(1)=\frac{1}{1000}\displaystyle{\sum_{i=1}^{1000}}I_{\left(\mathcal{IC}(\hat{m}_i=0)\geq \mathcal{IC}(\hat{m}_i=1)\right)},$$
which is the empirical power of the proposed test.\\
\ \\
\noindent As stated in the previous subsection, we aim to
assess the performance of (\ref{ic}); and for this purpose we use the criteria
$\phi(T)=\log T$ and $\phi(T)=\log (T/\Delta_t)$. Moreover, for $h(p)$,
we also consider two cases: $h(p)=p+1$ and $h(p)=p+2$ for whether or not there is a potential change in the diffusion coefficient $\sigma$ specified in (\ref{ou1}).
Consequently, we make a comparison on the basis of four penalty criteria:
$(p+1)\log T$, $(p+2)\log T$, $(p+1)\log(T/\Delta_t)$ and $(p+2)\log(T/\Delta_t)$.
The results are reported in Tables~\ref{table3}--\ref{table6} for each of the
scenarios.

\begin{table}[!htbp]
\small \caption{Empirical power of the test (in \%), under scenario 1, (\ref{sim1-1})} \centering
\begin{tabular}{|c|c|c|c|c|}
\hline
$T$&  $(p+1)\log T$ & $(p+2)\log T$ & $(p+1)\log(T/\Delta_t)$ & $(p+2)\log(T/\Delta_t)$\\
\hline
 $T$=5&100& 100 & 98.9 & 96.5\\
\hline
 $T$=10&100& 100 & 99.4 & 97.3\\
\hline
 $T$=20&100& 100 & 99.7 & 97.5\\
\hline
$T$=50&100&100& 99.6& 97.2 \\
\hline
\end{tabular}
\label{table3}
\end{table}

\begin{table}[!htbp]
\small \caption{Empirical significance level (in \%), under scenario 1, (\ref{sim1-2})} \centering
\begin{tabular}{|c|c|c|c|c|}
\hline
$T$&  $(p+1)\log T$ & $(p+2)\log T$ & $(p+1)\log(T/\Delta_t)$ & $(p+2)\log(T/\Delta_t)$\\
\hline
 $T$=5&83.1&67.6 & 1.5 & 0.1\\
\hline
 $T$=10&67.5& 46.8 &0.6 & 0\\
\hline
 $T$=20&54.4& 17.9 & 0 &0\\
\hline
$T$=50&19.8&5.2&0&0 \\
\hline
\end{tabular}
\label{table4}
\end{table}

\begin{table}[!htbp]
\small \caption{Empirical power of the test (in \%), under scenario 2, (\ref{sim2-1})} \centering
\begin{tabular}{|c|c|c|c|c|}
\hline
$T$&  $(p+1)\log T$ & $(p+2)\log T$ & $(p+1)\log(T/\Delta_t)$ & $(p+2)\log(T/\Delta_t)$\\
\hline
 $T$=5&100& 100 & 100 & 100\\
\hline
 $T$=10&100& 100 &100 & 100\\
\hline
 $T$=20&100& 100 & 100 & 100\\
\hline
$T$=50&100&100& 100& 100 \\
\hline
\end{tabular}
\label{table5}
\end{table}

\begin{table}[!htbp]
\small \caption{Empirical significance level (in \%), under scenario 2, (\ref{sim2-2})} \centering
\begin{tabular}{|c|c|c|c|c|}
\hline
$T$&  $(p+1)\log T$ & $(p+2)\log T$ & $(p+1)\log(T/\Delta_t)$ & $(p+2)\log(T/\Delta_t)$\\
\hline
 $T$=5&78.7& 55.4 & 0.4 & 0\\
\hline
 $T$=10&57.2&29.5 &0 & 0\\
\hline
 $T$=20&38.5& 13.8 & 0 & 0\\
\hline
$T$=50&10.8&2.6& 0& 0 \\
\hline
\end{tabular}
\label{table6}
\end{table}

\noindent From Tables~\ref{table1} and~\ref{table2}
as well as the plotted histograms, we see that both proposed methods
(\ref{cp1}) and (\ref{mle1})
estimate very accurately the exact rate of change point
($s^0=0.5$). In addition, as the time period $T$ increases from 10 to 50, the MSEs
of the two estimators decrease. These outcomes
confirm the theoretical findings for the asymptotic consistency
of our two proposed methods. \\
\ \\
For the estimated number $\hat{m}$ of change points,
one could see that, when there exists one change point in the model,
(\ref{ic}) gives a high empirical power
in both scenarios
with different penalty criteria and time periods; see Tables~\ref{table3} and \ref{table5}.
Within the penalty criteria employed, $\phi(T)=\log T$ provides slightly better
empirical power than that of $\phi(T)=\log (T/\Delta_t)$.
When there is no change point, Tables~\ref{table4}
and \ref{table6} reveal that the empirical significance levels, under different penalty
criteria, decrease as $T$ increases. These results also imply
that our proposed method is asymptotically consistent. \\
\ \\
Amongst the 4 penalty criteria, we observe that
when $\phi(T)=\log T$, the empirical significance level is relatively high when $T$ is small, whilst
the empirical significance level decreases when we change $h(p)$ from $p+1$ to $p+2$.
This outcome
tells us that it would be more appropriate in this case to use a penalty criterion that is
larger than $h(p)\phi(T)=(p+1)\log T$ for better estimation. On the other hand,
when using $\phi(T)=\log (T/\Delta_t)$, the performance is significantly
improved compared to that of $\phi(T)=\log T$.
In both scenarios, one could see that when the time period is small ($T$=5 and $T$=10), the empirical significance level of the proposed method is relatively high when using $\phi(T)=\log T$,
but decreases to almost 0\% when $\phi(T)=\log (T/\Delta_t)$. \\
\ \\
Overall, based
on the results in Tables~\ref{table3}--\ref{table6} for different cases,
we find that for a bigger $T$, $\hat{m}$ obtained via (\ref{ic}) under four different penalty criteria all
perform
consistently in estimating the number of change points.
However, when $T$ is small,  the performance based on the
criterion $h(p)\phi(T)=(p+1)\log(T/\Delta_t)$ is efficient
and stable vis-\'{a}-vis the other criteria in each case. This suggests that
$h(p)\phi(T)=(p+1)\log(T/\Delta_t)$ is appropriate for this simulation study.

\subsection{Implementation on observed financial market data with discussion}
\noindent We apply the estimation methods (\ref{cp1}), (\ref{mle1}) and (\ref{ic}) to two
different financial market data series. For each series, we fit the process with
the following two different mean-reverting OU processes with one change point.
\begin{eqnarray}
\textit{d}X_t&=&
\begin{cases}
(\mu_1^{(1)}-a^{(1)}X_t)\textit{d}t+\sigma \textit{d}W_t,
& \mbox{if }0<t<s T \\
(\mu_1^{(2)}-a^{(2)}X_t)\textit{d}t+\sigma \textit{d}W_t,
& \mbox{if }s T<t< T,
\end{cases},\label{realdata1-1}\\
\textit{d}X_t&=&
\begin{cases}
(\mu_1^{(1)}+\sqrt{2}\cos(\frac{\pi t}{2\Delta_t})\mu_2^{(1)}-a^{(1)}X_t)\textit{d}t+\sigma \textit{d}W_t,
& \mbox{if }0<t<s T \\
(\mu_1^{(2)}+\sqrt{2}\cos(\frac{\pi t}{2\Delta_t})\mu_2^{(2)}-a^{(1)}X_t)\textit{d}t+\sigma \textit{d}W_t,
& \mbox{if }s T<t< T,
\end{cases}.\label{realdata2-1}
\end{eqnarray}
where $X_t$ is the target of interest (i.e., spot price, log-transformed spot price, daily return, etc.) at time $t$. The no-change point versions of \ref{realdata1-1} and \ref{realdata1-1} are\\
\begin{eqnarray}
\textit{d}X_t&=&
(\mu_1-aX_t)\textit{d}t+\sigma \textit{d}W_t, \quad \mbox{if }0<t< T,
\label{realdata1-2}\\
\textit{d}X_t&=&
\left( \mu_1+\sqrt{2}\cos\left( \frac{\pi t}{2\Delta_t} \right) \mu_2-a X_t \right)\textit{d}t+\sigma \textit{d}W_t, \quad \mbox{if }0<t< T.\label{realdata2-2}
\end{eqnarray}
\ \\
\noindent For (\ref{realdata1-1}) and (\ref{realdata2-1}),
we apply (\ref{cp1}) and (\ref{mle1})
to estimate the unknown change point, whilst for (\ref{realdata1-2}) and  (\ref{realdata2-2}),
we train the MLE of drift parameters based on the entire time period.
Then, we use (\ref{ic}) to test the existence of a change point.
In our formulation $\sigma$ is assumed to remain unchanged for the entire time period,
and thus it may estimated using the data's realised volatility, i.e.,
$\hat{\sigma}=(\sum_{t_i\in[0,T]}(X_{t_{i+1}}-X_{t_i})^2/T)^{1/2}$. Alternatively, one may fit the data series to the model, take the standard error of the residuals and then divide it by $\sqrt{\Delta_t}$ (see Smith (2010)). \\
\ \\
\noindent When $\sigma$ is time-dependent, EWMA and GARCH-type volatility estimation methods
may be used. However, under this situation the MLE and related asymptotic
properties established in Dehling, et al. (2010) as well as the asymptotic properties
derived in this paper may need re-evaluation as they are all based on the time-independent
assumption of the diffusion coefficient $\sigma$. In this paper, we consider $\sigma$ to be time-independent
and after the estimated results are obtained, the most suitable model is chosen as the one yielding the least
SIC value. In addition, we also report the log likelihood for comparison. To this end,
we let $\mathcal{LL}_0$ and $\mathcal{LL}_1$ denote the log likelihood under $H_0$ and $H_1$, respectively.
\\

\subsubsection{Application to West Texas Intermediate Cushing crude oil spot price}
\noindent The first data set is the time series of West Texas Intermediate Cushing crude oil spot prices, which
was first described in Subsection~\ref{wti}. Our interest in this data set is justified by the fact that,
as mentioned in Subsection~\ref{wti}, it is often being considered as a benchmark in oil pricing.
Furthermore, the modelling of and point-change detection in commodity prices are important in the
valuation of commodity derivatives and risk management of portfolios with large commodity holdings.
\\
\ \\
\noindent For our preliminary attempt, we set the WTI crude oil spot price to be our target of interest. To fit the model, we choose $T=4$ and so $\Delta_t=4/1008$.
With the data fitted to the two candidate models using the proposed methods, we display the results shown in Table~\ref{table7.1} and Figure~\ref{fig:realdata1.1}.\\

\begin{table}[!htbp]
\small \caption{Change-point detection results for the WTI crude oil prices} \centering
\begin{tabular}{|c|c|c|c|c|c|c|c|}
\hline
Model&  LSSE method & MLL method & $\hat{m}$ &$\mathcal{LL}_1$& $\mathcal{LL}_0$ & $\mathcal{IC}(m=1)$ & $\mathcal{IC} (m=0)$ \\
\hline
(\ref{realdata1-1})& 2014-09-26 & 2014-09-26 & 1 &11.89&0.72&3.87&12.39\\
\hline
(\ref{realdata2-1})& 2014-09-26 & 2014-09-26 & 1&9.96&3.99&9.96&12.76\\
\hline
\end{tabular}
\label{table7.1}
\end{table}

\begin{figure}[htbp]
\includegraphics[height=2.3in,width=6in]{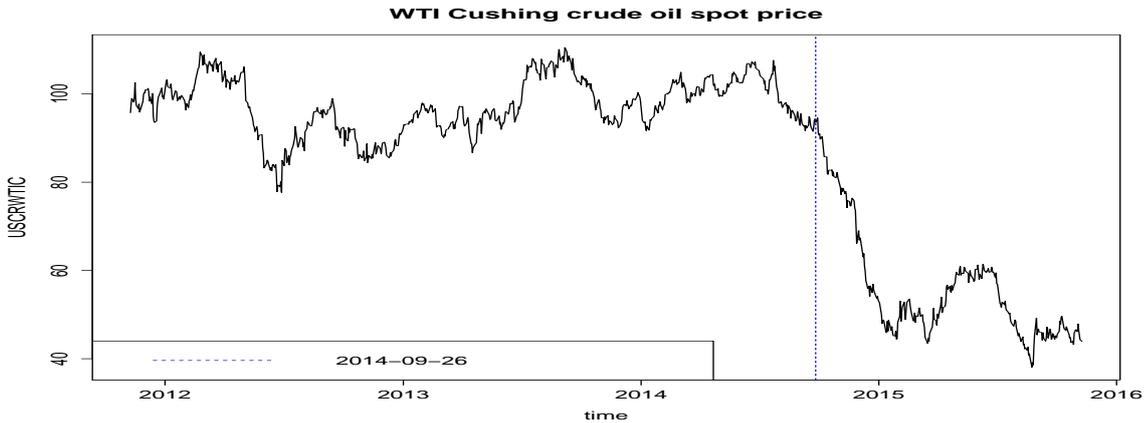}
\caption{\small WTI Cushing crude oil spot prices (09 November 2011--09 November 2015)}
\label{fig:realdata1.1}
\end{figure}
\noindent Looking at Table~\ref{table7.1} and Figure~\ref{fig:realdata1.1}, one can see that the value
of the log likelihood increases as the number of the coefficients in
the model increases. Further, both candidate models confirm the existence of a change point ($\hat{m}=1$)
during this time period. Under both models, the detected change point is the same (i.e., 26 September 26 2014)
for the proposed methods. On the other hand, despite the log-likelihood comparison showing that
the periodic mean-reverting model (\ref{realdata2-1}) produces higher log likelihood than the classical OU process with change point (\ref{realdata1-1}), the SIC, nonetheless, suggests that (\ref{realdata1-1}) is more appropriate than (\ref{realdata2-1}) for this data series.\\
\ \\
As suggested in Chen (2010), we also analyse the log-transformed WTI crude oil spot prices. We examine the
log-transformed WTI crude oil spot prices as our target of interest and re-apply the proposed techniques.
The results are shown in Table~\ref{table7.2} and Figure~\ref{fig:realdata1.2}. It is worth noting that the
detected change points for this log-transformed spot prices are still the same (i.e., 26 September 2014), as well as the behavior of log-likelihood (increase as the number of coefficients in the model increases); although, this time around, the results based on (\ref{realdata2-1}) fail to pass the test on the existence of a change point ($\hat{m}=0$). Judging from the SIC numbers, model (\ref{realdata1-1}) is again better than model (\ref{realdata2-1}). This suggests that (\ref{realdata1-1}), with the change point occurring on 26 September 2014, is more suitable for both the WTI Cushing crude oil price data and its log-transformed series.
\ \\
\begin{table}[!htbp]
\small \caption{Change point detection for the log-transformed WTI crude oil prices} \centering
\begin{tabular}{|c|c|c|c|c|c|c|c|}
\hline
Model&  LSSE method & MLL method & $\hat{m}$ & $\mathcal{LL}_1$&$\mathcal{LL}_0$ &$\mathcal{IC}(m=1)$& $\mathcal{IC}(m=0)$  \\
\hline
(\ref{realdata1-1})& 2014-09-26 & 2014-09-26 & 1 &8.15&0.82&11.37&12.18\\
\hline
(\ref{realdata2-1})& 2014-09-26 & 2014-09-26 & 0&13.68&4.26&14.12&12.20 \\
\hline
\end{tabular}
\label{table7.2}
\end{table}

\begin{figure}[htbp]
\includegraphics[height=2.3in,width=6in]{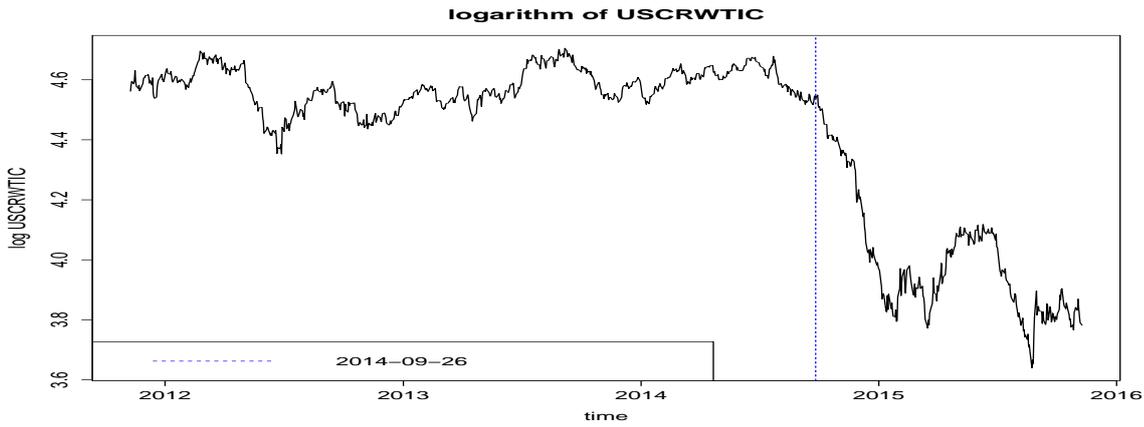}
\caption{\small Log-transformed WTI Cushing crude oil spot prices (09 November 2011--09 November 2015)}
\label{fig:realdata1.2}
\end{figure}
\ \\
\noindent  Based on our \textit{primary statistics of interest}, defined here as the estimated change point $\hat{\tau}$ and the associated MLE $\hat{\theta}^{(1)}$ and $\hat{\theta}^{(2)}$) obtained from the proposed methods, we use model (\ref{realdata1-1}) to generate a simulated crude oil price data set and log-transformed crude oil price data series. Additionally, we also generate a simulated data series based on (\ref{realdata1-2}) for comparison. The two simulated series are presented side-by-side in Figure~\ref{fig:simwti1}. \\
\begin{figure}[htpb]
\includegraphics[height=2.2in,width=3in]{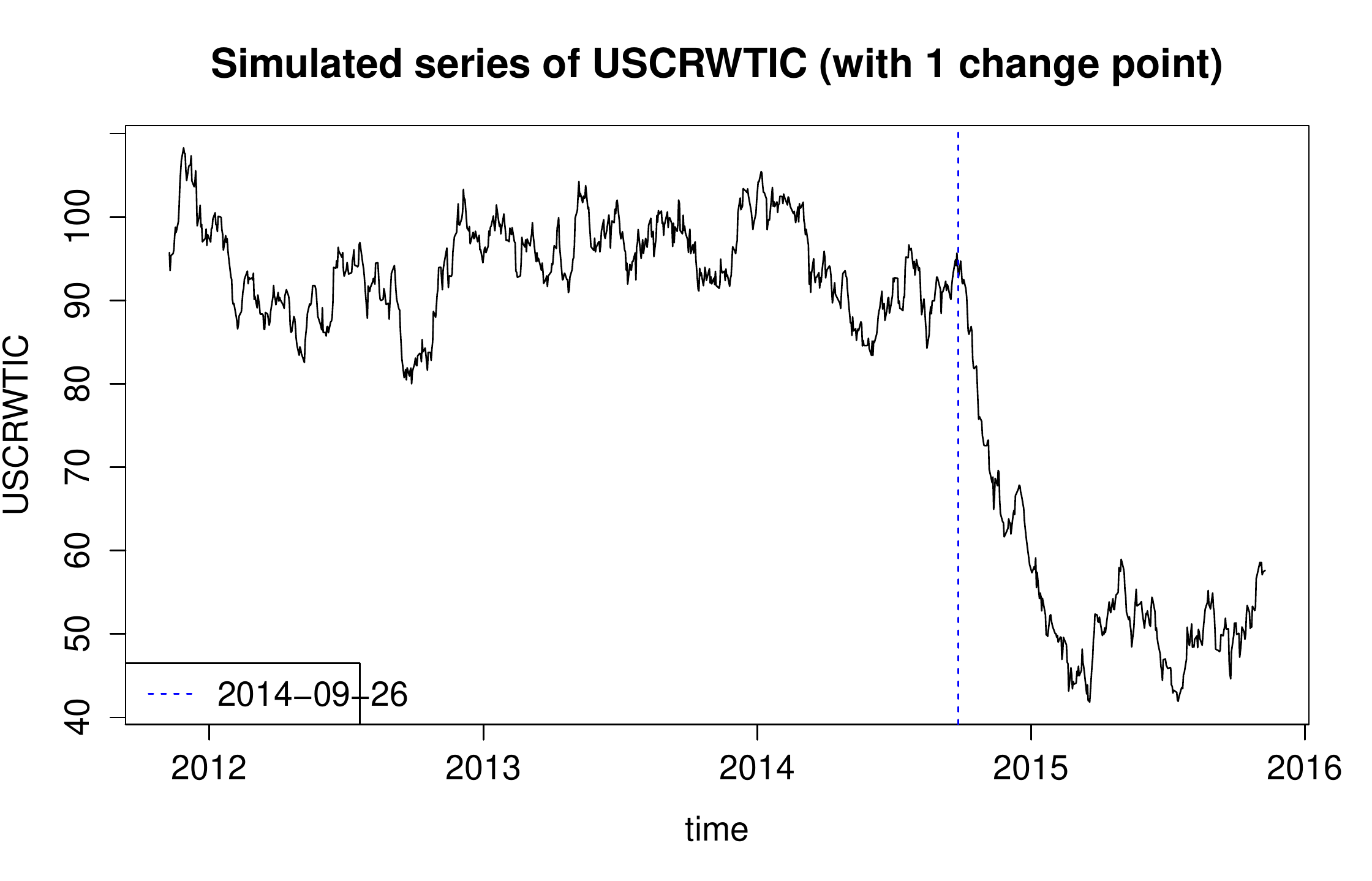}
\includegraphics[height=2.2in,width=3in]{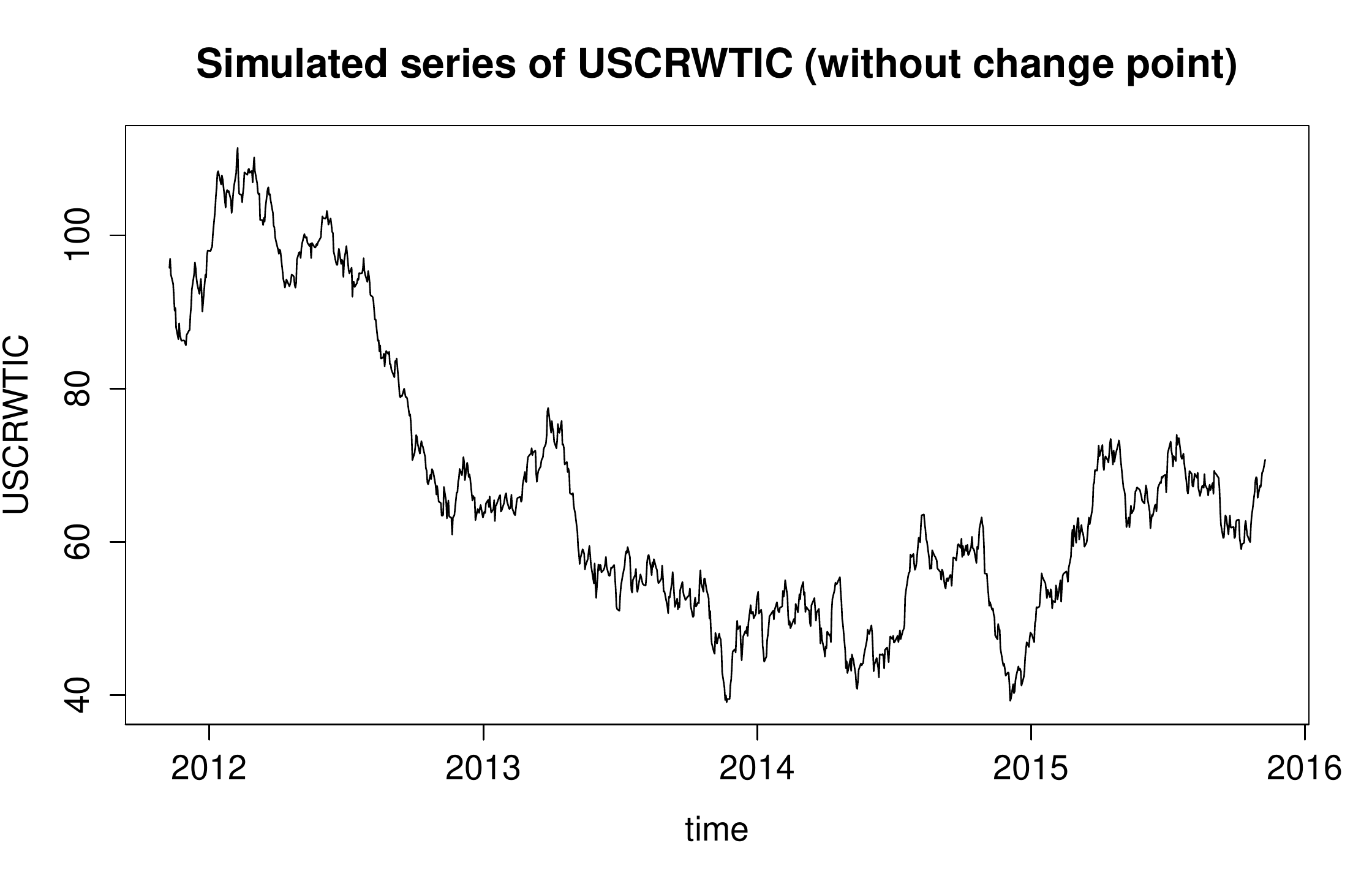}
\caption{\small Simulated series of WTI Cushing crude oil spot price (09 November 2011--09 November 2015)}
\label{fig:simwti1}
\end{figure}

\begin{figure}[htpb]
\includegraphics[height=2.3in,width=3in]{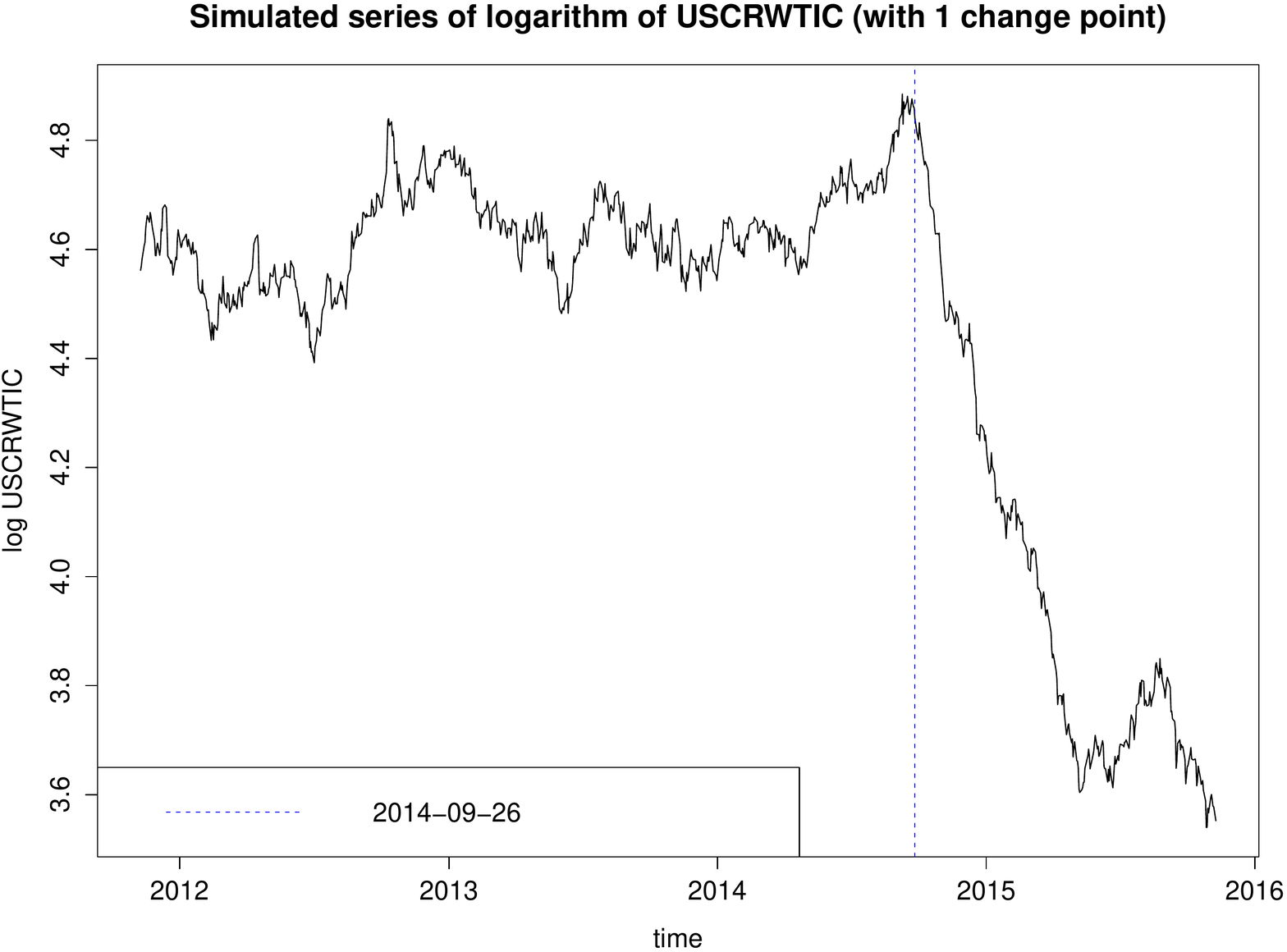}
\includegraphics[height=2.3in,width=3in]{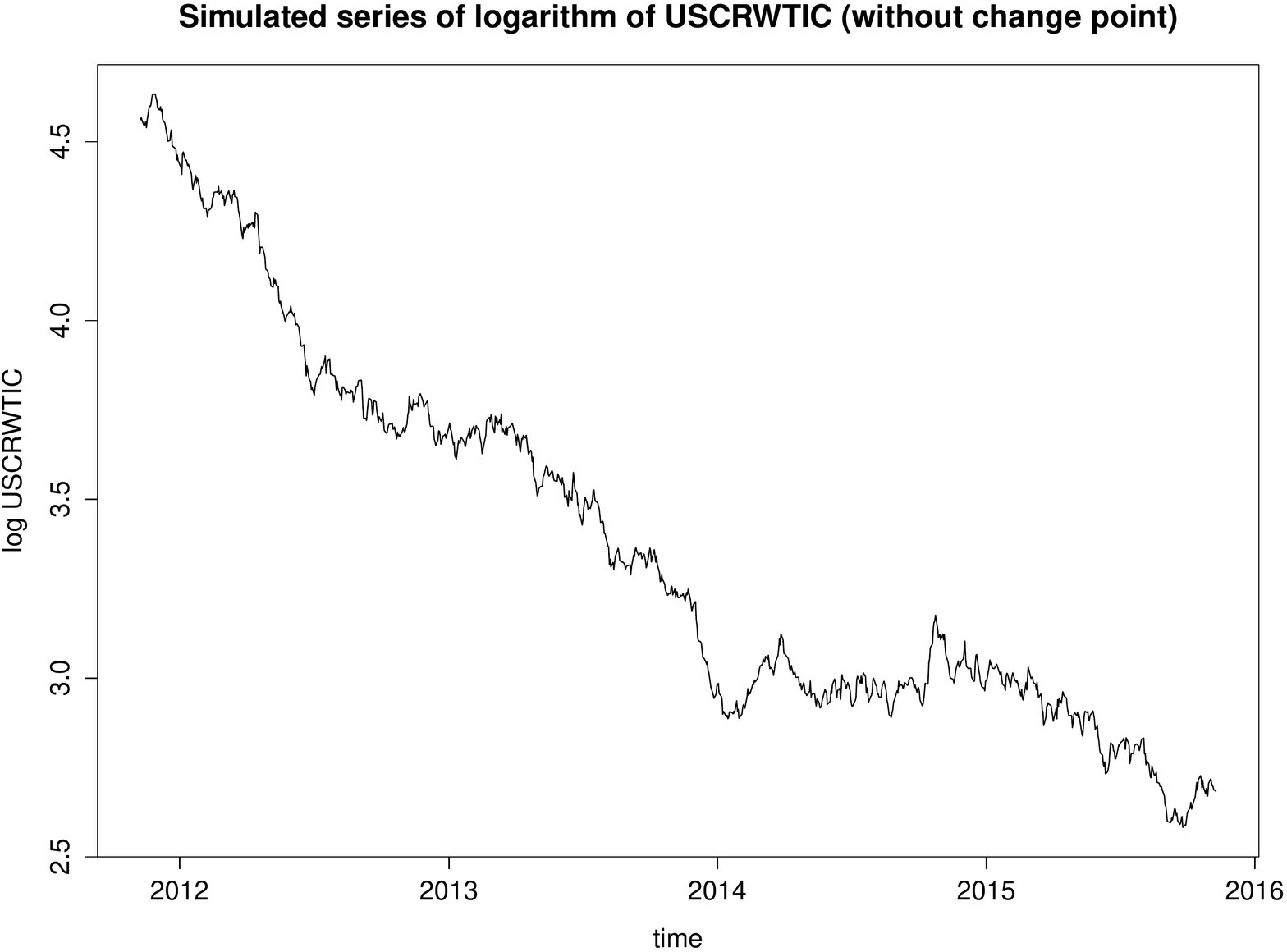}
\caption{\small Simulated series of log-transformed WTI Cushing crude oil spot price (09 November 2011--09 November 2015)}
\label{fig:simwti2}
\end{figure}
\ \\
\noindent By comparing the simulated series in Figures~\ref{fig:simwti1}
and \ref{fig:simwti2} to the original series (shown in  Figures~\ref{fig:realdata1.1}
and \ref{fig:realdata1.2}), we see that the simulated series generated by (\ref{realdata1-1}),
with the change point occurring on 26 September 2014, is closer to the original series than
the one generated by (\ref{realdata1-2}). This confirms the efficiency of the proposed methods for this series.

\subsubsection{Application to XAU currency}
\noindent We apply as well the proposed methods to the XAU currency data,
which was described in Subsection~\ref{xau}. This data set refers to the time-series of
prices, in US dollars, for a troy ounce of gold. Our interest in this kind of data set
is motivated by the pricing of currency swaps, futures and options. \\
\ \\
\noindent Figure~\ref{fig:realdata2.1} shows that the trend of XAU currency series changed over time. As in the previous study, we take both the XAU currency and its logarithm as our targets of interest. To fit the data with the candidate models, we choose $T=15$ and $\Delta_t=15/3913$. The results for the XAU currency are depicted in Table~\ref{table8.1} and Figure~\ref{fig:realdata2.1}, whilst the results for the log-transformed XAU currency are shown in Table~\ref{table8.2} and Figure~\ref{fig:realdata2.2}.\\
\ \\
\begin{table}[!htbp]
\small \caption{Change-point detection for XAU currency (15 years)} \centering
\begin{tabular}{|c|c|c|c|c|c|c|c|}
\hline
Model&  LSSE method & MLL method & $\hat{m}$ & $\mathcal{LL}_1$&$\mathcal{LL}_0$ &$\mathcal{IC}(m=1)$& $\mathcal{IC}(m=0)$ \\
\hline
(\ref{realdata1-1})& 2013-04-08 & 2013-04-08 & 1 &12.62&1.43&7.83 &13.67\\
\hline
(\ref{realdata2-1})& 2013-04-08 &2013-04-08 & 0&13.31&1.81&23.01&21.18\\
\hline
\end{tabular}
\label{table8.1}
\end{table}
\begin{figure}[htbp]
\includegraphics[height=2.2in,width=6in]{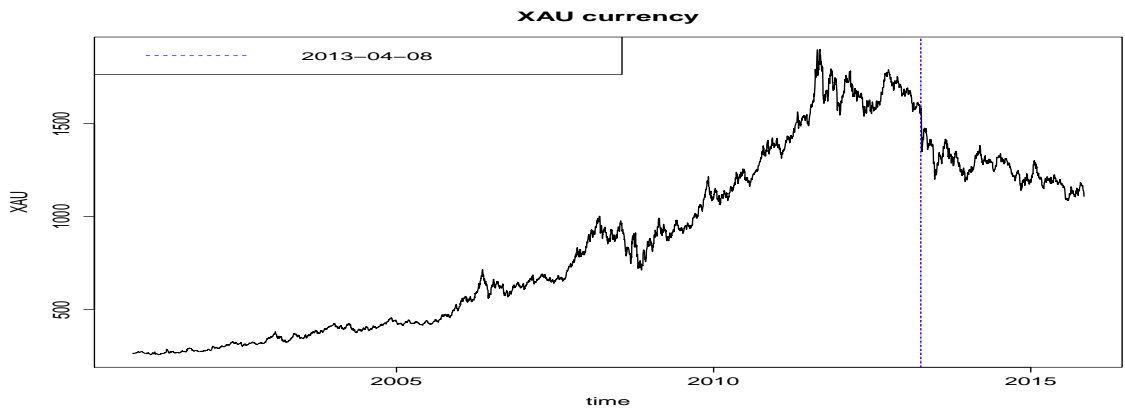}
\caption{\small Evolution of the XAU currency (03 November 2000--04 November 2015)}
\label{fig:realdata2.1}
\end{figure}

\begin{table}[!htbp]
\small \caption{Change point detection for the log-transformed XAU currency (15 years)} \centering
\begin{tabular}{|c|c|c|c|c|c|c|c|}
\hline
Model&  LSSE method & MLL method & $\hat{m}$ & $\mathcal{LL}_1$&$\mathcal{LL}_0$ & $\mathcal{IC}(m=1)$& $\mathcal{IC}(m=0)$\\
\hline
(\ref{realdata1-1})& 2011-08-19 & 2011-08-19 & 0 &7.08&3.25 &18.92&10.04\\
\hline
(\ref{realdata2-1})& 2011-08-19 &2011-08-19 & 0&7.41&3.46&34.80&17.89\\
\hline
\end{tabular}
\label{table8.2}
\end{table}
\begin{figure}[htbp]
\includegraphics[height=2.2in,width=6in]{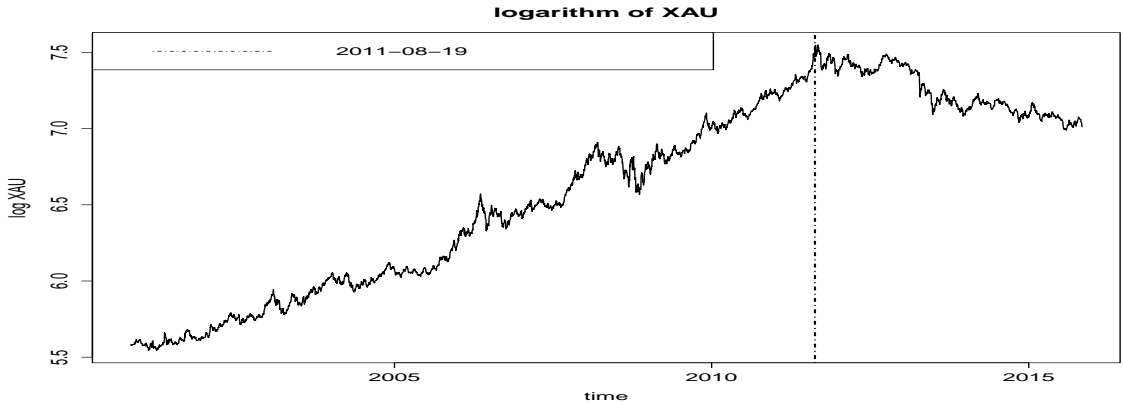}
\caption{\small Log-transformed XAU currency (03 November 2000--04 November 2015)}
\label{fig:realdata2.2}
\end{figure}
\noindent For the original XAU currency, one can see that log-likelihood increases as the number of coefficients in the model increases. Further, (\ref{realdata1-1}) successfully detects one change point on 08 April 2013, whilst (\ref{realdata2-1}) fails to detect the change point ($\hat{m}=0$). However, the SIC comparison shows that the classical OU process (SIC$=7.83$) is more appropriate than the periodic mean-reverting model (\ref{realdata2-1}) (SIC$=21.18$); we, therefore, select (\ref{realdata1-1}) with the change point on 08 April 2013
as the suitable model for this series.\\
\ \\
\noindent For the log-transformed XAU currency, Figure~\ref{fig:realdata2.2} illustrates
visually that the series is smoother than the original series (see Figure~\ref{fig:realdata2.1}),
and the potential change in the series becomes less clear.
Applying the proposed methods, both (\ref{realdata1-1}) and (\ref{realdata2-1})
detect the same change point on 19 August 2011, which is the time when the log-transformed XAU
currency almost reaches the highest value. Although the comparison of log-likelihood
functions indicates that  imposing a change point in the model
can produce higher log-likelihood and (\ref{realdata1-1}) is still more suitable than
(\ref{realdata2-1}) in terms of the SIC, both models fail to pass the test for the existence
of change point ($\hat{m}=0$). These results suggest that for this log-transformed data series,
imposing a change point into the model is not as efficient as compared to the original XAU currency series. \\
\ \\
\noindent Similar to the previous study,
we also generate, employing the estimated values of the primary statistics of interest, some simulated series based on (\ref{realdata1-1}) and (\ref{realdata2-1}) for both XAU currency and its log transform. The results are given in Figure~\ref{fig:simxau15.1} for the original XAU currency and \ref{fig:simxau15.2} for the log-transformed XAU currency, respectively.\\
 \ \\
\begin{figure}[htpb]
\includegraphics[height=2.2in,width=3in]{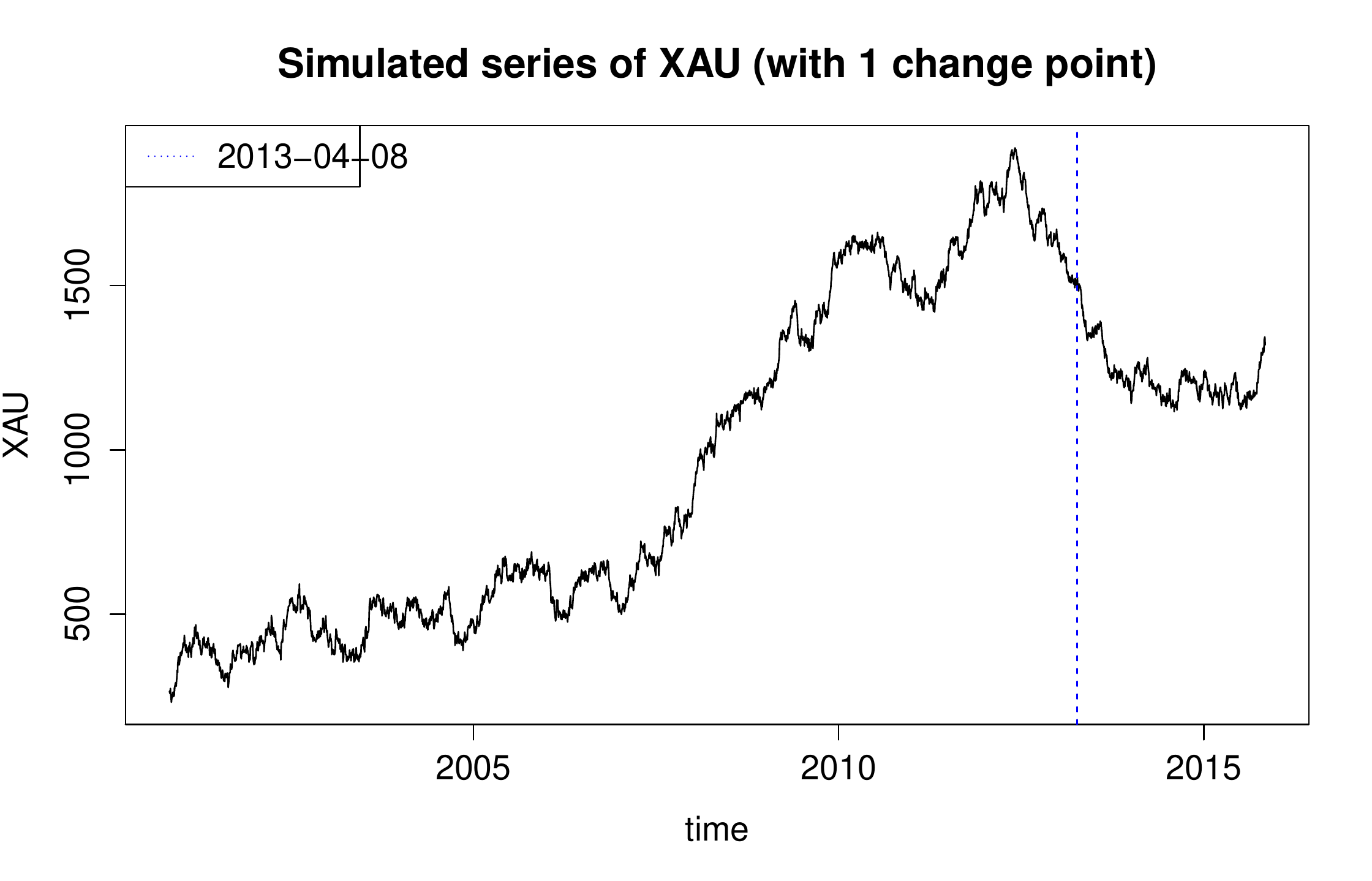}
\includegraphics[height=2.2in,width=3in]{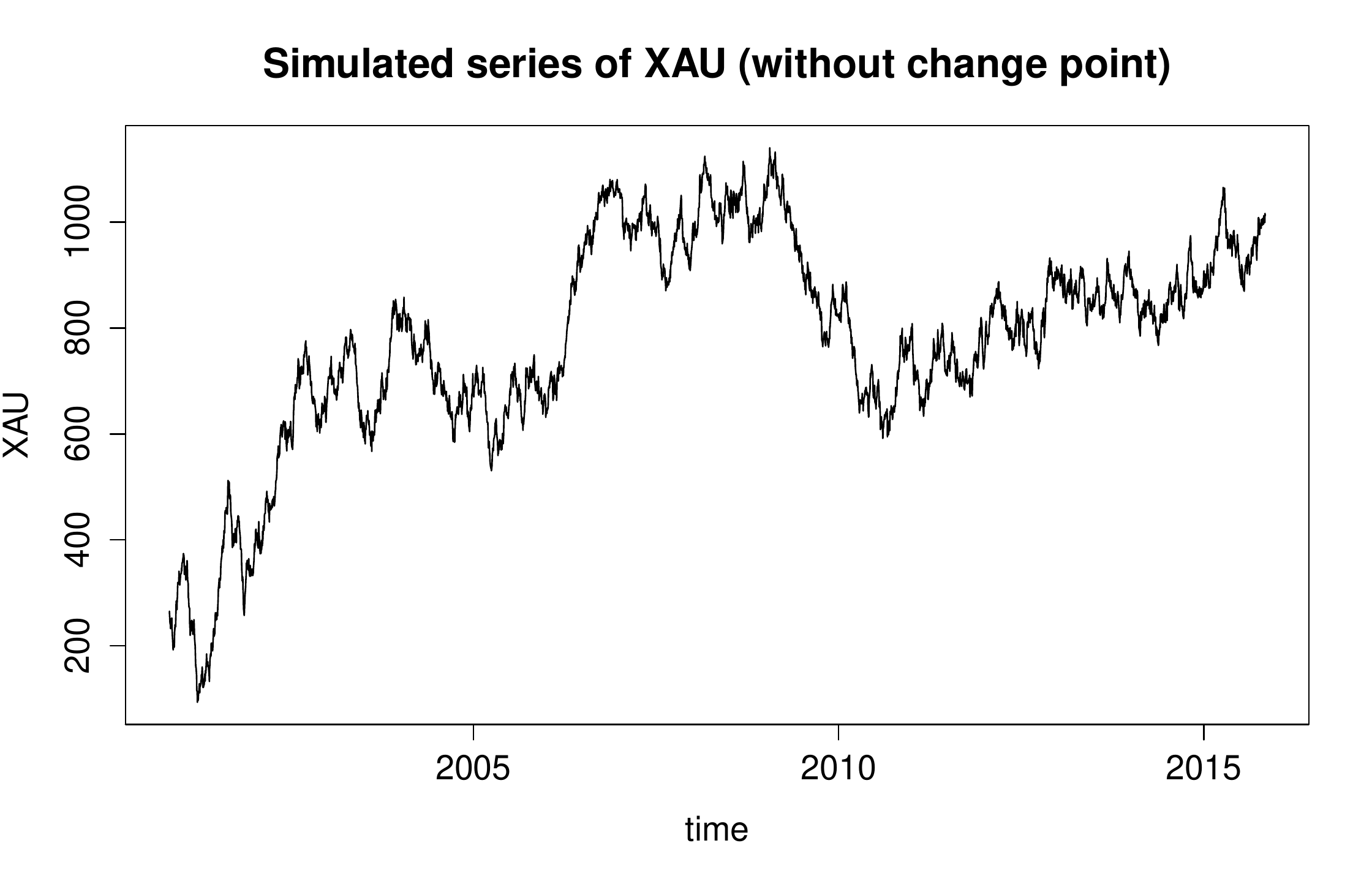}
\caption{\small Simulated series of XAU currency (03 November 2000--04 November 2015)}
\label{fig:simxau15.1}
\end{figure}
\begin{figure}[htpb]
\includegraphics[height=2.2in,width=3in]{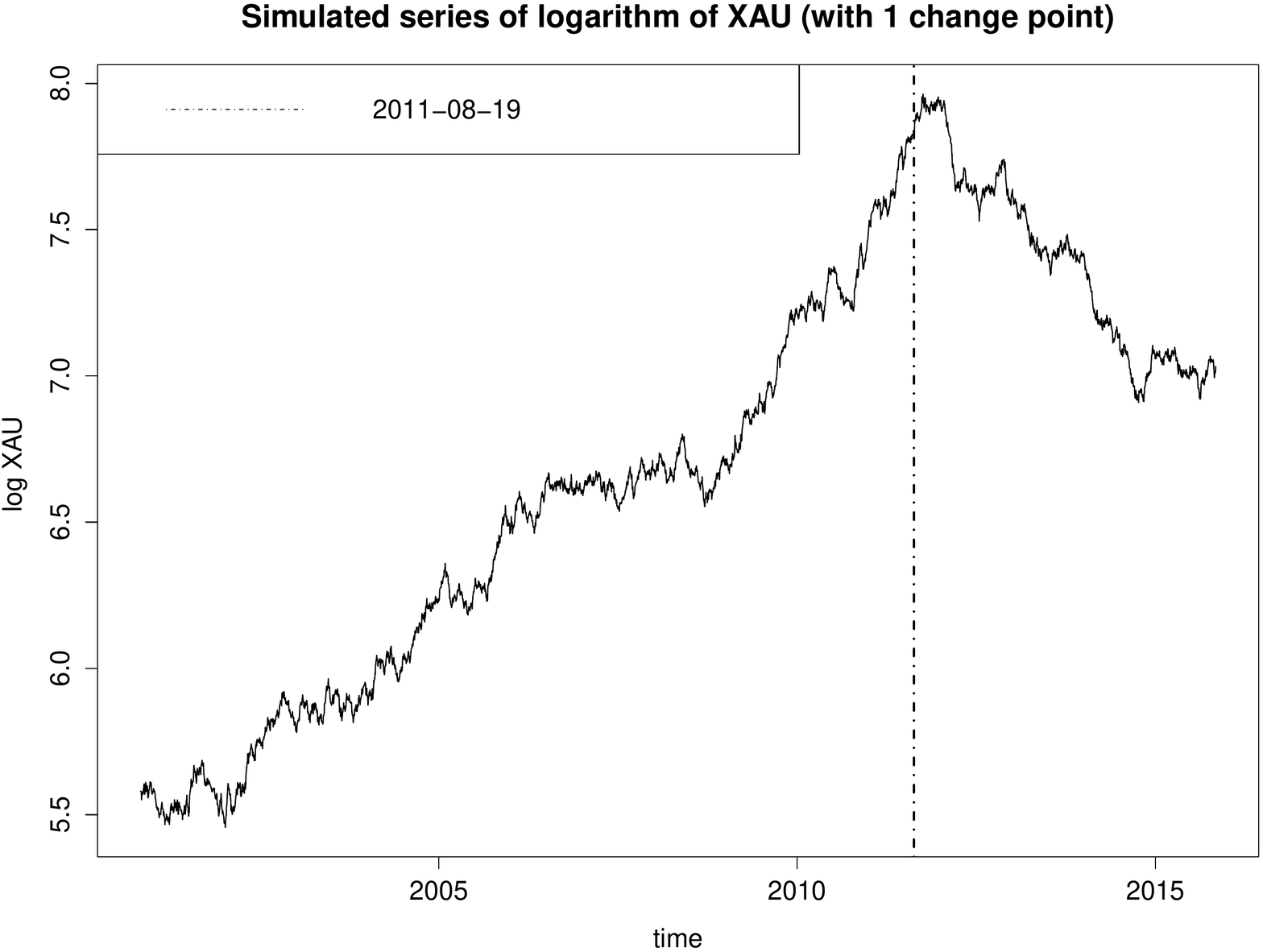}
\includegraphics[height=2.2in,width=3in]{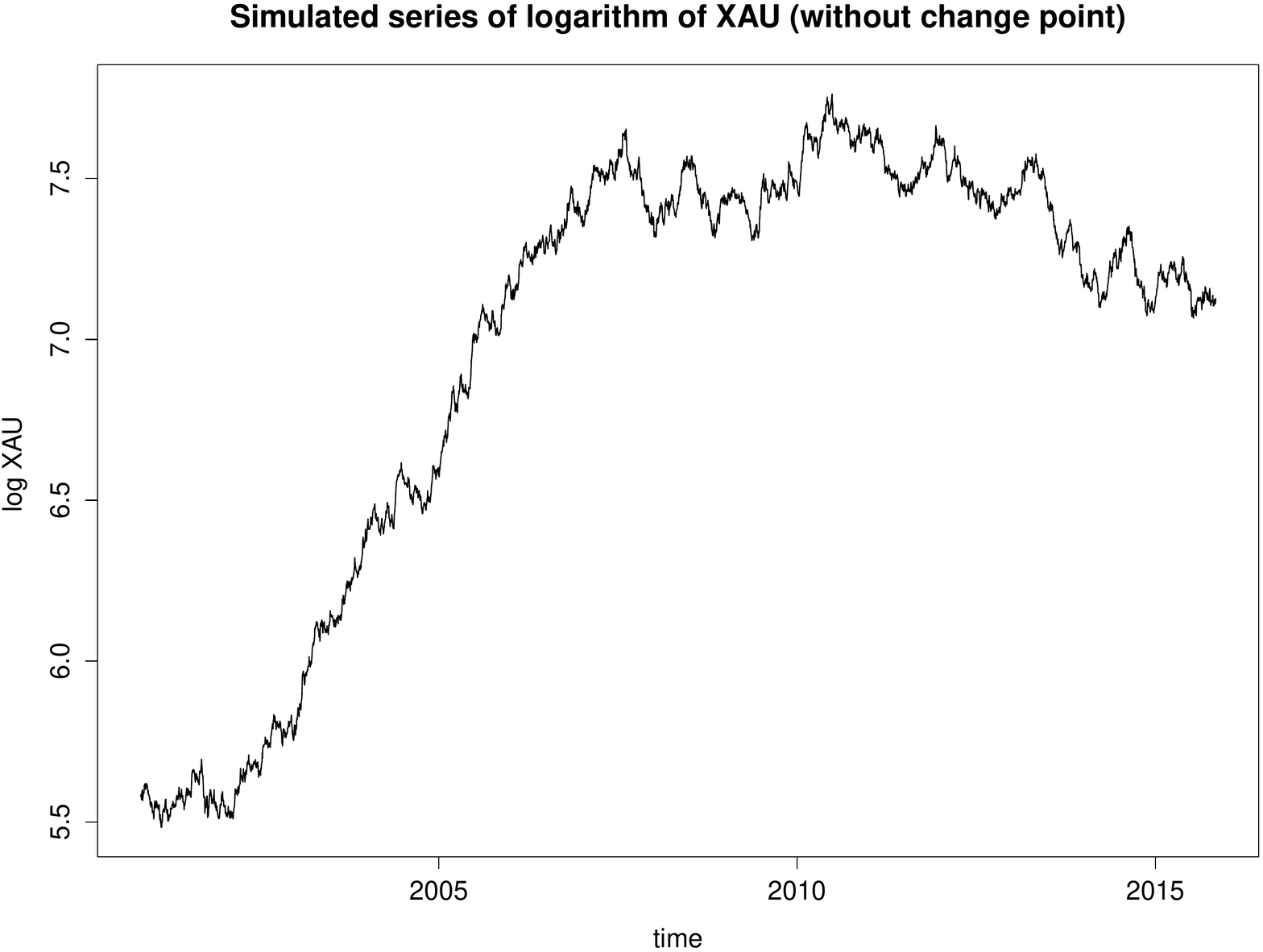}
\caption{\small Simulated series of log-transformed XAU currency (03 November 2000--04 November 2015)}
\label{fig:simxau15.2}
\end{figure}
\noindent For the original XAU currency, the comparison shows that the simulated series generated by (\ref{realdata1-1}) with the change point on 08 April 2013 is closer to the original series, especially around the location where the change point happens; this is in contrast to the model without a change point. This result confirms the efficiency of employing the change-point process (\ref{realdata1-1}) in improving the estimation of this series. However, the change in the log-transformed series is not clear as the difference between the two simulated series becomes small; they are both close to the original log-transformed series. This means that there may be no need to impose a change point for this log-transformed series; this is because the improvement is not as significant to the one obtained in the original series.\\
\ \\
\noindent To probe matters on this series further, we change the starting date of the data series to 04 November 2008. This reduces the time period from 15 years to 7 years, and one can see from Figure~\ref{fig:realdata3.1} that the XAU currency levels in this time period are all higher than \$700. We re-apply the methods to the XAU currency data and its log-transformed series. Table~\ref{table9.1} and Figure~\ref{fig:realdata3.1} show the results for the original XAU currency, and Table~\ref{table9.2} and Figure~\ref{fig:realdata3.2} display
the results
for the log-transformed XAU currency. \\
\\
\begin{table}[htbp]
\small \caption{Change-point detection for XAU currency (7 years)} \centering
\begin{tabular}{|c|c|c|c|c|c|c|c|}
\hline
Model&  LSSE method & MLL method & $\hat{m}$ &$\mathcal{LL}_1$&$\mathcal{LL}_0$ &$\mathcal{IC}(m=1)$& $\mathcal{IC}(m=0)$\\
\hline
(\ref{realdata1-1})& 2013-04-08 & 2013-04-08 & 0 &9.26&2.28&11.50&10.47\\
\hline
(\ref{realdata2-1})& 2013-04-08 &2013-04-08 & 0&9.62&2.56&25.80&17.40\\
\hline
\end{tabular}
\label{table9.1}
\end{table}
\begin{figure}[htbp]
\includegraphics[height=3.2in,width=6in]{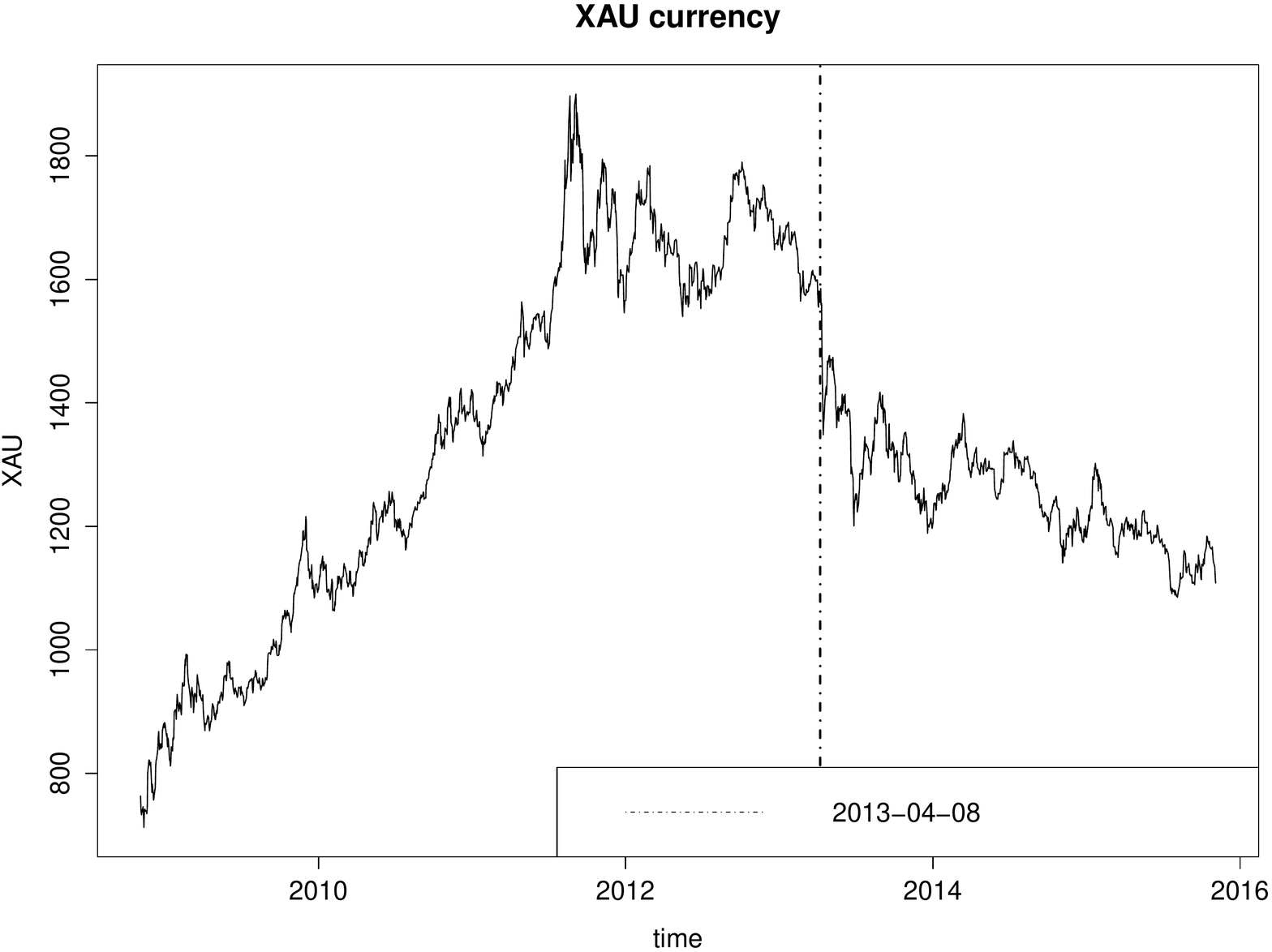}
\caption{\small XAU currency (04 November 2008--04 November 2015)}
\label{fig:realdata3.1}
\end{figure}
\ \\
\begin{table}[htbp]
\small \caption{Change-point detection for the log-transformed XAU currency (7 years)} \centering
\begin{tabular}{|c|c|c|c|c|c|c|c|}
\hline
Model&  LSSE method & MLL method & $\hat{m}$ &$\mathcal{LL}_1$&$\mathcal{LL}_0$& $\mathcal{IC}(m=1)$& $\mathcal{IC}(m=0)$\\
\hline
(\ref{realdata1-1})&2011-08-19 & 2011-08-19  & 0 &6.06&2.95&17.90&9.11\\
\hline
(\ref{realdata2-1})& 2009-02-23 &2009-02-23 & 0&7.06&3.31&30.92&15.89\\
\hline
\end{tabular}
\label{table9.2}
\end{table}
\begin{figure}[htbp]
\includegraphics[height=3.2in,width=6in]{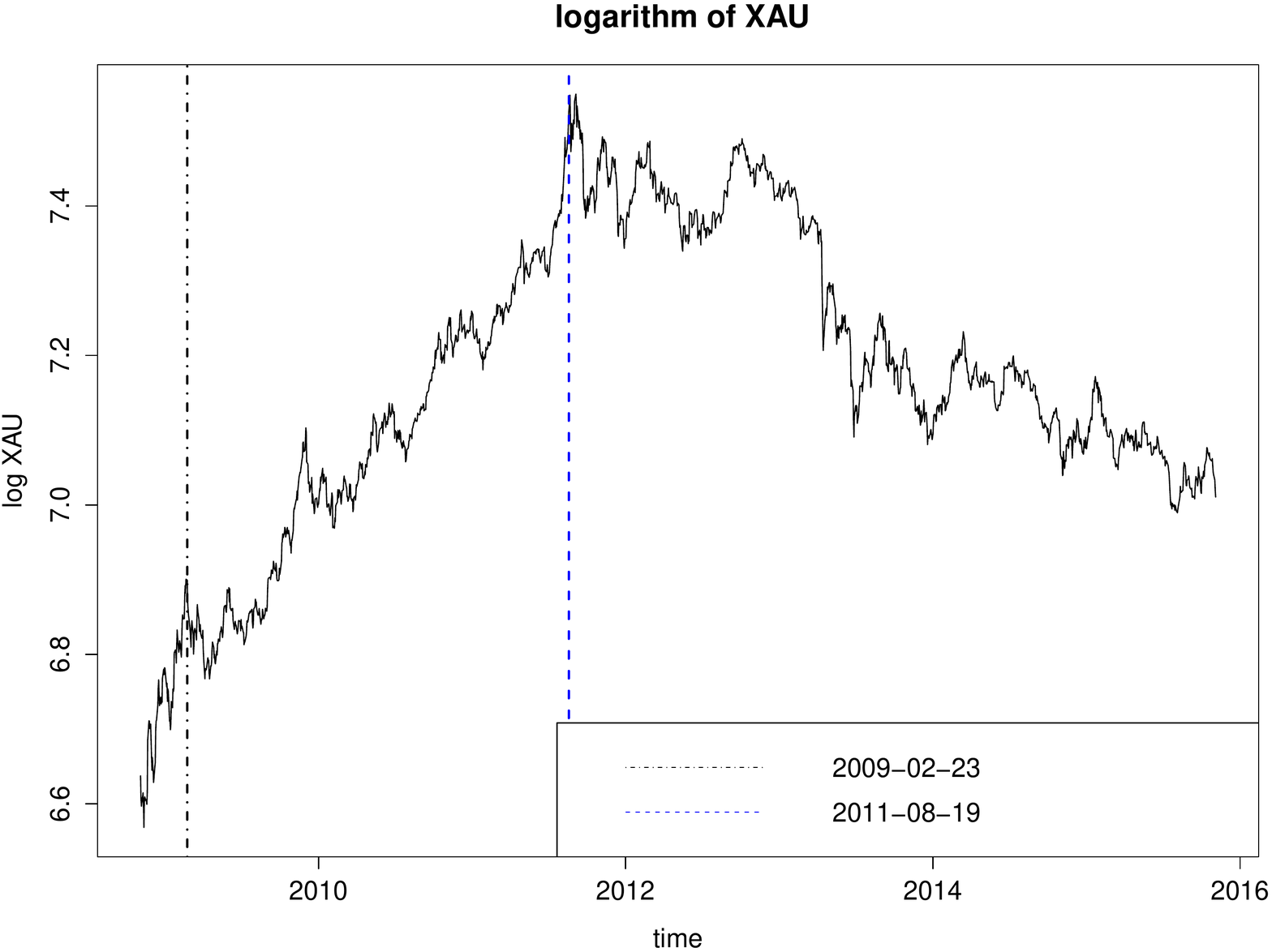}
\caption{\small XAU currency (04 November 2008--04 November 2015)}
\label{fig:realdata3.2}
\end{figure}
\begin{figure}[htpb]
\includegraphics[height=2.2in,width=3in]{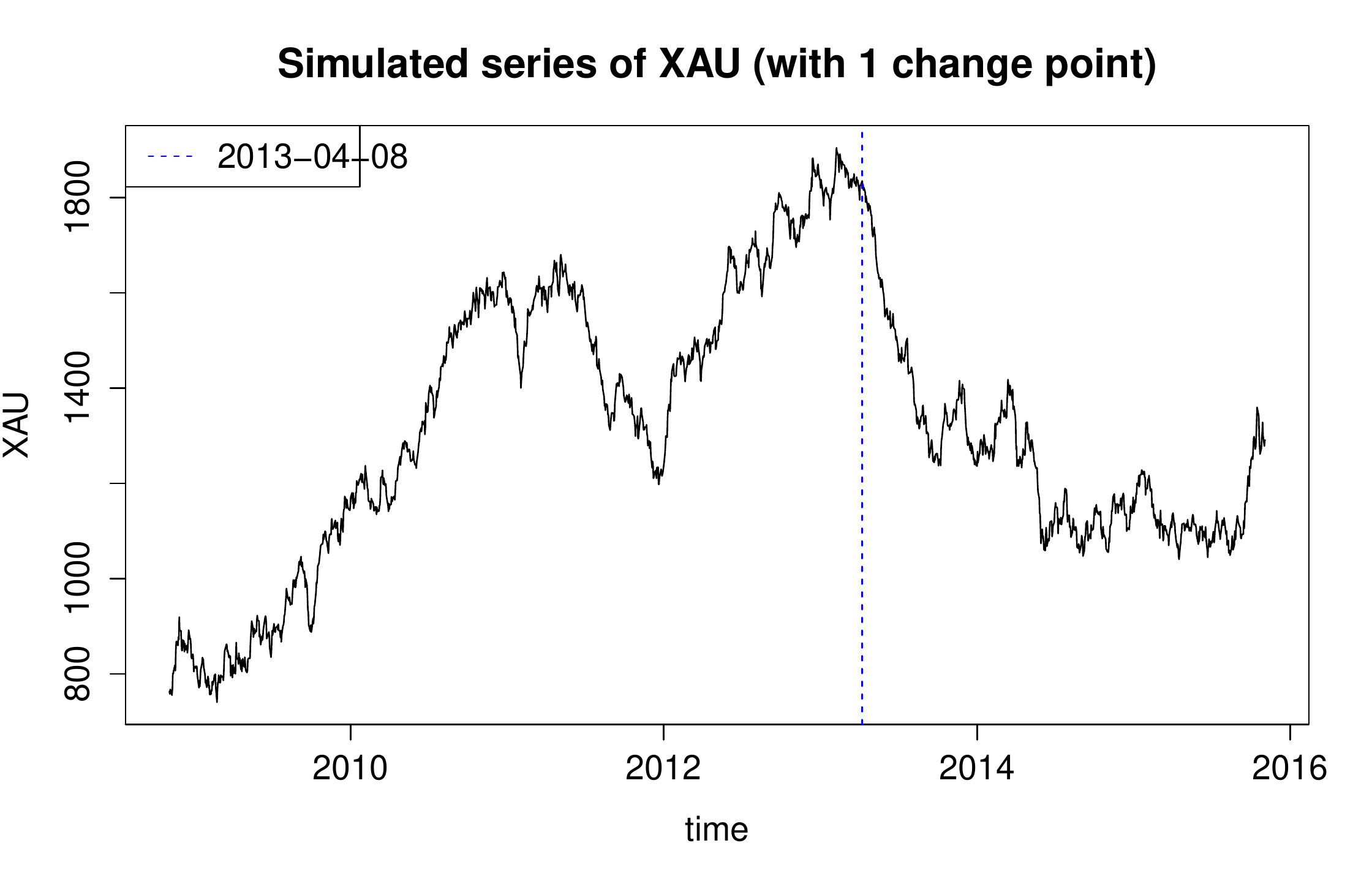}
\includegraphics[height=2.2in,width=3in]{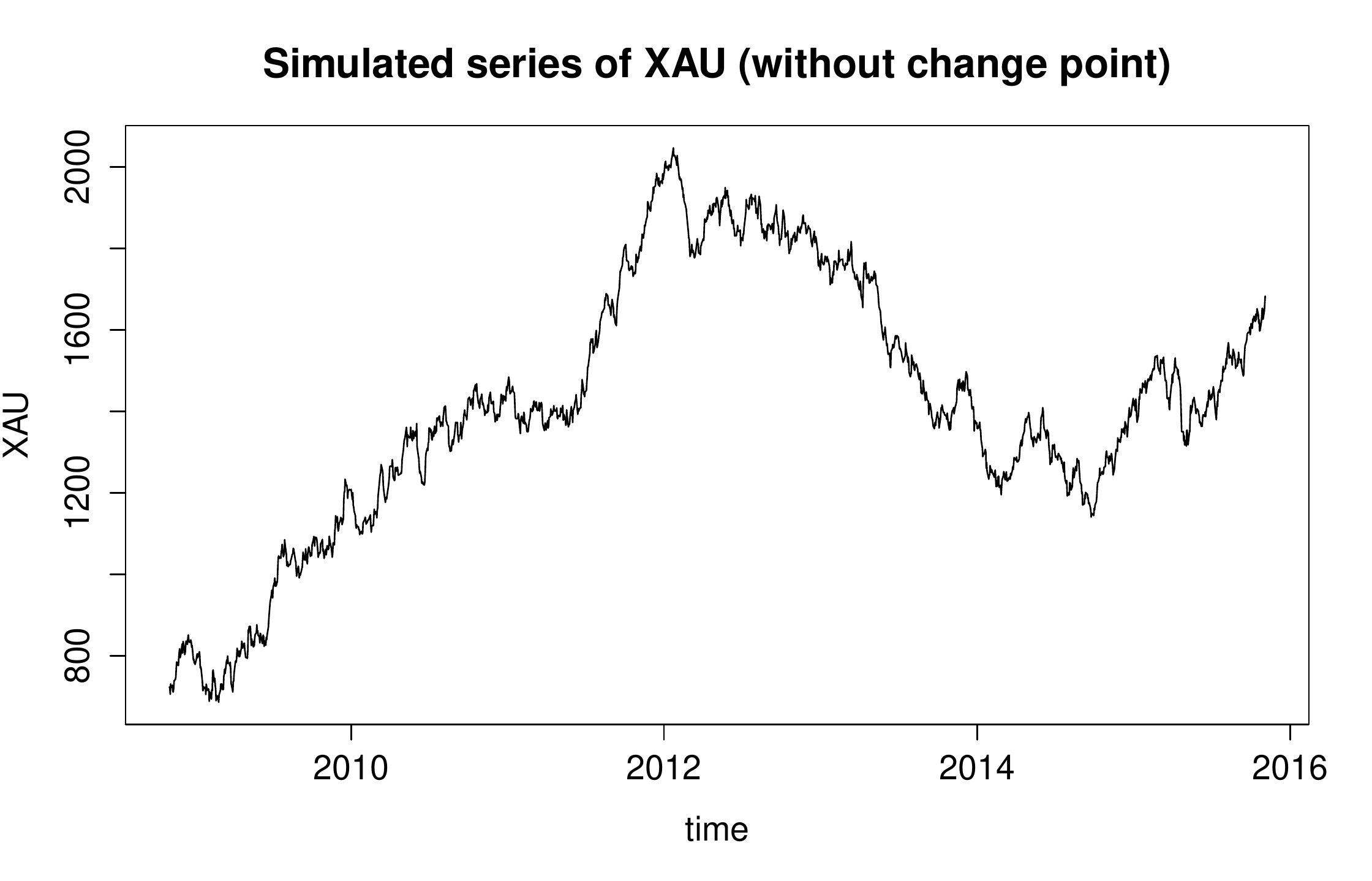}
\caption{\small Simulated series of XAU currency (04 November 2008--04 November 2015)}
\label{fig:simXAU7.1}
\end{figure}
\begin{figure}[htpb]
\includegraphics[height=2.2in,width=3in]{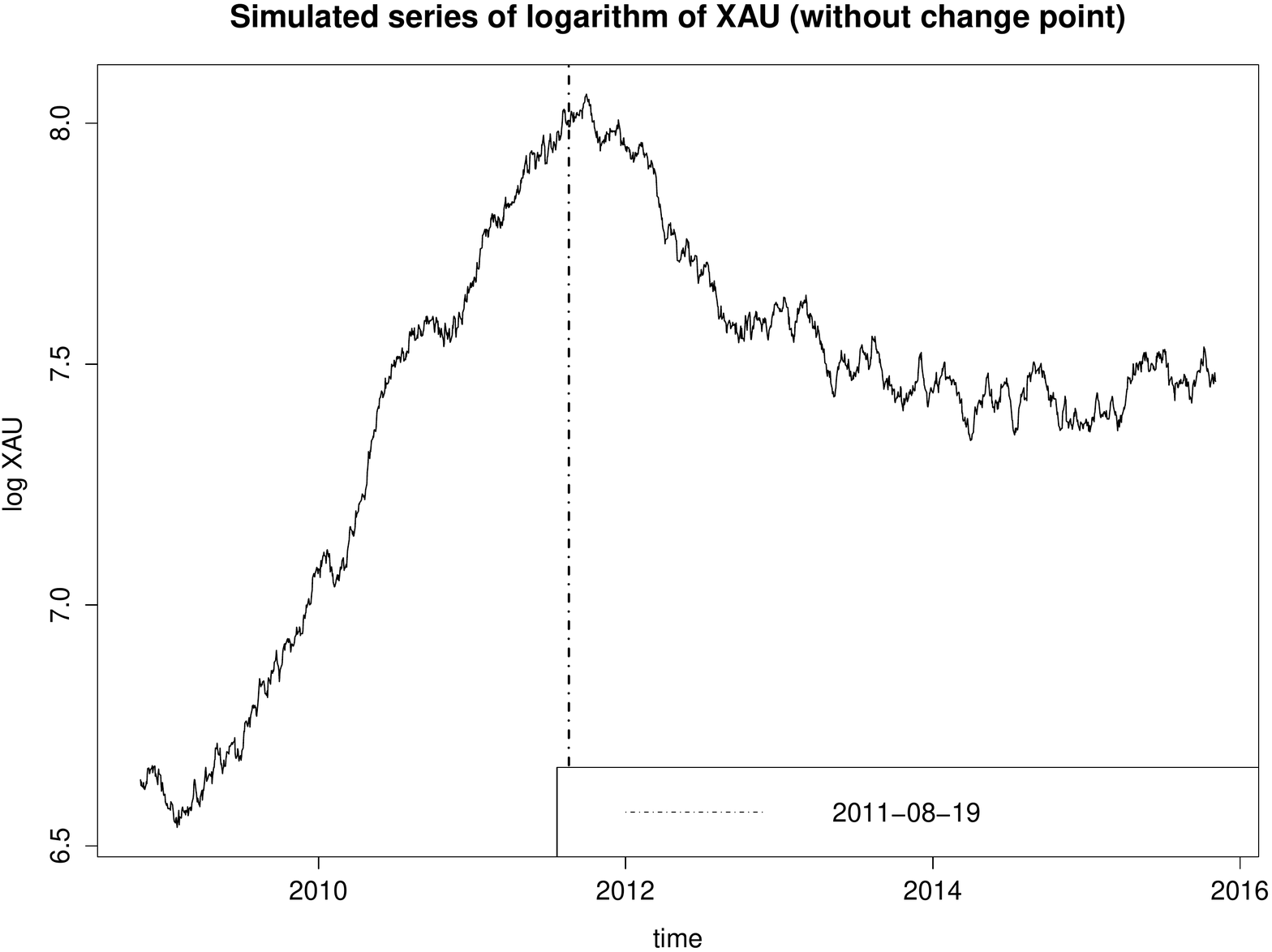}
\includegraphics[height=2.2in,width=3in]{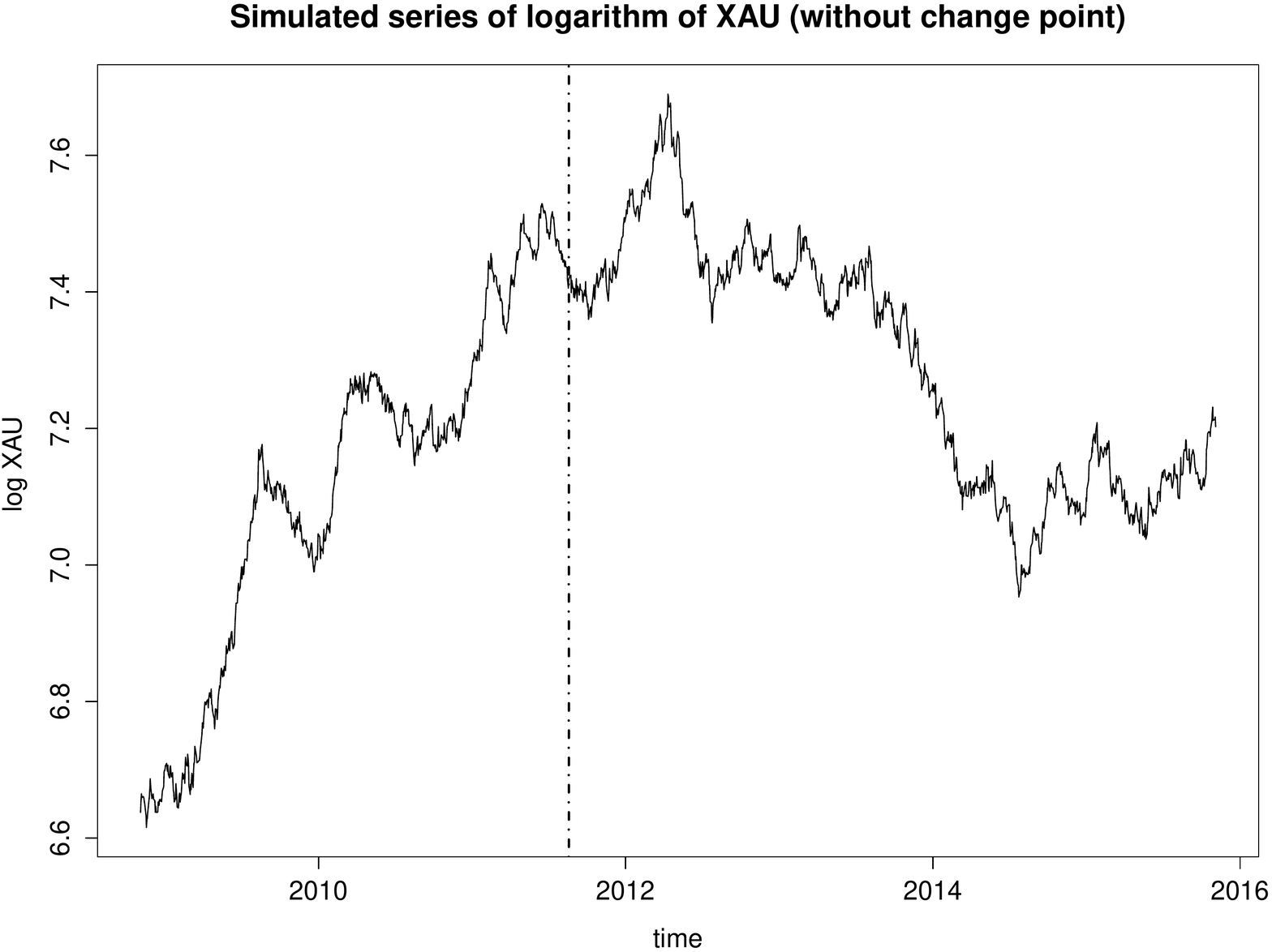}
\caption{\small Simulated series of log-transformed XAU currency (04 November 2008--04 November 2015)}
\label{fig:simXAU7.2}
\end{figure}
\ \\
\noindent As one can see from the results, for the original XAU currency,
the detected change point based on (\ref{cp1}) and (\ref{mle1}) in both
models are still the same as in the 15 years' time period, suggesting the
consistency of the proposed methods. Further, the comparisons
of log-likelihood suggest that increasing the number of coefficients in
the model are useful in producing higher log-likelihood. This finding is
consistent with the previous study. However, after applying (\ref{ic}) to test
the existence of change point, we find that both models fail to pass the test.
This tells us that using (\ref{realdata1-2}) would be enough for modelling the
series for this time period.\\
\ \\
On the other hand, after applying the log transformation to the XAU currency,
the SIC reveals that the model (\ref{realdata1-1}) is still more suitable
than (\ref{realdata2-1}) for this case. Moreover, as one can see from the
results shown in Table~\ref{table9.2}, together with the simulated series provided
in figure~\ref{fig:simXAU7.1} and \ref{fig:simXAU7.2}, similar outcome as in the
case of log-transformed XAU currency with 15 years' time period again suggests that
for this log-transformed data series, the improvement from employing a model with a
change point is not as significant as the one in the original XAU currency series
with a 15-years time period.\\
\ \\
\section{Conclusion}
\noindent The theoretical results and illustrative examples,
involving both simulated and observed data,
of this paper were motivated by the
practical considerations of point-change detection. Such motivations
are driven by applications centred on
confirming significant changes in the time-series data so that proper
responses and policies could be put in place as in the case of
interest rate-setting behaviour of monetary regulatory authority,
design of trading strategies in hedging and speculation, appropriate
calibration of models in financial product valuation, amongst others.
Our main contribution highlighted the development of
MLL- and LSSE-based methods in showing the existence or non-existence blue
of a change point and
in determining the unknown location whenever such a change point exists. We established
the equivalence of the estimators for the change point under the two methods.
In addition,
we provided conditions so that our estimators for the change point location are
asymptotically consistent, which in turn aided the design of an efficient
implementation algorithm. Our work certainly gives impetus for the investigation and
development of methodology suited in tackling the
multiple-change point problem, which is what we commonly encounter
in practice. Our research results aim to lay down the groundwork so that
further modifications could be made and new ideas could be adopted
in making further progress in point-change detection involving time
series models with more complex and elaborate dynamics and stylised features.\\
\ \\
\ \\
\appendix
\renewcommand{\thesection}{Appendix \Alph{section}}
\section{Proofs of Propositions~\ref{prn13}, \ref{prn14}, \ref{prn-mle1}
and \ref{prn-mle2} \label{AppendixA}}
\subsection{Preliminaries for the proofs of Propositions~\ref{prn13} and \ref{prn14}}
\noindent In this section, we let $||A||$ denote the Euclidean norm for
a vector $A$ and $||B||=\sqrt{\mathrm{trace}(B'B)}$ for a matrix $B$.
Further, let $\hat{u}_i$ be the residual of the $i$th element in (\ref{ou2}) based on the
estimated change point $\hat{\tau}$. That is,
$\hat{u}_i=Y_i-Z_i\hat{\theta}^{(j)}=Z_i{\theta}^{(j)}-Z_i\hat{\theta}^{(j)}+u_i$ ,
$j=1$ for $0<i \leq \hat{\tau}$ and $\tau=\hat{\tau}$ estimated by (\ref{cp1}).
Similarly, let $\hat{u}_i^0$ be the residual of the $i$th element in  (\ref{ou2})
based on the exact change point ${\tau}^0$ and the associated MLE of $\theta^{(j)}$
denoted by $\theta^{(j,0)}$, $j=1,2$. Without loss of generality, we assume that $\hat{\tau}>\tau^0$.\\
\ \\
\noindent The proofs of Propositions~\ref{prn13} and \ref{prn14} both rely on
investigating the behaviour of
\begin{equation}\label{prn1-eq1}\phi\left(\sum_{t_i\in[0,T]}
\hat{u}_i'\hat{u}_i-\sum_{t_i\in[0,T]}\hat{u}_i^{0\prime}\hat{u}_i^{0}\right).\end{equation}

\noindent We divide the time period $[0,T]$ involved in (\ref{prn1-eq1}) into
3 sub-intervals: $[0,\tau^0]$, $(\tau^0,\hat{\tau}]$ and $(\hat{\tau},T]$. Then,
by substituting the expressions for $\hat{u}_i$ and $\hat{u}_i^0$ into (\ref{prn1-eq1})
and applying the identity $(a+b)^2=a^2+2ab+b^2$ to expand the quadric error terms, we have
\begin{eqnarray}
(\ref{prn1-eq1})&=&\phi\left(\sum_{t_i\in[0,\tau^0]}(Z_i({\theta}^{(1)}-\hat{\theta}^{(1)}))^2
+\sum_{t_i\in (\tau^0,\hat{\tau}]}(Z_i({\theta}^{(2)}-\hat{\theta}^{(1)}))^2
+\sum_{t_i\in (\hat{\tau},T]}(Z_i({\theta}^{(2)}-\hat{\theta}^{(2)}))^2\right)\label{prn1-eq2}\\
&&-\phi\sum_{t_i\in[0,\tau^0]}(Z_i{\theta}^{(1)}-Z_i\hat{\theta}^{(1,0)})^2-\phi\sum_{t_i\in(\tau^0,T]}
(Z_i{\theta}^{(2)}-Z_i\hat{\theta}^{(2,0)})^2\label{prn1-eq3}\\
&&+2 \phi \left(\sum_{t_i\in[0,\tau^0]}u_iZ_i'(\hat{\theta}^{(1,0)}-\hat{\theta}^{(1)})+\sum_{t_i\in (\tau^0,\hat{\tau}]}u_iZ_i'
(\hat{\theta}^{(2,0)}-\hat{\theta}^{(1)})+\sum_{t_i\in (\hat{\tau},T]}u_iZ_i'
({\theta}^{(2,0)}-\hat{\theta}^{(2)})\right)\label{prn1-eq4},
\end{eqnarray}
where $\phi$ is a positive scalar to be defined later,~
$\hat{\theta}^{(1)}=Q_{(0,\hat{\tau})}^{-1}\tilde{R}_{(0,\hat{\tau})},$~
$\hat{\theta}^{(2)}=Q_{(\hat{\tau},T)}^{-1}\tilde{R}_{(\hat{\tau},T)},$
$\hat{\theta}^{(1,0)}=Q_{(0,{\tau}^0)}^{-1}\tilde{R}_{(0,{\tau}^0)}$,
and $\hat{\theta}^{(2,0)}=Q_{({\tau}^0,T)}^{-1}\tilde{R}_{({\tau}^0,T)}$, respectively.\\
\ \\
\noindent It follows as well from Proposition~2.2.8 in Zhange~(2015) (see also Proposition 4.1 in Dehling, et al.~(2010)) that
$\tilde{R}_{(0,\tau^0)}=Q_{(0,\tau^0)}\theta^{(1)}$ $+\sigma R_{(0,{\tau^0})},$
$\tilde{R}_{(\tau^0,\hat{\tau})}=Q_{(\tau^0,\hat{\tau})}\theta^{(2)}+\sigma R_{{\tau^0},\hat{\tau}},$
$\tilde{R}_{(\tau^0,T)}=Q_{(\tau^0, T)}\theta^{(2)}+\sigma R_{({\tau^0},T)}$
and $\tilde{R}_{(\hat{\tau},T)}=Q_{(\hat{\tau},T)}\theta^{(2)}+\sigma R_{(\hat{\tau},T)}$.
Therefore, $\hat{\theta}^{(1,0)}-{\theta}^{(1)}=\sigma Q_{(0,\tau^0)}^{-1}R_{(0,\tau^0)}$,
and $\hat{\theta}^{(2,0)}-{\theta}^{(2)}=\sigma Q_{(\tau^0,T)}^{-1}R_{(\tau^0,T)}$.
In this case, we have
\begin{equation}\label{prn1-eq3-1}(\ref{prn1-eq3})=-\phi\sigma^2 (R_{(0,\tau^0)}'
Q_{(0,\tau^0)}^{-1}\sum_{t_i\in[0,\tau^0]}Z_i'Z_i Q_{(0,\tau^0)}^{-1} R_{(0,\tau^0)}
+ R_{(\tau^0,T)}'Q_{(\tau^0,T)}^{-1}
\sum_{t_i\in(\tau^0,T]}Z_i'Z_i Q_{(\tau^0,T)}^{-1} R_{(\tau^0,T)}).\end{equation}
\ \\
\noindent Similarly,
 \begin{eqnarray}\hat{\theta}^{(1,0)}-\hat{\theta}^{(1)}&=&\sigma Q_{(0,\tau^0)}^{-1}
 R_{(0,\tau^0)} -   Q_{(0,\hat{\tau})}^{-1}Q_{(\tau^0,\hat{\tau})} (\theta^{(2)}-\theta^{(1)})
 -\sigma Q_{(0,\hat{\tau})}^{-1} R_{(0,\hat{\tau})}\label{prn1-eq4-1},\\
\hat{\theta}^{(2,0)}-\hat{\theta}^{(1)}&=&\sigma Q_{(\tau^0,T)}^{-1} R_{(\tau^0,T)}
-   Q_{(0,\hat{\tau})}^{-1}Q_{(0,\tau^0)} (\theta^{(1)}-\theta^{(2)})
-\sigma Q_{(0,\hat{\tau})}^{-1} R_{(0,\hat{\tau})}\label{prn1-eq4-2},\\
\hat{\theta}^{(2,0)}-\hat{\theta}^{(2)}&=&\sigma Q_{(\tau^0,T)}^{-1}
R_{(\tau^0,T)} -\sigma Q_{(\hat{\tau},T)}^{-1} R_{(\hat{\tau},T)}\label{prn1-eq4-3},
\end{eqnarray}
and
\begin{equation}\label{prn1-eq4-0}(\ref{prn1-eq4})=2 \phi \left(\sum_{t_i\in[0,\tau^0]}
u_iZ_i'(\ref{prn1-eq4-1}) +\sum_{t_i\in (\tau^0,\hat{\tau}]}
u_iZ_i'(\ref{prn1-eq4-2})+\sum_{t_i\in (\hat{\tau},T]}u_iZ_i'(\ref{prn1-eq4-3})\right).\end{equation}

\begin{proof}[Proof of Proposition~\ref{prn13}]
\noindent Take $\displaystyle \phi=\frac{1}{T}$. In general, (\ref{prn1-eq1}) is non-positive
with probability 1 since by (\ref{cp1}), $\hat{\tau}$ is chosen from all
possible values in $[0,T]$ to minimise the SSR, whilst $\tau^0$ is just a
particular value in $[0,T]$. Hence, it suffices to show that if the rate of
$\tau^0$, given by $s^0=\tau^0/T$,  can not be consistently estimated by
$\hat{s}=\hat{\tau}/T$, then $(\ref{prn1-eq1})>0$ with positive probability
and thus we have a contradiction. \\
\ \\
\noindent First, note that in case that the rate of the change point ${\tau^0}$
can not be consistently estimated, then with positive probability there exists
an $\eta>0$ such that $\hat{s}T-s^0 T>\eta T>L^0$ for large $T$. In this case,
we have $(\ref{prn1-eq2})\geq  C_1 ||\theta^{(1)}-\theta^{(2)}||^2$ for
some $C_1>0$ with positive probability (see Bai and Perron (1998), Lemma 2). \\
\ \\
\noindent To proceed further, we prove the following inequality. Note that by (\ref{sol2}), $\sup_{t\geq 0} \textit{E}((X_t)^2)\leq K_1$ for some $K_1$ with $0<K_1<\infty$. This implies that for $0<\tau_1^*<\tau_2^*\leq T$,
\begin{equation}\label{ineq}
\int_{\tau_1^*}^{{\tau}_2^*}E(X_t^2)dt \leq K_1(\tau_2^*-\tau_1^*).
 \end{equation}
\noindent Further, by the Markov inequality and Ito isometry, together with inequality \eqref{ineq}, we have
\begin{eqnarray*}
\textrm{P}\left(\frac{1}{\sqrt{\tau_2^*-\tau_1^*}}|\int_{\tau_1^*}^{{\tau_2}^*}X_td W_t|>K^*\right)
&\leq& \frac{\textit{E}\left(|\int_{\tau_1^*}^{\tau_2^*}X_td W_t|^2\right)}{(\tau_2^*-\tau_1^*)(K^*)^2}
=\frac{\int_{\tau_1^*}^{{\tau}_2^*}E(X_t^2)dt}{(\tau_2^*-\tau_1^*)(K^*)^2}
\leq \frac{K_1(\tau_2^*-\tau_1^*)}{(\tau_2^*-\tau_1^*)(K^*)^2}=\frac{K_1}{(K^*)^2}.
\nonumber
\end{eqnarray*}
\noindent Under Assumption 2, it follows from (B.28) in Zhang~(2015) that
\begin{eqnarray*}
\textrm{P}\left(\frac{1}{\sqrt{\tau_2^*-\tau_1^*}}\left| \int_{\tau_1^*}^{{\tau_2}^*}\varphi_i(t)d W_t \right| >K^*\right)
&\leq& \frac{\textit{E}\left( \left| \int_{\tau_1^*}^{\tau_2^*}\varphi_i(t)d W_t\right| ^2\right)}{(\tau_2^*-\tau_1^*)(K^*)^2}
\leq \frac{1}{(K^*)^2}. \nonumber
\end{eqnarray*}
\noindent Hence, by letting $K^*=\log^{a^*} T$ or $K^*= (\hat{\tau}-\tau^0)^{a^*}$
for $0<a^*<1/2$, the above probability tends to 0 as $T$
tends to infinity. This implies that
\begin{equation}\label{conv_R}\frac{1}{\sqrt{\tau_2^*-\tau_1^*}}||R_{(\tau_1^*,\tau_2^*)}||
=O_p(\log^{a^*} T) \quad \textrm{or}\quad O_p( (\hat{\tau}-\tau^0)^{a^*})
\quad \textrm{for any}
\quad 0<\tau_1^*<\tau_2^*\leq T.\end{equation}
\ \\
Similarly for the discretised case, $\frac{1}{\sqrt{\tau_2^*-\tau_1^*}}||
\sum_{t_i\in (\tau_1^*,\tau_2^*]}Z_iu_i||=O_p(\log^{a^*} T)$ or $O_p( (\hat{\tau}-\tau^0)^{a^*})$.
Furthermore, it follows from Proposition~\ref{conQ1} - \ref{conQ2}, and Continuous
Mapping Theorem that
\begin{equation}\label{conv_Q}||\frac{1}{T}Q_{(\tau_1^*,\tau_2^*)}||=O_p(1)
\textrm{ and } ||TQ_{(\tau_1^*,\tau_2^*)}^{-1}||=O_p(1) \end{equation}
for any $\tau_1^*,\tau_2^*\in \{0,\tau^0,\hat{\tau}, T\}$ such that $\tau_1^*<\tau_2^*$.
Then, applying the Cauchy-Schwarz's inequality to (\ref{prn1-eq3-1}) and (\ref{prn1-eq4-0}),
together with above asymptotic results, we have that (\ref{prn1-eq3}) and (\ref{prn1-eq4})
are both $o_p(1)$. Hence, (\ref{prn1-eq1}) is dominated by (\ref{prn1-eq2}),
which is positive, and thus gives a contradiction.
Therefore, $\hat{s}-s^0\xrightarrow[T\rightarrow \infty]{P}0$.
\end{proof}
\ \\
\ \\
\begin{proof}[Proof of Proposition~\ref{prn14}]
\noindent Write $\phi:=\frac{1}{\hat{\tau}-\tau^0}$ and $V_\eta:=\{\tau: |\tau-\tau^0|\leq\eta T\}$.
It follows from Proposition~\ref{prn13} that for each
$\eta>0$, $P(\hat{\tau}\in V_\eta)\xrightarrow[T \rightarrow \infty]{} 1$. Therefore,
we only need to investigate the sum of squared error $SSE(\hat{\tau})$ for those
$\hat{\tau}\in V_\eta$.  For $C>0$, define the set $V_{\eta}(C)=\{{\tau}: C<|{\tau}-{\tau^0}|<\eta T\}$
and let $\hat{\tau}$ be the estimated change point with the minimum taken over
the set $V_\eta(C)$.
Then,  it suffices to show that  the order of these three terms, or one of the three terms
is larger than all of the remaining terms in (\ref{prn1-eq1}), and that leads to a contradiction
since the term $(\ref{prn1-eq1}) \leq 0$ with probability 1. \\
\ \\
\noindent First notice that (\ref{prn1-eq2}) is $O_p(1)$ instead of $o_p(1)$.
Hence, it is difficult
to compare it with (\ref{prn1-eq1}) directly. In this case, we need to factorise the
term (\ref{prn1-eq1}). We observe that
\begin{eqnarray*}
\sum_{t_i\in[0,\tau^0]}(Z_i({\theta}^{(1)}-\hat{\theta}^{(1)}))^2&=&(a1)+(a2)+(a3),
\end{eqnarray*}
where
\begin{eqnarray}
(a1)&=&(\theta^{(1)}-\theta^{(2)})'Q_{({\tau^0},\hat{\tau})}Q_{(0,\hat{\tau})}^{-1}
\sum_{t_i\in[0,\tau^0]}Z_i'Z_iQ_{(0,\hat{\tau})}^{-1}Q_{(\tau^0,\hat{\tau})}
(\theta^{(1)}-\theta^{(2)}), \nonumber \\
(a2)&=& R_{(0,\hat{\tau})}\sigma^2Q_{(0,\hat{\tau})}^{-1}\times\sum_{t_i\in[0,\tau^0]}Z_i'Z_i
Q_{(0,\hat{\tau})}^{-1} R_{(0,\hat{\tau})} \nonumber \\
(a3)&=&2\sigma(\theta^{(1)}-\theta^{(2)})'Q_{({\tau^0},\hat{\tau})}Q_{(0,
\hat{\tau})}^{-1}
\sum_{t_i\in[0,\tau^0]}Z_i'Z_i  Q_{(0,\hat{\tau})}^{-1} R_{(0,\hat{\tau})}. \nonumber
\end{eqnarray}
\noindent Similarly,
\begin{eqnarray*}\sum_{t_i\in (\tau^0,\hat{\tau}]}
(Z_i({\theta}^{(2)}-\hat{\theta}^{(1)}))^2&=& (a4)+(a5)+(a6),\end{eqnarray*}
where
\begin{eqnarray}
(a4)&=&(\theta^{(1)}-\theta^{(2)})'Q_{(0,{\tau^0})}Q_{(0,\hat{\tau})}^{-1}
\sum_{t_i\in (\tau^0,\hat{\tau}]}Z_i'Z_iQ_{(0,\hat{\tau})}^{-1}Q_{(0,\tau^0)}
(\theta^{(1)}-\theta^{(2)}), \nonumber \\
(a5)&=& R_{(0,\hat{\tau})}\sigma^2 Q_{(0,\hat{\tau})}^{-1} \times
\sum_{t_i\in (\tau^0,\hat{\tau}]}
Z_i'Z_i  Q_{(0,\hat{\tau})}^{-1} R_{(0,\hat{\tau})}, \nonumber \\
(a6)&=&2\sigma(\theta^{(1)}-\theta^{(2)})'Q_{(0,{\tau^0})}
Q_{(0,\hat{\tau})}^{-1}\sum_{t_i\in (\tau^0,\hat{\tau}]}Z_i'Z_i
Q_{(0,\hat{\tau})}^{-1} R_{(0,\hat{\tau})}. \nonumber
\end{eqnarray}
\noindent and
\begin{eqnarray*}
\sum_{t_i\in (\hat{\tau},T]}(Z_i{\theta}^{(2)}-Z_i\hat{\theta}^{(2)})'
(Z_i{\theta}^{(2)}-Z_i\hat{\theta}^{(2)})&=&\sigma^2 R_{(\hat{\tau},T)}
Q_{(\hat{\tau},T)}^{-1}\sum_{t_i\in (\hat{\tau},T]}Z_i'Z_i
Q_{(\hat{\tau},T)}^{-1} R_{(\hat{\tau},T)}=(a7).
\end{eqnarray*}

\noindent One could see that $\phi (a1)$, $\phi (a4)$ and $\phi (a6)$ are
all of order $o_p(1)$. Additionally, by Theorem A.1 in Tobing and McGlichrist
(1992), we have that for large $T$,
\begin{equation}\label{tm-1}Q_{(0,\hat{\tau})}^{-1}=Q_{(0,{\tau^0})}^{-1}+O_p
\left(\frac{\hat{\tau}-\tau^0}{T^2}\right).\end{equation}
Therefore,
\begin{eqnarray}&& R_{(0,\hat{\tau})}' Q_{(0,\hat{\tau})}^{-1}
\sum_{t_i\in[0,\tau^0]}Z_i'Z_i  Q_{(0,\hat{\tau})}^{-1} R_{(0,\hat{\tau})}
- R_{(0,\tau^0)}' Q_{(0,{\tau}^0)}^{-1}\sum_{t_i\in[0,\tau^0]}Z_i'Z_i
Q_{(0,\tau^0))}^{-1} R_{(0,\tau^0))}\label{prn1-eq5} \\
&=&(R_{(0,{\tau}^0)}+R_{(\tau^0,\hat{\tau})})'\left( Q_{(0,{\tau}^0)}^{-1}
+O_p\left(\frac{\hat{\tau}-\tau^0}{T^2}\right)\right)
\sum_{t_i\in[0,\tau^0]}Z_i'Z_i  \left( Q_{(0,{\tau}^0)}^{-1}
+O_p\left(\frac{\hat{\tau}-\tau^0}{T^2}\right)\right)\nonumber\\
&&\times(R_{(0,{\tau}^0)}+R_{(\tau^0,\hat{\tau})})- R_{(0,\tau^0)}
Q_{(0,{\tau}^0)}^{-1}\sum_{t_i\in[0,\tau^0]}Z_i'Z_i  Q_{(0,\tau^0))}^{-1}
R_{(0,\tau^0)},\nonumber\\
&&O_p\left(\frac{\hat{\tau}-\tau^0}{T^2}\right)R_{(0,{\tau}^0)}
\sum_{t_i\in[0,\tau^0]}Z_i'Z_i  Q_{(0,\tau^0))}^{-1}R_{(0,\tau^0)}
+O_p\left(\frac{(\hat{\tau}-\tau^0)^2}{T^4}\right)R_{(0,{\tau}^0)}
\sum_{t_i\in[0,\tau^0]}Z_i'Z_i  R_{(0,\tau^0)}\nonumber\\
&&+2R_{(0,\tau^0)} Q_{(0,{\tau}^0)}^{-1}\sum_{t_i\in[0,\tau^0]}
Z_i'Z_i Q_{(0,{\tau}^0)}^{-1}R_{(\tau^0,\hat{\tau})}
+O_p\left(\frac{\hat{\tau}-\tau^0}{T^2}\right)R_{(0,\tau^0)}
Q_{(0,{\tau}^0)}^{-1}\sum_{t_i\in[0,\tau^0]}Z_i'Z_i R_{(\tau^0,\hat{\tau})}\nonumber\\
&&+ O_p\left(\frac{\hat{\tau}-\tau^0}{T^2}\right) R_{(0,\tau^0)}
\sum_{t_i\in[0,\tau^0]}Z_i'Z_i Q_{(0,{\tau}^0)}^{-1}R_{(\tau^0,\hat{\tau})}
+ O_p\left(\frac{(\hat{\tau}-\tau^0)^2}{T^4}\right) R_{(0,\tau^0)}
\sum_{t_i\in[0,\tau^0]}Z_i'Z_i  R_{(\tau^0,\hat{\tau})}\nonumber\\
&&+R_{(\tau^0,\hat{\tau})} Q_{(0,{\tau}^0)}^{-1}\sum_{t_i\in[0,\tau^0]}
Z_i'Z_i Q_{(0,{\tau}^0)}^{-1}R_{(\tau^0,\hat{\tau})}
+ O_p\left(\frac{\hat{\tau}-\tau^0}{T^2}\right) R_{(\tau^0,\hat{\tau})}
Q_{(0,{\tau}^0)}^{-1}\sum_{t_i\in[0,\tau^0]}Z_i'Z_i  R_{(\tau^0,\hat{\tau})}\nonumber\\
&&+O_p\left(\frac{(\hat{\tau}-\tau^0)^2}{T^4}\right) R_{(\tau^0,\hat{\tau})}
\sum_{t_i\in[0,\tau^0]}Z_i'Z_i  R_{(\tau^0,\hat{\tau})}\nonumber.
\end{eqnarray}

\noindent Since $\hat{\tau}-\tau^0\leq \eta T$ for each $\eta>0$, and
using the asymptotic results in the proof of Proposition~\ref{prn13}
with small enough $\eta$, we have $\phi \sigma^2 (\ref{prn1-eq5})=o_p(1)$.
Similarly,
$$\phi\sigma^2 \left(R_{(\hat{\tau},T)} Q_{(\hat{\tau},T)}^{-1}
\sum_{t_i\in (\hat{\tau},T]}Z_i'Z_i  Q_{(\hat{\tau},T)}^{-1}
R_{(\hat{\tau},T)}-R_{(\tau^0,T)} Q_{(\tau^0,T)}^{-1}\sum_{t_i\in(\tau^0,T]}
Z_i'Z_i  Q_{(\tau^0,T)}^{-1} R_{(\tau^0,T)} \right)=o_p(1).$$

\noindent Consequently, $(a2)+(a4)+(\ref{prn1-eq2})=o_p(1)$. For $(a3)$, we have
\begin{eqnarray*}
&&(\theta^{(1)}-\theta^{(2)})'Q_{(0,{\tau^0})}Q_{(0,\hat{\tau})}^{-1}
\frac{{\sum_{t_i\in (\tau^0,\hat{\tau}]}Z_i'Z_i}}{\hat{\tau}-\tau^0}
Q_{(0,\hat{\tau})}^{-1}Q_{(0,\tau^0)}(\theta^{(1)}-\theta^{(2)})\label{ratecov_1}\\
&=& (\theta^{(1)}-\theta^{(2)})'Q_{(0,{\tau^0})}\left(Q_{(0,{\tau^0})}^{-1}
+O_p\left(\frac{\hat{\tau}-\tau^0}{T^2}\right)\right)
\frac{{\sum_{t_i\in (\tau^0,\hat{\tau}]}Z_i'Z_i}}{\hat{\tau}-\tau^0}
\left(Q_{(0,{\tau^0})}^{-1}+O_p\left(\frac{\hat{\tau}-\tau^0}{T^2}\right)\right)\\
&&\times Q_{(0,\tau^0)}(\theta^{(1)}-\theta^{(2)})
= (\theta^{(1)}-\theta^{(2)})'\frac{{\sum_{t_i\in (\tau^0,\hat{\tau}]}Z_i'Z_i}}
{\hat{\tau}-\tau^0}(\theta^{(1)}-\theta^{(2)})+o_p(1),
\end{eqnarray*}
with
$ (\theta^{(1)}-\theta^{(2)})'\frac{{\sum_{t_i\in (\tau^0,\hat{\tau}]}Z_i'Z_i}}
{\hat{\tau}-\tau^0}(\theta^{(1)}-\theta^{(2)})
\geq \gamma_1 ||\theta^{(1)}-\theta^{(2)}||^2,$
where $\gamma_1$ is the minimum eigenvalue of
$\frac{\sum_{t_i\in (\tau^0,\hat{\tau}]}Z_i'Z_i}{\hat{\tau}-\tau^0}.$
Under Assumption 4 with a suitable choice of $C$, we have that for
$\hat{\tau}\in V_\eta (C)$, $\gamma_1$ is bounded away from 0.
Hence, $\phi(a3)\geq C_2 ||\theta^{(1)}-\theta^{(2)}||^2$ for
some $C_2>0$. Moreover, applying (\ref{tm-1}) to (\ref{prn1-eq4}),
together with some factorisations, we have that $(\ref{prn1-eq4})=o_p(1)$.
Therefore, the term $\phi(a3)$ dominates all others and it is positive
with probability 1 for large $T$. This implies that with large probability,
$(\ref{prn1-eq1})>0$, which gives a contradiction. This indicates that with
large probability $\hat{\tau}$ can not be in the set $V_{\eta}(C)$ and
hence $P(T|\hat{s}-s^0|\geq C)\leq \epsilon$ when $T$ is large.
\end{proof}
\ \\
\begin{rem}\label{remark2}As discussed in Remark~\ref{remark1} in Section 4, in case the shift is of shrinking magnitude with shrinking speed $v_T$, the asymptotic behavior (i) discussed in Remark~\ref{remark1} may be verified by following the same arguments in the proof of Proposition~\ref{prn13} with $\phi=\frac{1}{T^{2r^*}}$, together with the fact that $T\phi||\theta^{(1)}-\theta^{(2)}||^2=(T^{1-2r^*}v_T^2||\mathbf{M}||^2) \xrightarrow[T\rightarrow \infty]{}\infty$ and $\log T/T^{2r^*} \xrightarrow[T\rightarrow \infty]{}0$. On the other hand, (ii) may be verified by following similar arguments as in the proof of Proposition~\ref{prn14} to investigate the set $V_{\eta}(C, v_T)=\{{\tau}: C/v_T^2<|{\tau}-{\tau^0}|<\eta T\}$ instead of $V_{\eta}(C)$.
\end{rem}
\ \\
\subsection{Proof of  Proposition~\ref{prn-mle1} and \ref{prn-mle2}}
\noindent Write
$V(t):=(\varphi_1(t),...,\varphi_p(t),-X_{t})$ and let $\log \mathcal{L}_1$
be the log likelihood function based on the estimated change point $\hat{{\tau}}$
and the associated MLE $(\hat{\theta}^{(1)},\hat{\theta}^{(2)})$. That is,
$\log\mathcal{L}_1=\frac{1}{\sigma^2}(\int_{0}^{\hat{\tau}}S(\hat{\theta}^{(1)},t,X_t)dX_t
+\int_{\hat{\tau}}^T S(\hat{\theta}^{(2)},t,X_t)dX_t )-\frac{1}
{2\sigma^2}(\int_{0}^{\hat{\tau}}S^2(\hat{\theta}^{(1)},t,X_t)dt
+\int_{\hat{\tau}}^TS^2(\hat{\theta}^{(2)},t,X_t)dt.$
Similarly, we let $\log\mathcal{L}_0$ be the log likelihood function
based on the exact value of change point $\tau^0$ and the associated
MLE $(\hat{\theta}^{(1,0)},\hat{\theta}^{(2,0)})$. Then, the proof of
Propositions~\ref{prn-mle1} and \ref{prn-mle2} rely on the behaviour of
\begin{equation}\label{prn2-eq1}\phi (\log\mathcal{L}_1- \log\mathcal{L}_0).\end{equation}
\ \\
\noindent Now, without loss of generality, suppose first that $\hat{\tau}>\tau^0$ and divide the time period $[0,T]$ into three sub-intervals: $[0,\tau^0]$,
$(\tau^0,\hat{\tau}]$ and $(\hat{\tau},T]$. Using relation (\ref{ou1}),
along with some algebraic manipulation, we have
\begin{eqnarray*}
 (\ref{prn2-eq1})&=&\frac{1}{2\sigma^2}\left(\int_{0}^{{\tau^0}}(V(t)\hat{\theta}^{(1)})^2dt
+\int_{\tau^0}^{\hat{\tau}}(V(t)\hat{\theta}^{(1)})^2dt
+\int_{\hat{\tau}}^{T}(V(t)\hat{\theta}^{(2)})^2dt\right)\\
&&-\frac{1}{2\sigma^2}\left(\int_{0}^{{\tau^0}}(V(t)\hat{\theta}^{(1,0)})^2dt
+\int_{\tau^0}^{\hat{\tau}}(V(t)\hat{\theta}^{(2,0)})^2dt
+\int_{\hat{\tau}}^{T}(V(t)\hat{\theta}^{(2,0)})^2dt\right)\\
&&+\frac{1}{\sigma}\left(\int_{0}^{{\tau^0}}(V(t)\hat{\theta}^{(1)})dW_t
+\int_{\tau^0}^{\hat{\tau}}(V(t)\hat{\theta}^{(1)})dW_t
+ \int_{\hat{\tau}}^{T}(V(t)\hat{\theta}^{(2)})dW_t\right)\\
&&- \frac{1}{\sigma}\left(\int_{0}^{{\tau^0}}(V(t)\hat{\theta}^{(1,0)})dW_t
+\int_{\tau^0}^{\hat{\tau}}(V(t)\hat{\theta}^{(2,0)})dW_t
+ \int_{\hat{\tau}}^{T}(V(t)\hat{\theta}^{(2,0)})dW_t\right).
\end{eqnarray*}

\noindent We plug in the expressions of the MLEs, and
after some algebraic computations, we get
\begin{eqnarray}
  (\ref{prn2-eq1})&=&\frac{\phi}{\sigma}\left(2(\theta^{(1)}-\theta^{(2)})'R_{({\tau}^0,\hat{\tau})}
- 2(\theta^{(1)}-\theta^{(2)})'Q_{(\tau^0,\hat{\tau})}Q_{(0,\hat{\tau})}^{-1}
R_{(0,\hat{\tau})}\right)\quad\quad\label{prn2-eq2}\\
&&+ \frac{3\phi}{2}\left(R_{(0,\hat{\tau})}'Q_{(0,\hat{\tau})}^{-1} R_{(0,\hat{\tau})}
- R_{(0,{\tau}^0)}'Q_{(0,{\tau}^0)}^{-1} R_{(0,\tau^0)}
+ R_{(\hat{\tau},T)}'Q_{(\hat{\tau},T)}^{-1} R_{(\hat{\tau},T)}\right)\label{prn2-eq3}\\
&&- \frac{3\phi}{2}R_{(\tau^0,T)}'Q_{(\tau^0,T)}^{-1} R_{(\tau^0,T)}\label{prn2-eq4}\\
&&- \frac{\phi}{2\sigma^2}(\theta^{(1)}-\theta^{(2)})' Q_{(0,\tau^0)}
Q_{(0,\hat{\tau})}^{-1}Q_{({\tau^0},\hat{\tau})}(\theta^{(1)}-\theta^{(2)})\label{prn2-eq5}.
\end{eqnarray}

\noindent The remaining parts of the proofs for
Propositions~\ref{prn-mle1} and \ref{prn-mle2} depend on investigating the asymptotic behaviours of
(\ref{prn2-eq2}) - (\ref{prn2-eq5}).

\begin{proof}[Proof of Proposition~\ref{prn-mle1}.]
Let $\displaystyle \phi=\frac{1}{T}$. Note that $\log \mathcal{L}_1$ is taken to be the maximum
of the log likelihood function from all possible choices of $\tau\in[0,T]$
whilst $\log\mathcal{L}_0$ is based on one particular change point $\tau^0\in [0,T]$.
It follows from the definition of MLE that $(\ref{prn2-eq1})\geq 0$ with probability 1.
However, if the rate of change point $s^0$ is not consistently estimated by $\hat{s}$,
then with positive probability, we have $\hat{\tau}-\tau^0= (\hat{s}-s^0) T > L_0^*$, where $L_0^*$ is defined in Assumption \ref{asm4}. In this case, under Assumptions 3 and 5,
the minimum eigenvalue of $D^*=\frac{1}{2(\hat{s}-s^0) T}(Q_{(0,\tau^0)}Q_{(0,\hat{\tau})}^{-1}Q_{({\tau^0},
\hat{\tau})}+Q_{({\tau^0},\hat{\tau})}Q_{(0,\hat{\tau})}^{-1}Q_{(0,\tau^0)})$ is bounded away from 0.
Denoting this mimimum eigenvalue by $\gamma_2$, we have
\begin{equation}\label{prn2-eq6} \frac{\phi}{2\sigma^2}(\theta^{(1)}
-\theta^{(2)})' Q_{(0,\tau^0)}Q_{(0,\hat{\tau})}^{-1}Q_{({\tau^0},
\hat{\tau})}(\theta^{(1)}-\theta^{(2)})=\frac{(\hat{s}-s^0)}{4\sigma^2}(\theta^{(1)}
-\theta^{(2)})' D^*(\theta^{(1)}-\theta^{(2)})\geq  \frac{(\hat{s}-s^0) \gamma_2}
{4\sigma^2 }||\theta^{(1)}-\theta^{(2)}||^2.\end{equation}
So, $ (\ref{prn2-eq5})\leq - C_3 ||\theta^{(1)}-\theta^{(2)}||^2$
with $\displaystyle C_3=\frac{(\hat{s}-s^0) \gamma_2}{4\sigma^2 }>0$. \\
\ \\
Using (\ref{conv_R}) and (\ref{conv_Q}), together with Cauchy-Schwarz's inequality, along
with some algebraic computations, we find that (\ref{prn2-eq2}) - (\ref{prn2-eq4})
are all of order $o_p(1)$. Hence, (\ref{prn2-eq1})
is dominated by (\ref{prn2-eq5}), which is negative.
This means that  $(\ref{prn2-eq1})<0$ with positive probability,
which is a contradiction. Therefore, for large $T$ and $\forall\epsilon>0$, $\hat{s}-s^0<\epsilon.$
This implies that $\hat{s}-s^0\xrightarrow[T\rightarrow\infty]{P}0$,
which completes the proof.
\end{proof}

\begin{proof}[Proof of Proposition~\ref{prn-mle2}]
\noindent Write $\displaystyle \phi:=\frac{1}{\hat{\tau}-\tau^0}$ and
$V_\eta:=\{\tau: |\tau-\tau^0|\leq\eta T\}$. Then,
it follows from Proposition~\ref{prn-mle1} that for each
$\eta>0$, $P(\hat{\tau}\in V_\eta)\xrightarrow[T \rightarrow \infty]{} 1$.
Therefore, we only need to investigate the asymptotic behaviour of
(\ref{prn2-eq2})--(\ref{prn2-eq5}) for those $\hat{\tau}\in V_\eta$.
For $C>0$, define the set $V_{\eta}(C)=\{{\tau}: C<|{\tau}-{\tau^0}|<\eta T\}$
and let $\hat{\tau}$ be the estimated change point with the minimum taken
over the set $V_\eta(C)$. Then,  it suffices to show that for some $C>0$,
such that for any $\hat{\tau}\in V_\eta(C)$, $(\ref{prn1-eq1})<0$ with positive probability ,
and this leads to a contradiction since the term $(\ref{prn1-eq1}) \leq 0$ with probability 1.
This would imply that for some $C>0$ and any $0<\eta<1$, the global optimisation can
not be achieved on the set $V_{\eta}(C)$. Thus with large probability,
$|\hat{\tau}-\tau^0|\leq C$. \\
\ \\
\noindent First note that under Assumption 4, with a suitable $C$ such that $C>L_0$,
it follows from (\ref{prn2-eq6}) that $ (\ref{prn2-eq5})\leq - C_3 ||\theta^{(1)}-\theta^{(2)}||^2$
for some $C_3>0$. \\
\ \\
\noindent Next, for (\ref{prn2-eq2}), by (\ref{conv_R}) and the Cauchy-Schwarz's inequality,
we have
\begin{eqnarray*}\frac{2(\theta^{(1)}-\theta^{(2)})' R_{({\tau}^0,\hat{\tau})} }
{\sigma (\hat{\tau}-\tau^0)}
\leq \frac{2}{\sigma \sqrt{\hat{\tau}-\tau^0}} || \theta^{(1)}-\theta^{(2)}
|| || \frac{1}{\sqrt{\hat{\tau}-\tau^0}} R_{({\tau}^0,\hat{\tau})}||
=(\hat{\tau}-\tau^0)^{a^*-1/2}O_p(1),
\end{eqnarray*}
where $0<a^*<1/2$, and
\begin{eqnarray*}\frac{2(\theta^{(1)}-\theta^{(2)})' Q_{(\tau^0,\hat{\tau})}
Q_{(0,\hat{\tau})}^{-1}R_{(0,\hat{\tau})} }{\sigma (\hat{\tau}-\tau^0)}
\leq \frac{2}{\sigma \sqrt{T}} || \theta^{(1)}-\theta^{(2)}|| ||
\frac{Q_{(\tau^0,\hat{\tau})}}{ (\hat{\tau}-\tau^0)} || ||
TQ_{(0,\hat{\tau})}^{-1} ||  || \frac{1}{\sqrt{T}} R_{(0,\hat{\tau})}||\\
=o_p(1)
\end{eqnarray*}

\noindent For (\ref{prn2-eq4}), applying again Theorem A.1 in
Tobing and McGlichrist (1992), we have that for large $T$,
\begin{equation}\label{tm-2}Q_{({\tau}^0, T)}^{-1}=Q_{(\hat{\tau},T)}^{-1}
+O_p\left(\frac{\hat{\tau}-\tau^0}{T^2}\right),\end{equation}
Together with (\ref{tm-1}), we obtain
\begin{eqnarray*}
R_{(0,\hat{\tau})}'Q_{(0,\hat{\tau})}^{-1} R_{(0,\hat{\tau})}
- R_{(0,{\tau}^0)}'Q_{(0,{\tau}^0)}^{-1} R_{(0,\tau^0)}
= R_{(\tau^0,\hat{\tau})}'Q_{(0,{\tau}^0)}^{-1} R_{(\tau^0,\hat{\tau})}
- 2R_{(0, \tau^0)}'Q_{(0,{\tau}^0)}^{-1} R_{(\tau^0,\hat{\tau})}\\
+O_p\left(\frac{\hat{\tau}-\tau^0}{T^2}\right) R_{(0,\hat{\tau})}' R_{(0,\hat{\tau})}
\end{eqnarray*}
\noindent and
\begin{eqnarray*}
 R_{(\tau^0,T)}'Q_{(\tau^0,T)}^{-1} R_{(\tau^0,T)}- R_{(\hat{\tau},T)}'
 Q_{(\hat{\tau},T)}^{-1} R_{(\hat{\tau},T)}
=R_{(\tau^0,\hat{\tau})}'Q_{(\hat{\tau},T)}^{-1} R_{(\tau^0,\hat{\tau})}
- 2R_{(\hat{\tau},T)}'Q_{(\hat{\tau},T)}^{-1} R_{(\tau^0,\hat{\tau})}\\
+O_p\left(\frac{\hat{\tau}-\tau^0}{T^2}\right) R_{(\tau^0,T)}' R_{(\tau^0,T)}.
\end{eqnarray*}
From (\ref{conv_R}), (\ref{conv_Q}), together with Cauchy-Schwarz's
inequality, we have $(\ref{prn2-eq4})=o_p(1)$ for large $T$.
Therefore, by choosing a suitable large $C$, we obtain
$(\ref{prn2-eq5})+(\ref{prn2-eq2})<0$,
which implies that $(\ref{prn2-eq1})<0$ and this gives a contradiction.
Therefore, with large probability, $\hat{\tau}$ can not be in the set
$V_{\eta}(C)$ and hence $P(T|\hat{s}-s^0|\geq C)\leq \epsilon$ when
$T$ is large.
\end{proof}

\section{Proof of Proposition~\ref{prn-ictest}.}\label{append-5}
\noindent This proof can be completed by comparing $\mathcal{IC}(m=0)$ and $\mathcal{IC}(m=1)$
under $\textrm{H}_0$ and $\textrm{H}_1$, respectively. Moreover, note that
$\log(T/\Delta_t)=\log T-\log(\Delta_t)$ and $\log T$ is just a special
case of $\log(T/\Delta_t)$ with $\Delta_t=1$. Hence, in the succeeding proof,
we only prove the case $\phi(T)=\log(T/\Delta_t)$ with $\Delta_t$ a fixed constant.

\begin{proof}
Under $\textrm{H}_0$, $m^0=0$ and $\theta^{(1)}=\theta^{(2)}$. In this case, for $\mathcal{IC}(m=0)$, it
follows from (\ref{ou1}) that
\begin{eqnarray*}\mathcal{IC}(m=0)
= -2\left( \frac{1}{\sigma^2}\int_{0}^{T}S(\hat{\theta}^{(1)},t,X_t)dX_t
- \frac{1}{2\sigma^2}\int_{0}^{T}S(\hat{\theta}^{(1)},t,X_t)^2dt\right) +h(p)
\log (T/\Delta_t)\\
=-2\left( \frac{1}{2\sigma^2}\theta^{(1)\prime}Q_{(0,T)}\theta^{(1)}
+\frac{2}{\sigma}\theta^{(1)\prime}R_{(0,T)}+ \frac{3\sigma^2}{2}R_{(0,T)}\rq{}
Q_{(0,T)}^{-1}R_{(0,T)}\right)+h(p)\log (T/\Delta_t).
\end{eqnarray*}

\noindent Assume further that we get $\hat{\tau}$ from (\ref{mle1})
when $m$ is set to 1 (note that in this case $\theta^{(2)}$ is still equal to $\theta^{(1)}$). Then,
\begin{eqnarray*}\mathcal{IC}(m=1)
= -\frac{1}{\sigma^2}\theta^{(1)\prime}Q_{(0,T)}\theta^{(1)}-\frac{4}
{\sigma}\theta^{(1)\prime}R_{(0,T)}- 3 R_{(0,\hat{\tau})}\rq Q_{(0,\hat{\tau})}^{-1}
R_{(0,\hat{\tau})}\\-3R_{(\hat{\tau},T)}\rq{}Q_{(\hat{\tau},T)}^{-1}R_{(\hat{\tau},T)}
+2h(p)(\log T- \log(\Delta_t))
\end{eqnarray*}

\noindent so that
\begin{eqnarray}\label{prn3-eq1}\mathcal{IC}(m=1)-\mathcal{IC}(m=0)=3R_{(0,T)}\rq{}
Q_{(0,T)}^{-1}R_{(0,T)}-3 R_{(0,\hat{\tau})}\rq Q_{(0,\hat{\tau})}^{-1}R_{(0,\hat{\tau})}
\nonumber\\-3R_{(\hat{\tau},T)}\rq{}Q_{(\hat{\tau},T)}^{-1}R_{(\hat{\tau},T)}
+h(p)(\log T- \log(\Delta_t)).\end{eqnarray}
Moreover, by (\ref{conv_R}) and (\ref{conv_Q}), together with the Cauchy-Schwarz's
inequality,
$$ R_{(0,T)}\rq{}Q_{(0,T)}^{-1}R_{(0,T)}\leq
||\frac{1}{\sqrt{T}}R_{(0,T)}||^2  ||TQ_{(0,T)}^{-1}||=O_p(\log^{2a^*} T),
$$ where $0<a^*<1/2.$
Similarly, $R_{(0,\hat{\tau})}\rq Q_{(0,\hat{\tau})}^{-1}R_{(0,\hat{\tau})}$
and $R_{(\hat{\tau},T)}\rq{}Q_{(\hat{\tau},T)}^{-1}R_{(\hat{\tau},T)}$
are also of order $O_p(\log^{2a^*} T)$. Therefore, for large $T$ and
fixed $\Delta_t$, (\ref{prn3-eq1}) is dominated by $h(p)\log T$, which is positive.
This implies that under $H_0$, the probability of $\mathcal{IC}(m=0)>\mathcal{IC}(m=1)$ tends to 0 as $T$ tends to $\infty$. \\
\ \\
\noindent Under $\textrm{H}_1$, let $\tau^0$ be the exact value of
the change point and $\hat{\tau}$ be the estimator of $\tau^0$ obtained from
(\ref{mle1}). Without loss of generality, we assume that $\hat{\tau}>\tau^0$. So,
\begin{eqnarray*}\mathcal{IC}(m=0)=-\frac{1}{\sigma^2}\theta^{(2)\prime}Q_{(0,T)}\theta^{(2)}
-\frac{1}{\sigma^2}(\theta^{(1)}-\theta^{(2)})'Q_{(0,\tau^0)}Q_{(0,T)}^{-1}Q_{(0,\tau^0)}
(\theta^{(1)}-\theta^{(2)})\\
-3R_{(0,T)}'Q_{(0,T)}^{-1}R_{(0,T)}-\frac{4}{\sigma}\theta^{(2)\prime}R_{(0,T)}
-\frac{4}{\sigma}(\theta^{(1)}-\theta^{(2)})'Q_{(0,\tau^0)}Q_{(0,T)}^{-1}R_{(0,T)}
\\-\frac{4}{\sigma^2}\theta^{(2)\prime}Q_{(0,\tau^0)}(\theta^{(1)}-\theta^{(2)})
+h(p)(\log T- \log(\Delta_t)),\end{eqnarray*}
and
\begin{eqnarray*}\mathcal{IC}(m=1)=-\frac{1}{\sigma^2}
\theta^{(2)\prime}Q_{(0,T)}\theta^{(2)}-\frac{1}{\sigma^2}
(\theta^{(1)}-\theta^{(2)})'Q_{(0,\tau^0)}Q_{(0,\hat{\tau})}^{-1}
Q_{(0,\tau^0)}(\theta^{(1)}-\theta^{(2)})\\
-3R_{(0,\hat{\tau})}'Q_{(0,\hat{\tau})}^{-1}R_{(0,\hat{\tau})}
-3R_{(\hat{\tau},T)}'Q_{(\hat{\tau},T)}^{-1}R_{(\hat{\tau},T)}
-\frac{4}{\sigma}(\theta^{(1)}-\theta^{(2)})'Q_{(0,\tau^0)}
Q_{(0,\hat{\tau})}^{-1}R_{(0,\hat{\tau})}\\-\frac{4}
{\sigma}\theta^{(2)\prime}R_{(0,T)}-\frac{4}{\sigma^2}
\theta^{(2)\prime}Q_{(0,\tau^0)}(\theta^{(1)}
-\theta^{(2)})+2h(p)(\log T- \log(\Delta_t)).\end{eqnarray*}

\noindent Therefore,
\begin{eqnarray*}\mathcal{IC}(m=1)-\mathcal{IC}(m=0)=
-\frac{1}{\sigma^2}(\theta^{(1)}-\theta^{(2)})'
Q_{(0,\tau^0)}(Q_{(0,\hat{\tau})}^{-1}
- Q_{(0,T)}^{-1})Q_{(0,\tau^0)}(\theta^{(1)}-\theta^{(2)})\\
-3R_{(0,\tau^0)}'Q_{(0,\tau^0)}^{-1}R_{(0,\tau^0)}
-3R_{(\tau^0,T)}'Q_{(\tau^0,T)}^{-1}R_{(\tau^0,T)}
+3R_{(0,T)}'Q_{(0,T)}^{-1}R_{(0,T)}-\frac{4}{\sigma}
(\theta^{(1)}-\theta^{(2)})'Q_{(0,\tau^0)}\\\times Q_{(0,\hat{\tau})}^{-1}
R_{(0,\hat{\tau})}
-\frac{4}{\sigma}(\theta^{(1)}-\theta^{(2)})'Q_{(0,\tau^0)}
Q_{(0,T)}^{-1}R_{(0,T)}+h(p)(\log T- \log(\Delta_t)).\end{eqnarray*}
\noindent Multiplying both sides of the above identity by $\frac{1}{T}$,
and using (\ref{conv_R}), (\ref{conv_Q}) and the Cauchy-Schwarz's inequality,
we have
 \begin{eqnarray}\label{ci-eq01}\frac{1}{T}(\mathcal{IC}(m=1)-\mathcal{IC}(m=0))=
\frac{1}{\sigma^2}(\theta^{(1)}-\theta^{(2)})'\frac{1}{T}
Q_{(\tau^0,\hat{\tau})}TQ_{(0,\hat{\tau})}^{-1}\frac{1}{T}Q_{(0,\tau^0)}
(\theta^{(1)}-\theta^{(2)})\nonumber\\
-\frac{(1-s^0)}{\sigma^2}(\theta^{(1)}-\theta^{(2)})'\frac{1}{(1-s^0)T}
Q_{(\tau^0,T)} Q_{(0,T)}^{-1}Q_{(0,\tau^0)}
(\theta^{(1)}-\theta^{(2)})+o_p(1).\quad
\label{prn3-eq2}\end{eqnarray}
Note that the second term in (\ref{ci-eq01}) is equal to
$$-\frac{(1-s^0)}{\sigma^2}(\theta^{(1)}-\theta^{(2)})'\frac{1}{2(1-s^0)T}
(Q_{(\tau^0,T)} Q_{(0,T)}^{-1}Q_{(0,\tau^0)}+Q_{(0,\tau^0)} Q_{(0,T)}^{-1}Q_{(\tau^0,T)} )
(\theta^{(1)}-\theta^{(2)}).$$
\noindent Using the similar argument as in the proof of Proposition~\ref{prn-mle1},
we have that under Assumption \ref{asm4}, the second term is less than
$-C_4 ||\theta^{(1)}-\theta^{(2)}||^2$ for some $C_4>0$. It also follows
from Proposition~\ref{prn-mle2} that $||\frac{1}{T}Q_{(\tau^0,\hat{\tau})}||
=\frac{\hat{\tau}-\tau^0}{T}||\frac{1}{\hat{\tau}-\tau^0}Q_{(\tau^0,\hat{\tau})}||
=o_p(1)$ for large $T$. Thus, by (\ref{conv_Q}), together with the Cauchy-Schwarz's
inequality we have
$$\frac{1}{\sigma^2}(\theta^{(1)}-\theta^{(2)})'\frac{1}{T}Q_{(\tau^0,\hat{\tau})}T
Q_{(0,\hat{\tau})}^{-1}\frac{1}{T}Q_{(0,\tau^0)}(\theta^{(1)}-\theta^{(2)})=o_p(1).$$
This tells us that (\ref{prn3-eq2}) is dominated by the first term for large
$T$ and it is negative. Therefore $\mathcal{IC}(m=0)>\mathcal{IC}(m=1)$ with probability 1.
\end{proof}

\ \\
\ \\
\ \\
\noindent {\bf Acknowledgements:} \textit{Fuqi Chen and Rogemar Mamon wish to thank the hospitality
and financial support of the Fields Institute for Research in Mathematical
Sciences, Toronto, Ontario, Canada, where this research was
conceived and partially conducted. Likewise, all authors gratefully acknowledge the generous support
of Charmaine Dean and Matt Davison on this research collaboration. Helpful comments
from Charmaine Dean and Reg Kuperger are very much appreciated.} \\
\ \\
\ \\
\ \\

\end{document}